\documentclass[11pt]{amsart}
\usepackage{amssymb}
\usepackage{amscd}
\usepackage{amsmath}
\usepackage{amsmath}
\usepackage[all]{xy}
\usepackage{mathrsfs}

\usepackage{color}

\usepackage{bm}

\theoremstyle{plain}
\newtheorem{theorem}{Theorem}[section]
\newtheorem{lemma}[theorem]{Lemma}
\newtheorem{corollary}[theorem]{Corollary}
\newtheorem{proposition}[theorem]{Proposition}
\newtheorem{conjecture*}{Conjecture}
\newtheorem{theorem*}{Theorem}
\newtheorem{corollary*}{Corollary}

\newtheorem{conjecture}[theorem]{Conjecture}
\theoremstyle{definition}

\newtheorem{hypothesis}[theorem]{Hypothesis}
\theoremstyle{definition}
\newtheorem{definition}[theorem]{Definition}
\newtheorem{remark}[theorem]{Remark}

\newtheorem*{acknowledgments}{Acknowledgements}

\font\russ=wncyr10  1
\def\sha{\hbox{\russ\char88}}

\DeclareMathOperator{\Gal}{Gal}

\DeclareMathOperator{\Hom}{Hom}

\DeclareMathOperator{\Spec}{Spec}
\DeclareMathOperator{\N}{N}

\DeclareMathOperator{\im}{im}

\newcommand{\CC}{\mathbb{C}}

\newcommand{\DD}{\mathbb{D}}

\newcommand{\GG}{\mathbb{G}}
\newcommand{\HH}{\mathbb{H}}

\newcommand{\PP}{\mathbb{P}}
\newcommand{\QQ}{\mathbb{Q}}

\newcommand{\RR}{\mathbb{R}}

\newcommand{\ZZ}{\mathbb{Z}}

\newcommand{\cD}{\mathcal{D}}

\newcommand{\cL}{\mathcal{L}}

\newcommand{\cH}{\mathcal{H}}
\newcommand{\sH}{\mathscr{H}}
\newcommand{\cI}{\mathcal{I}}

\newcommand{\cN}{\mathcal{N}}
\newcommand{\cO}{\mathcal{O}}
\newcommand{\cQ}{\mathcal{Q}}

\newcommand{\frp}{\mathfrak{p}}

\newcommand{\fz}{\mathfrak{z}}

\newcommand{\id}{\mathrm{id}}

\newcommand{\ord}{\mathrm{ord}}

\newcommand{\drig}{\mathbb{D}_{\rm rig}^\dagger}

\newcommand{\catname}[1]{\textnormal{{\textsf{#1}}}}

\newcommand{\DR}{\catname{R}}
\newcommand{\DL}{\catname{L}}

\newcommand{\rgamma}{\DR\Gamma}
\newcommand{\rhom}{\DR\Hom}
\newcommand{\lotimes}{\otimes^{\DL}}

\makeatletter
 
  \@addtoreset{equation}{subsection}
\makeatother

\begin{document}

\title[]{On derivatives of Kato's Euler system\\
 for elliptic curves}

\author{David Burns, Masato Kurihara and Takamichi Sano}

\begin{abstract}
In this paper we study a new conjecture concerning Kato's Euler system of zeta elements 
for elliptic curves $E$ over $\QQ$. This conjecture, which we refer to as the  `Generalized Perrin-Riou Conjecture', predicts a precise congruence relation between a `Darmon-type derivative' of the zeta element of $E$ over an arbitrary real abelian field 
and the critical value of an appropriate higher derivative of the $L$-function of $E$ over $\QQ$. We prove that the conjecture specializes in the relevant case of analytic rank one to recover Perrin-Riou's conjecture 
on the logarithm of Kato's zeta element. Under mild hypotheses we also prove that the  `order of vanishing' part of the conjecture is valid in arbitrary rank. An Iwasawa-theoretic analysis of our approach leads to the formulation and proof of a natural higher rank generalization of Rubin's formula concerning derivatives of $p$-adic $L$-functions. In addition, we establish a concrete and apparently new connection between the $p$-part of the classical Birch and Swinnerton-Dyer Formula and the Iwasawa Main Conjecture in arbitrary rank and for arbitrary reduction at $p$. In a forthcoming paper we will show that the Generalized Perrin-Riou Conjecture implies (in arbitrary rank) the conjecture of Mazur and Tate concerning congruences for modular elements and, by using this approach, we are able to give a proof, under certain mild and natural hypotheses, that the Mazur-Tate Conjecture is valid in analytic rank one. 
 \end{abstract}

\address{King's College London,
Department of Mathematics,
London WC2R 2LS,
U.K.}
\email{david.burns@kcl.ac.uk}

\address{Keio University,
Department of Mathematics,
3-14-1 Hiyoshi\\Kohoku-ku\\Yokohama\\223-8522,
Japan}
\email{kurihara@math.keio.ac.jp}

\address{Osaka City University,
Department of Mathematics,
3-3-138 Sugimoto\\Sumiyoshi-ku\\Osaka\\558-8585,
Japan}
\email{sano@sci.osaka-cu.ac.jp}

\maketitle

\tableofcontents

\section{Introduction} \label{Intro}

\subsection{Background}

A central problem in modern
number theory is to
understand the arithmetic meaning of the values of zeta
and $L$-functions.

The Birch and Swinnerton-Dyer Conjecture and main conjecture in Iwasawa theory are important instances of this problem, being respectively related to the Hasse-Weil $L$-function of an elliptic curve and to the $p$-adic $L$-function of an appropriate motive.

For an elliptic curve $E$ defined over $\QQ$, significant progress on the problem was made by Kato in \cite{katoasterisque} who used Beilinson elements in the $K$-theory of modular curves to define canonical `zeta elements' in \'{e}tale (Galois) cohomology groups that could be explicitly related to the values of Hasse-Weil $L$-functions.

To be a little more precise we fix an odd prime $p$, a finite abelian extension $F$ of $\QQ$, a finite set of places $S$ of $\QQ$ that contains the  archimedean place, $p$, all primes that ramify in $F$ and all primes at which $E$ has bad reduction. We write $\mathcal{O}_{F,S}$ for the subring of $F$ comprising elements that are integral at all non-archimedean places whose residue characteristic does not belong to $S$ and $T_p(E)$ for the $p$-adic Tate module of $E$.

Then the zeta element $z_{F}$ constructed by Kato belongs to the \'{e}tale cohomology group $H^{1}({\mathcal O}_{F,S}, T_{p}(E))$ and is explicitly related via the dual exponential map to the value at one of the Hasse-Weil $L$-function of $E$ (for more precise statements see \S 2).

As $F$ varies over finite subfields of the cyclotomic $\ZZ_{p}$-extension of $\QQ$, these elements $z_{F}$ form a projective system that can be used to recover the $p$-adic $L$-function of $E$.

In addition, as $F$ varies more generally, the elements $z_{F}$ form an Euler system and so can be used to bound the
$p$-adic Selmer group of $E$.

In this way zeta elements have led to partial results on both the main conjecture and Birch and Swinnerton-Dyer Conjecture for $E$.

For this reason, such elements have subsequently been much studied in the literature and have led to numerous important results.

Our main purpose in these articles is to investigate a conjectural property of Kato's elements that it seems has not been observed previously and to demonstrate that this property, whenever valid, has significant applications.

The conjecture itself predicts a precise link between a `Darmon-type' derivative of $z_F$ for any given $F$ and the value at the critical point of an appropriate higher derivative of the $L$-function of $E$ over $\QQ$.

This conjectural link constitutes a simultaneous refinement of well-known conjectures of Perrin-Riou \cite{PR} and of Mazur and Tate \cite{MT} and will be described in more detail in the next section.


Although we shall not pursue it here, it seems reasonable to expect that the general approach we develop can also be applied to elliptic curves with complex multiplication, with the role of Kato's zeta elements being replaced by elliptic units twisted by a Hecke character. 

We also expect that it should be possible to extend our approach to the setting of abelian varieties and to modular forms and their families, and we hope to return to these questions in a subsequent article.

\subsection{Conjectures and results at finite level} We shall now give an overview of the central conjecture that we formulate and the evidence for it that we have so far obtained.

\subsubsection{}At the outset we fix a finite {\it real} abelian extension $F$ of $\QQ$ and set $G:=\Gal(F/\QQ)$.

Then, following a general idea introduced by Darmon in \cite{DH}, the key object of our study will be the element
$${\mathcal N}_{F/\QQ}(z_{F}):=\sum_{\sigma \in G}
\sigma(z_{F}) \otimes \sigma^{-1}$$
of $H^1({\mathcal O}_{F,S}, T_{p}(E)) \otimes_{\ZZ_{p}}
\ZZ_{p}[G]$.


We write $r$ for the rank of $E(\QQ)$ and assume that $r > 0$, that $E(\QQ)$ has no element of order $p$ and that the $p$-part of the Tate-Shafarevich group of $E/\QQ$ is finite.

Then, under these hypotheses, in Definition \ref{def eta} we shall use the leading term at $s=1$ of $L(E,s)$ to (unconditionally) define a canonical `Birch and Swinnerton-Dyer element' $\eta^{{\rm BSD}}$ in the dimension one vector space over $\CC_p$ that is spanned by $\bigwedge_{\ZZ_{p}}^r
H^1(\ZZ_S, T_{p}(E))$.


With $I$ denoting the augmentation ideal of $\ZZ_{p}[G]$, we shall also define (in \S\ref{sec boc}) a canonical
`Bockstein regulator map'
$${\rm Boc}_F: {\bigwedge}_{\ZZ_{p}}^r H^1(\ZZ_S, T_{p}(E)) \longrightarrow
H^1(\ZZ_S, T_{p}(E)) \otimes_{\ZZ_p} I^{r-1}/I^r .$$

Finally we note the $\ZZ_p$-module $H^1({\mathcal O}_{F,S}, T_{p}(E))$ is free and so $H^1({\mathcal O}_{F,S}, T_{p}(E)) \otimes_{\ZZ_p} I^{r-1}$ identifies with a submodule of $H^1({\mathcal O}_{F,S}, T_{p}(E)) \otimes_{\ZZ_p} \ZZ_p[G]$.

Then, in terms of this notation, the central conjecture of this article can be stated as follows. 

\begin{conjecture}[The Generalized Perrin-Riou Conjecture] \label{conj01}\
\begin{itemize}
\item[(i)] (`Order of vanishing') ${\mathcal N}_{F/\QQ}(z_{F})$ belongs to $H^1({\mathcal O}_{F,S}, T_{p}(E)) \otimes_{\ZZ_p} I^{r-1}$.
\item[(ii)] (`Integrality') If $r > 1$, then $\eta^{\rm BSD}$ belongs to $\bigwedge_{\ZZ_{p}}^rH^1(\ZZ_S, T_{p}(E))$.
\item[(iii)] (`Leading term formula') The image of ${\mathcal N}_{F/\QQ}(z_{F})$ in $H^1({\mathcal O}_{F,S}, T_{p}(E)) \otimes_{\ZZ_p} I^{r-1}/I^r$ is equal to
 ${\rm Boc}_F(\eta^{{\rm BSD}})$.
\end{itemize}
\end{conjecture}

\begin{remark} A precise statement of Conjecture \ref{conj01} will be given as Conjecture \ref{mrs}. 
For the moment, we note a key advantage of its formulation is that it uses a construction of regulators that works in the same way for all reduction types. 
 A further crucial advantage is that, in the case $r=1$, the conjecture takes a particularly simple form and can be proved under various natural hypotheses.
\end{remark}

In the rest of this section we outline the evidence that we have obtained for the above conjecture and also explain why it constitutes a simultaneous refinement and generalization of conjectures of Perrin-Riou and of Mazur and Tate.

\subsubsection{}We observe first that the containment predicted by Conjecture \ref{conj01}(i) can be studied by using the equivariant theory of Euler systems that was recently described by Sakamoto and the first and third authors in \cite{bss2}.

In particular, by using this approach we are able to prove that Conjecture \ref{conj01}(i) is valid under certain mild hypotheses.

For example, the following concrete result will follow directly from stronger results that we prove in \S\ref{zefi sec}. This result is a natural analogue for zeta elements of the main result of Darmon \cite[Th. 2.4]{DH} concerning Heegner points. 



\begin{theorem}\label{vanishing evidence} The containment of Conjecture \ref{conj01}(i) is valid if all of the following conditions are satisfied.
\begin{itemize}
\item[(a)] $p>3$;
\item[(b)] the $p$-primary parts of $\sha(E/F)$ and $\sha(E/\QQ)$ are finite;
\item[(c)] the image of the representation $G_\QQ \to {\rm Aut}(T_p(E)) \simeq {\rm GL}_2(\ZZ_p)$ contains ${\rm SL}_2(\ZZ_p)$;
\item[(d)] for every prime number $\ell$ in $S$ the group $E(\QQ_\ell)$ contains no element of order $p$.
\end{itemize}
\end{theorem}
%

Concerning Conjecture \ref{conj01}(ii), we can show in all cases that the predicted containment is valid  whenever the $p$-part of the Birch and Swinnerton-Dyer Formula for $E$ over $\QQ$, or `${\rm BSD}_p(E)$' as we shall abbreviate it in the sequel, is valid.
(In fact, a stronger version of this result will be proved in Proposition \ref{prop eta}).

Finally, to discuss the prediction of Conjecture \ref{conj01}(iii) we shall initially specialize to the case that the analytic rank $\ord_{s=1}L(E,s)$ of $E$ is equal to one.

In this case, well-known results of Gross and Zagier and of Kolyvagin (amongst others) imply that $r=1$ and so parts (i) and (ii) of Conjecture \ref{conj01} are valid trivially.

It is also straightforward to check in this case that the equality in Conjecture \ref{conj01}(iii) is valid for any, and therefore every, choice of  field $F$ if and only if one has $z_{\QQ}=\eta^{\rm BSD}$.

By analysing the latter equality, we shall thereby obtain the explicit interpretation of this case of Conjecture \ref{conj01} that is given in the next result. (A proof of this result will be explained in Remark \ref{PR equiv}(ii)).

In the sequel we write $L_{S}(E,s)$ for the $S$-truncated Hasse-Weil $L$-function of $E$.

\begin{theorem}\label{theorem01} If $E$ has analytic rank one, then Conjecture \ref{conj01} is valid for any, and therefore every, field $F$ if and only if one has $z_\QQ \in H^1_f(\QQ,T_p(E))$ and
$$\log_{\omega}(z_\QQ)=\frac{L_{S}'(E,1)}{\Omega^+ \cdot
\langle x,x \rangle_\infty} \log_\omega(x)^2.$$
Here $\log_{\omega}: H_f^1(\QQ, T_{p}(E)) \rightarrow
\QQ_{p}$ is the formal logarithm associated to the (fixed) 
N\'{e}ron differential $\omega$, $L_{S}'(E,1)$ denotes the value at $s=1$ of the first derivative of $L_{S}(E,s)$, $\Omega^+$ is the real N\'{e}ron period, $x$ is a generator of
$E(\QQ)$ modulo torsion and $\langle -,- \rangle_\infty$ is
the N\'eron-Tate height pairing.\end{theorem}

The displayed equality in Theorem \ref{theorem01} is equivalent to the central conjecture formulated by Perrin-Riou in \cite[\S3.3]{PR}. For this reason, Theorem \ref{theorem01} allows us to regard Conjecture \ref{conj01} as a natural extension of Perrin-Riou's conjecture to elliptic curves of arbitrary rank.

In addition, Perrin-Riou's conjecture has recently been verified in several important cases and Theorem \ref{theorem01} implies all such results provide evidence for Conjecture \ref{conj01} in the case of analytic rank one.

For example, in \cite[Th. 2.4(iv)]{buyuk perrin} B\"uy\"ukboduk has proved Perrin-Riou's conjecture in the case that $E$ has good supersingular reduction at $p$ and square-free conductor and in \cite[Th. A]{venerucci} Venerucci has proved (a weak version of) the conjecture in the split multiplicative case. In addition, analogous results in the good ordinary case are obtained by  B\"uy\"ukboduk, Pollack, and Sasaki in \cite{BPS} (Bertolini and Darmon have also announced a proof in \cite{BD}). 

To discuss Conjecture \ref{conj01} in the case of arbitrary rank, we assume that ${\rm BSD}_p(E)$, and hence also the containment of Conjecture \ref{conj01}(ii), is valid.

In this case we will show in the second part  of this article \cite{bks5} that, under mild additional hypotheses, the equality of Conjecture \ref{conj01}(iii) implies a refined, and in certain key respects better-behaved, version of the celebrated conjecture formulated by Mazur and Tate in \cite{MT} concerning congruence relations between modular symbols and the discriminants of algebraic height pairings that are defined in terms of the geometrical theory of bi-extensions.

In particular, since there are by now many curves $E$ over $\QQ$ of analytic rank one and primes $p$ for which ${\rm BSD}_p(E)$ is known to be valid (by recent work of Jetchev, Skinner and Wan \cite{JSW} and Castella \cite{castella}), we can thereby deduce the validity of the Mazur-Tate Conjecture in this case from Theorems \ref{vanishing evidence} and \ref{theorem01} and the results on Perrin-Riou's conjecture that are recalled above.

In this way we shall obtain the first verifications of the Mazur-Tate Conjecture for any curves $E$ for which $L(E,1)$ vanishes.
%

This deduction gives a clear indication of the interest of, and new insight that can be obtained from, the general approach that underlies the formulation of Conjecture \ref{conj01}.

In this regard, we also observe that one of the key motivations behind the development of this approach was an attempt to understand if there was a natural analogue for elliptic curves of the conjecture formulated in \cite[Conj. 5.4]{bks1} in the setting of the multiplicative group.

We finally recall that the latter conjecture was itself formulated as a natural strengthening of the `refined class number formula for $\mathbb{G}_m$' that was previously conjectured by the third author \cite{sano}, and (independently) by Mazur and Rubin \cite{MRGm}.

In the next section we shall focus on the conjecture in the important special case that $F$ is contained in the cyclotomic
$\ZZ_{p}$-extension of $\QQ$.

\subsection{Iwasawa-theoretic considerations} We shall also show that the simultaneous study of Conjecture \ref{conj01} for the family of intermediate fields $F$ of the cyclotomic $\ZZ_{p}$-extension $\QQ_{\infty}$ of $\QQ$ gives a more rigid framework that sheds light on a range of  important problems.

\subsubsection{}To explain this, for each natural number $n$ we write $\QQ_n$ for the unique subfield of $\QQ_\infty$ of degree $p^n$  over $\QQ$.

We know the validity of Conjecture \ref{conj01}(i) with $F = \QQ_n$ (see Proposition \ref{kato vanish}), and we write $\kappa_{n}$ for the image of $\cN_{\QQ_n/\QQ}(z_{\QQ_n})$ under the natural projection
\[ H^1(\mathcal{O}_{\QQ_n,S},T_p(E))\otimes_{\ZZ_p} I_n^{r-1} \to H^1(\mathcal{O}_{\QQ_n,S},T_p(E))\otimes_{\ZZ_p} I_n^{r-1}/I_n^{r},\]
where $I_n$ denotes the augmentation ideal of $\ZZ_p[\Gal(\QQ_n/\QQ)]$.

Then we can show the element $\kappa_n$ belongs to the subgroup $H^1(\ZZ_S,T_p(E))\otimes_{\ZZ_p} I_n^{r-1}/I_n^{r}$ of $H^1(\mathcal{O}_{\QQ_n,S},T_p(E))\otimes_{\ZZ_p} I_n^{r-1}/I_n^{r}$ and, moreover, that as $n$ varies the elements $\kappa_n$ are compatible with the natural projection maps
\[ H^1(\ZZ_S,T_p(E))\otimes_{\ZZ_p} I_n^{r-1}/I_n^{r} \to H^1(\ZZ_S,T_p(E))\otimes_{\ZZ_p} I_{n-1}^{r-1}/I_{n-1}^{r}.\]

Hence, writing $I$ for the augmentation ideal of $\ZZ_p[[\Gal(\QQ_\infty/\QQ)]]$, one obtains an element of $H^1(\ZZ_S,T_p(E)) \otimes_{\ZZ_p}I^{r-1}/I^r$ by setting
$$\kappa_\infty:=\varprojlim_n \kappa_{n} \in \varprojlim_n H^1(\ZZ_S,T_p(E))\otimes_{\ZZ_p} I_n^{r-1}/I_n^r \simeq H^1(\ZZ_S,T_p(E)) \otimes_{\ZZ_p}I^{r-1}/I^r$$
(cf. Definition \ref{def iw}).

In addition, the family of maps $({\rm Boc}_{\QQ_n})_n$ induces a canonical homomorphism
\[  \CC_p\cdot {\bigwedge}_{\ZZ_{p}}^r H^1(\ZZ_S, T_{p}(E)) \rightarrow
\CC_p\cdot H^1(\ZZ_S, T_{p}(E)) \otimes_{\ZZ_p} I^{r-1}/I^r\]
and the fact that the $\ZZ_p$-module $I^{r-1}/I^r$ is torsion-free implies that the natural map
\[ H^1(\ZZ_S, T_{p}(E)) \otimes_{\ZZ_p} I^{r-1}/I^r \to \CC_p\cdot H^1(\ZZ_S, T_{p}(E)) \otimes_{\ZZ_p} I^{r-1}/I^r\]
is injective. In particular, this allows one to formulate Conjecture \ref{conj01}(iii) for the family of elements $\cN_{\QQ_n/\QQ}(z_{\QQ_n})$ without having to assume the validity of Conjecture \ref{conj01}(ii).

We shall show (in Proposition  \ref{prop exp}) that this version of Conjecture \ref{conj01}(iii) is
equivalent to the following prediction.

In the sequel we write $L_{S}^{(r)}(E,1)$ for the coefficient of $(s-1)^r$ in the Taylor expansion at $s=1$ of $L_S(E,s)$.

\begin{conjecture}[Conjecture \ref{mrs2}]\label{conj02} If $r$ is also equal to the analytic rank $\ord_{s=1}L(E,s)$ of $E$, then one has
$$\kappa_\infty=\frac{L_{S}^{(r)}(E,1)}{\Omega^+\cdot R_\infty}\cdot R_\omega^{\rm Boc},$$
where $\Omega^+$ is the real N\'{e}ron period,
$R_\infty$ is  the N\'eron-Tate regulator and $R_\omega^{\rm Boc}$ is the 
`Bockstein regulator' in
$H^1(\ZZ_S,T_p(E)) \otimes_{\ZZ_p}
I^{r-1}/I^r$ that is introduced in Definition \ref{def alg}.
\end{conjecture}

\begin{remark} If $r$ is equal to $\ord_{s=1}L(E,s)$, then the $r$-th derivative of $L_S(E,s)$ is holomorphic at $s=1$ and its (non-zero) value at $s=1$ is equal to $r!\cdot L_{S}^{(r)}(E,1)$.\end{remark}

\begin{remark}\label{nek rem} We will show that the Bockstein regulator that occurs in Conjecture \ref{conj02} has the following properties.\

\noindent{}(i) If $r=1$, then 
$$R_\omega^{\rm Boc}=\log_\omega(x)\cdot x$$
for any element $x$ of $E(\QQ)$ that generates $E(\QQ)$ modulo torsion (cf. Remark \ref{remark boc1}).

\noindent{}(ii) Suppose that $E$ does not have additive reduction at $p$ and write $\langle -,-\rangle_p$ for the classical $p$-adic height pairing. Then for any element $x$ of $E(\QQ)$ one has
$$\langle x, R_\omega^{\rm Boc}\rangle_p =\log_\omega(x)\cdot R_p,$$
where $R_p$ denotes the $p$-adic regulator (cf. Theorems \ref{reg prop} and \ref{reg prop 2}).
\end{remark}

If $r=1$, then $\kappa_\infty$ simply coincides with $z_\QQ$ and so Remark 
\ref{nek rem}(i) implies that Conjecture \ref{conj02} is valid if and only if one has
$$z_\QQ = \frac{L_{S}'(E,1)}{\Omega^+\cdot R_\infty}
\log_\omega(x)\cdot x
$$
for any element $x$ of $E(\QQ)$ that generates $E(\QQ)$ modulo torsion. This equality is equivalent to Perrin-Riou's conjecture. 

In addition, whilst Remark \ref{nek rem}(ii) implies that the Bockstein regulator $R_\omega^{\rm Boc}$ is a variant of the classical $p$-adic regulator, a key role will be played in our approach by the fact that $R_\omega^{\rm Boc}$ can be defined even in the case that $E$ has additive reduction at $p$ (in which case a construction of the $p$-adic regulator is still unknown).

\subsubsection{}To interpret Conjecture \ref{conj02} in terms of $p$-adic $L$-functions, we must first prove a `Generalized Rubin Formula' for the element $\kappa_{\infty}$.

To discuss this result, and some of its consequences, we assume until further notice that $E$ does not have additive reduction at $p$.

 If $E$ has good reduction at $p$, then we write $\alpha$ for
an allowable root of the Hecke polynomial $X^2-a_pX +p$. We set $\beta:=p/\alpha$. 

 If $E$ has non-split multiplicative
reduction at $p$, then we set $\alpha:=-1 $ and $\beta:=-p$. 

We also write $\mathcal{L}_{S,p}^{(r)}$ for the `$r$-th derivative' of the $S$-truncated $p$-adic $L$-function $\mathcal{L}_{S,p}$ of $E$ (for a precise definition of this term see \S\ref{gen RF section}).

\begin{theorem}[{The Generalized Rubin Formula, Theorem \ref{main}}]\label{theorem1i}\
\begin{itemize}
\item[(i)] If $E$ has good or non-split multiplicative reduction at $p$, then for every element $x$ of $E(\QQ)$ one has
$$ \langle x, \kappa_\infty\rangle_p = \left( 1-\frac 1\alpha \right)^{-1}\left(1-\frac 1\beta\right)\log_\omega(x)\cdot  \cL_{S,p}^{(r)} .$$
\item[(ii)] If $E$ has split multiplicative reduction at $p$, then for every element $x$ of $E(\QQ)$ one has
$$ \langle x, \kappa_\infty\rangle_p \cdot \mathcal{L} = \left(1-\frac 1p\right)\log_\omega(x)\cdot  \cL_{S,p}^{(r+1)} ,$$
where $\mathcal{L}$ denotes the `$\mathcal{L}$-invariant' of $E$ (see Remark \ref{remL}).
\end{itemize}
\end{theorem}

\begin{remark} If $r =1$, then one has $\kappa_\infty = z_\QQ$ and Theorem \ref{theorem1i}(i) recovers the formula that is proved by Rubin in \cite[Th. 1(ii)]{rubin} in the case that $E$ has good ordinary reduction
at $p$.
\end{remark}

We shall then show that this result has the following consequences.

\begin{corollary}[{Corollary \ref{cor beilinson}}]\label{cor1}\
The Generalized Perrin-Riou Conjecture of Conjecture \ref{conj02} implies the following `$p$-adic Beilinson Formula': one has
$$
 \left( 1-\frac 1\alpha \right)^{-1}\left(1-\frac 1\beta\right) \cL_{S,p}^{(r)} =\frac{L_{S}^{(r)}(E,1)}{\Omega^+\cdot R_\infty}R_p
$$
if $E$ has good or non-split multiplicative reduction at $p$, and
$$
    \cL_{S,p}^{(r+1)} =\mathcal{L}\cdot  \frac{L^{(r)}_{S\setminus\{p\}}(E,1)}{\Omega^+\cdot R_\infty}R_p
 $$
if $E$ has split multiplicative reduction at $p$.
\end{corollary}

In the next result we refer to the Iwasawa Main Conjecture
 for $E$ and $\QQ_\infty/\QQ$ that is formulated in Conjecture \ref{IMC}.




\begin{corollary}[{Corollary \ref{cor th1}}] \label{cor iw1i}
Assume $\sha(E/\QQ)$ is finite. Then the Iwasawa Main Conjecture for $E$ and $\QQ_\infty/\QQ$ implies the validity up to multiplication by elements of $\ZZ_p^\times$ of the $p$-adic Birch and Swinnerton-Dyer Formula for $E$.
\end{corollary}

\begin{remark} If $E$ has good reduction at $p$ and its $p$-adic height pairing is non-degenerate, then the result of Corollary \ref{cor iw1i} was first proved by Perrin-Riou in \cite[Prop. 3.4.6]{PR}.\end{remark}

\subsubsection{}Our general approach also allows a systematic analysis of descent arguments in Iwasawa theory without having to make any restrictive hypotheses on the reduction type of $E$ at $p$.

In particular, in this way we are able to prove the following result even in the case that $E$ has additive reduction at $p$.

\begin{theorem}[{Theorem \ref{th2}}]\label{iw2}
Assume all of the following hypotheses:
\begin{itemize}
\item $\sha(E/\QQ)$ is finite;
\item the analytic rank of $E$ is equal to the rank $r$ of $E(\QQ)$;
\item the Iwasawa Main Conjecture of Conjecture \ref{IMC} is valid;
\item the Generalized Perrin-Riou Conjecture of Conjecture \ref{conj02} is valid;
\item the Bockstein regulator $R_\omega^{\rm Boc}$ does not vanish.
\end{itemize}

Then there exists an element $u$ of $\ZZ^\times_p$ such that
$$\frac{L^{(r)}(E,1)}{\Omega^+\cdot R_\infty} = u\cdot \frac{\#\sha(E/\QQ)\cdot {\rm Tam}(E)}{\# E(\QQ)_{\rm tors}^2},$$
where ${\rm Tam}(E)$ denotes the product of the Tamagawa factors of $E/\QQ$.

In particular, the conjecture ${\rm BSD}_p(E)$ is valid.
\end{theorem}

As far as we are aware, this is the first result in which a concrete connection between the $p$-part of the (classical) Birch and Swinnerton-Dyer Formula and the Iwasawa Main Conjecture has been established in the context of either arbitrary analytic rank or arbitrary reduction at $p$.

In this regard we also note that Theorem \ref{iw2} is a natural analogue of the main result of the present authors in \cite{bks2} in which a strategy for proving the equivariant Tamagawa Number Conjecture for $\GG_m$ is established (see Remark \ref{rem3}).

\subsection{General notation} For the reader's convenience we collect together some of the general notation that will be used throughout this article.

At the outset we fix an {\it odd} prime number $p$. The symbol $\ell$ will also usually denote a prime number.

For a field $K$, the absolute Galois group of $K$ is denoted by $G_K$.

We fix an algebraic closure $\overline \QQ$ of $\QQ$. We also fix an algebraic closure $\overline \QQ_p$ of $\QQ_p$ and fix an embedding $\overline \QQ \hookrightarrow \overline \QQ_p$.

For a positive integer $m$, we denote by $\mu_m \subset \overline \QQ$ the group of $m$-th roots of unity.

For an abelian group $X$, we use the following notations:
\begin{itemize}
\item $X_{\rm tors}$: the subgroup of torsion elements;
\item $X_{\rm tf}:=X/X_{\rm tors}$: the torsion-free quotient;
\item ${\rm rank}(X):={\rm rank}_\ZZ(X_{\rm tf})$;
\item $X[p]$: the subgroup of elements annihilated by $p$;
\item $X[p^\infty]$: the subgroup of elements annihilated by a power of $p$.
\end{itemize}

If $X$ is endowed with an action of complex conjugation, we denote by $X^+$ the subgroup of $X$ fixed by the action.

If $X$ is an $R$-module (with $R$ a commutative ring), we set
$$X^\ast:=\Hom_R(X,R).$$
Note that this notation has ambiguity, since $X$ may be regarded as an $R'$-module with another ring $R'$ and $X^\ast$ can mean $\Hom_{R'}(X,R')$. However, this ambiguity would not make any danger of confusion since the meaning is usually clear from the context.

For an element $x \in X$, we denote by $\langle x \rangle_R$ the submodule generated by $x$ over $R$. We abbreviate it to $\langle x \rangle$ when $R$ is clear from the context. 

Suppose that $X$ is a free $R$-module with basis $\{x_1,\ldots,x_r\}$. We denote by
$$x_i^\ast: X \to R$$
the dual of $x_i$, i.e., the map defined by
$$x_j \mapsto \begin{cases}
1 &\text{ if $i=j$},\\
0 &\text{ if $i\neq j$}.
\end{cases}$$

For a perfect complex $C$ of $R$-modules, we denote by ${\det}_R(C)$ the determinant module of $C$. This module is understood to be a graded invertible $R$-module (with the grade suppressed from the symbol).

For a number field $F$ and a finite set $S$ of places of $\QQ$, we denote by $\cO_{F,S}$ the ring of $S_F$-integers of $F$, where $S_F$ denotes the set of places of $F$ lying above a place in $S$. In particular, $\cO_{\QQ,S}$ is denoted simply by $\ZZ_S$. We denote by $\rgamma(\cO_{F,S},-)$ the etale cohomology complex $\rgamma_{\text{\'et}}(\Spec (\cO_{F,S}),-)$.

As usual, the notation $H^i_f(F,-)$ indicates the Bloch-Kato Selmer group and $H_f^i(F_v,-)$ the Bloch-Kato local condition for a place $v$ of $F$ .

For an elliptic curve $E$ defined over $\QQ$, we denote by $L(E,s)$ the Hasse-Weil $L$-function of $E$. For a finite set $S$ of places of $\QQ$, we denote by $L_S(E,s)$ the $S$-truncated $L$-function of $E$. We denote by $L_S^\ast(E,1)$ the leading term at $s=1$.

The Tate-Shafarevich group of $E$ over a number field $F$ is denoted by $\sha(E/F)$. The product of Tamagawa factors of $E/\QQ$ is denoted by ${\rm Tam}(E)$.

We use some other standard notations concerning elliptic curves and modular curves, such as $\Gamma(E,\Omega_{E/\QQ}^1)$, $H_1(E(\CC),\QQ)$, $E_1(\QQ_p)$, $Y_1(N)$, $X_1(N)$, etc.

\section{Formulation of the Generalized Perrin-Riou Conjecture}\label{sec 2}

We fix a prime number $p$ and assume throughout the article that $p$ is odd.

\subsection{Kato's Euler system}\label{sec kato}

Let $E$ be an elliptic curve over $\QQ$ of conductor $N$.

Fix a modular parametrization $\phi: X_1(N) \to E$ and write $ f=\sum_{n=1}^\infty a_n q^n$
for the normalized newform of weight 2 and level $N$ corresponding to $E$.

Let $T_p(E)$ be the $p$-adic Tate module of $E$ and set $V:=\QQ_p\otimes_{\ZZ_p}T_p(E)$. Let $T$ be a $G_\QQ$-stable sublattice of $V$ that is given by the image of the following map:
\begin{eqnarray}\label{mod}
H^1(Y_1(N)\times_\QQ \overline \QQ, \ZZ_p(1)) &\hookrightarrow &H^1(Y_1(N)\times_\QQ \overline \QQ, \QQ_p(1)) \\
&\twoheadrightarrow& H^1(X_1(N)\times_\QQ \overline \QQ,\QQ_p(1)) \nonumber \\
&\xrightarrow{\phi_\ast}& H^1(E\times_\QQ \overline \QQ,\QQ_p(1))\nonumber \\
&=&V^\ast(1) \nonumber\\
&\simeq& V,\nonumber
\end{eqnarray}
where the second arrow is the Manin-Drinfeld splitting, the third is induced by $\phi$ and the last is induced by the Weil pairing.

Note that $T$ identifies with the maximal quotient of $H^1(Y_1(N)\times_\QQ \overline \QQ, \ZZ_p(1))$ on which Hecke operators $T(n)$ act via $a_n$ and may be different from $T_p(E)$. If $E[p]$ is an irreducible $G_\QQ$-representation, we may assume $T= T_p(E)$.

We fix the following data:
\begin{itemize}
\item an embedding $\overline \QQ \hookrightarrow \CC$;
\item a finite set $S$ of places of $\QQ$ such that $\{\infty\}\cup \{\ell \mid pN\} \subset S$;
\item integers $c, d >1$ such that $cd$ is coprime to $6$ and all primes in $S$, and that $c \equiv d \equiv 1 $ (mod $N$);
\item an element $\xi \in {\rm SL}_2(\ZZ)$.
\end{itemize}

For this data and any positive integer $m$ that is coprime to $cd$, Kato constructed in \cite[(8.1.3)]{katoasterisque} a `zeta element'
$${}_{c,d}z_m(\xi,S_m) := {}_{c,d}z_m^{(p)}(f,1,1,\xi,S_m\setminus \{\infty\})$$
in $H^1(\cO_{\QQ(\mu_m),S_m},T),$ where $S_m$ denotes the set $S \cup \{\ell \mid m\}$.

It is also known that the collection $({}_{c,d} z_m(\xi,S_m))_m$ forms an Euler system (see \cite[Ex. 13.3]{katoasterisque}).

For a finite abelian extension $F$ of $\QQ$ that is unramified outside $S$, we set
$${}_{c,d}z_F={}_{c,d}z_F(\xi,S):={\rm Cor}_{\QQ(\mu_m)/F}({}_{c,d}z_m(\xi,S)),$$
where $m=m_F$ denotes the conductor of $F$.

For later purposes we make a specific choice of $\xi$ as follows. Just as in (\ref{mod}), the fixed modular parametrization $\phi: X_1(N) \to E$ induces a map
\begin{eqnarray}\label{homo}
H_1(X_1(N)(\CC), \{{\rm cusps}\}, \ZZ) &\simeq &H^1(Y_1(N)(\CC),\ZZ(1)) \\
&\to& H^1(E(\CC),\QQ(1))\simeq H_1(E(\CC), \QQ),\nonumber
\end{eqnarray}
where the first and last isomorphisms are obtained by the Poincar\'e duality.

We write $\sH$ for the image of this map (so $\sH$ is a lattice of $H_1(E(\CC),\QQ)$) and let
$$\delta(\xi) \in \sH$$
denote the image under the map (\ref{homo}) of the modular symbol
$$\{\xi(0), \xi(\infty)\} \in H_1(X_1(N)(\CC), \{{\rm cusps}\}, \ZZ).$$
Let $g$ denote the complex conjugation and set $e^+:=(1+g)/2$.

We then fix $\xi$ so that the following condition is satisfied:
\begin{equation}\label{xi condition}
\text{the element}\,\, e^+\delta(\xi) \,\,\text{ of }\,\, H_1(E(\CC),\QQ)^+\,\,\text{ is a }\,\, \ZZ_{(p)}\text{-basis of }\, (\ZZ_{(p)} \otimes_\ZZ \sH)^+.
\end{equation}
The existence of such $\xi \in {\rm SL}_2(\ZZ)$ is justified as follows. By a well-known theorem of Manin, we know that $H_1(X_1(N)(\CC),\{{\rm cusps}\}, \ZZ)$ is generated by the set $\{ \{\alpha(0), \alpha(\infty)\} \mid \alpha \in {\rm SL}_2(\ZZ)\}$. This implies that the $\ZZ_{(p)}$-module $(\ZZ_{(p)} \otimes_\ZZ \sH)^+$ is generated by the set $\{e^+ \delta(\alpha) \mid \alpha \in {\rm SL}_2(\ZZ)\}$. Since $(\ZZ_{(p)} \otimes_\ZZ \sH)^+ \simeq \ZZ_{(p)}$ and $\ZZ_{(p)}$ is local, Nakayama's lemma implies the existence of $\xi \in {\rm SL}_2(\ZZ)$ such that $e^+\delta(\xi)$ generates $(\ZZ_{(p)} \otimes_\ZZ \sH)^+$.

Throughout this article, we also fix a minimal Weierstrass model of $E$ over $\ZZ$ and let
$$\omega \in \Gamma(E,\Omega_{E/\QQ}^1)$$
be the corresponding N\'eron differential.

We define the real period for $(\omega,\xi)$ by setting
\begin{eqnarray}\label{per xi}
\Omega_\xi:=\int_{e^+\delta(\xi)} \omega.
\end{eqnarray}
(In general, this integral need only agree with the usual real N\'{e}ron period $\Omega^+$ up to multiplication by an element of $\QQ^\times$. However, if $E[p]$ is irreducible, then $\Omega_\xi$ and $\Omega^+$ will agree up to multiplication by an element of $\ZZ_{(p)}^\times$.)

Then Kato's reciprocity law \cite[Th. 6.6 and 9.7]{katoasterisque} gives the formula
\begin{eqnarray}\label{kato rec}
\exp_\omega^\ast({}_{c,d}z_\QQ)=cd(c-1)(d-1) \frac{L_S(E,1)}{\Omega_\xi} \text{ in }\QQ,
\end{eqnarray}
where $\exp_\omega^\ast: H^1(\ZZ_S,T) \to H^1(\QQ_p, T) \to \QQ_p$ is the dual exponential map associated to $\omega$.

\begin{remark}
As in \cite[Th. 12.5]{katoasterisque}, one may normalize Kato's zeta element in order to construct an element $z$ of $H^1(\ZZ_S,V)$ with the property that $\exp_\omega^\ast(z)= L_{\{p\}}(E,1)/\Omega^+,$ where the $L$-fucntion is truncated just at $p$ rather than at all places in $S$. However, one does not in general know that this element $z$ lies in $H^1(\ZZ_S,T)$. This delicate integrality issue is the reason that we prefer to use ${}_{c,d}z_\QQ={}_{c,d}z_\QQ(\xi,S)$ rather than the normalized element. In addition, if $H^1(\ZZ_S,T)$ is $\ZZ_p$-free,  then one expects that the element
$$z_\QQ:=\frac{1}{cd(c-1)(d-1)}\cdot {}_{c,d}z_\QQ$$ 
of $H^1(\ZZ_S,V)$ actually belongs to $H^1(\ZZ_S,T)$ but, as far as we are aware, this has not been proved in full generality.

\end{remark}

\subsection{Birch and Swinnerton-Dyer elements}\label{sec bsd}

In this subsection, we introduce a natural notion of `Birch and Swinnerton-Dyer element'.

Such elements constitute an analogue for elliptic curves of the `Rubin-Stark elements' that are associated to the multiplicative group.

In the sequel we shall denote the `algebraic rank' ${\rm rank}(E(\QQ))$ of $E$ over $\QQ$ by $r_{\rm alg}$ or often, for simplicity, by $r$.

Throughout this section we shall then assume the following.

\begin{hypothesis}\label{hyp}\
\begin{itemize}
\item[(i)] $H^1(\ZZ_S,T)$ is $\ZZ_p$-free;
\item[(ii)] $r_{\rm alg} > 0$;
\item[(iii)] $\sha(E/\QQ)[p^\infty]$ is finite.
\end{itemize}
\end{hypothesis}

\begin{remark}
If $E[p]$ is irreducible, then $T=T_p(E)$ and $E(\QQ)[p]=0$ so Hypothesis \ref{hyp}(i) is automatically satisfied. 
\end{remark}

Following \cite[Lem. 6.1]{bss2}, we note that these assumptions imply the existence of a canonical isomorphism
\begin{eqnarray}\label{h1}
H^1(\ZZ_S,V)\simeq \QQ_p \otimes_\ZZ E(\QQ)
\end{eqnarray}
and also, since the image of the localization map $H^1(\ZZ_S,V) \to H^1(\QQ_p,V)$ lies in $H^1_{f}(\QQ_p,V)=\QQ_p \otimes_{\ZZ_p}E_1(\QQ_p) $, of a canonical short exact sequence
\begin{eqnarray}\label{short}
0 \to \QQ_p \otimes_{\ZZ_p} E_1(\QQ_p)^\ast \to \QQ_p \otimes_\ZZ E(\QQ)^\ast \to H^2(\ZZ_S,V) \to 0.
\end{eqnarray}

We fix an embedding $\RR \hookrightarrow \CC_p$ and consider the following canonical `period-regulator' isomorphism of $\CC_p$-modules
\begin{eqnarray*}
\lambda: \CC_p \otimes_{\ZZ_p} {\bigwedge}_{\ZZ_p}^r H^1(\ZZ_S,T) &\simeq & \CC_p \otimes_\ZZ {\bigwedge}_\ZZ^r E(\QQ) \\
 &\simeq& \CC_p \otimes_\ZZ {\bigwedge}_\ZZ^r E(\QQ)^\ast \\
 &\simeq& \CC_p \otimes_{\QQ_p} \left(E_1(\QQ_p)^\ast \otimes_{\ZZ_p} {\bigwedge}_{\QQ_p}^{r-1} H^2(\ZZ_S,V) \right) \\
 &\simeq& \CC_p \otimes_{\QQ_p} \left( \Gamma(E,\Omega_{E/\QQ}^1) \otimes_\QQ  {\bigwedge}_{\QQ_p}^{r-1} H^2(\ZZ_S,V)\right)\\
 &\simeq& \CC_p \otimes_{\QQ_p} \left(H_1(E(\CC),\QQ)^{+,\ast} \otimes_\QQ {\bigwedge}_{\QQ_p}^{r-1} H^2(\ZZ_S, V) \right).
\end{eqnarray*}
Here the first isomorphism is induced by (\ref{h1}), the second by the N\'eron-Tate height pairing
$$ \langle -,-\rangle_\infty:E(\QQ) \times E(\QQ) \to \RR,$$
the third by (\ref{short}), the fourth by the dual exponential map
$$\exp^\ast: E_1(\QQ_p)^\ast \to \QQ_p \otimes_\QQ \Gamma(E,\Omega_{E/\QQ}^1),$$
the last by the period map
$$\Gamma(E,\Omega_{E/\QQ}^1) \to H_1(E(\CC),\RR)^{+,\ast}; \ \omega \mapsto (\gamma \mapsto \int_\gamma \omega).$$

\begin{definition}\label{def eta} Fix an element $\bm{x}$ of the space  ${\bigwedge}_{\QQ_p}^{r-1} H^2(\ZZ_S,V)$. Then the {\it Birch and Swinnerton-Dyer element} $\eta_{\bm{x}}^{\rm BSD} = \eta^{\rm BSD}_{\bm{x}}(\xi,S) $ of the data $\xi$, $S$ and $\bm{x}$ is the element of $\CC_p \otimes_{\ZZ_p} {\bigwedge}_{\ZZ_p}^r H^1(\ZZ_S,T)$ obtained by setting
$$\eta_{\bm{x}}^{\rm BSD}:=\lambda^{-1} \left( L_S^\ast(E,1) \cdot (e^+\delta(\xi)^\ast \otimes \bm{x})\right).$$
The `$(c,d)$-modified Birch and Swinnerton-Dyer element' for the given data is the element
$${}_{c,d}\eta_{\bm{x}}^{\rm BSD}:=cd(c-1)(d-1)\cdot\eta_{\bm{x}}^{\rm BSD}.$$
\end{definition}



\begin{remark}\label{bsd1} Each choice of an ordered basis of $E(\QQ)_{\rm tf}$ gives rise to a natural choice of element $\bm{x}$ as above (see \S\ref{exp int}). In the special case $r=1$ and $\bm{x}=1$, the above definition simplifies to an equality
$$\eta_{\bm{x}}^{\rm BSD} = \frac{L_S^\ast(E,1)}{\Omega_\xi\cdot R_\infty} \cdot \log_\omega(x)\cdot x$$
in $\CC_p\otimes_\ZZ E(\QQ) \simeq \CC_p \otimes_{\ZZ_p} H^1(\ZZ_S,T)$, where $R_\infty$ is the N\'eron-Tate regulator, $\log_\omega : E(\QQ) \to E(\QQ_p) \to \QQ_p$ is the formal group logarithm associated to $\omega$ and $x$ is any element of $E(\QQ)$ that generates $E(\QQ)_{\rm tf}$.
\end{remark}

The $p$-part of the Birch-Swinnerton-Dyer Formula for $E$ asserts that there should be an equality of $\ZZ_p$-submodules of $\CC_p$ of the form
 $$L^\ast(E,1) \cdot \ZZ_p = (\# \sha(E/\QQ)[p^\infty]\cdot {\rm Tam}(E)\cdot \# E(\QQ)_{\rm tors}^{-2} \cdot \Omega^+ \cdot R_\infty)
 \cdot \ZZ_p,$$
 where ${\rm Tam}(E)$ denotes the product of the Tamagawa factors of $E/\QQ$. 
In the sequel we shall abbreviate this equality of lattices to `${\rm BSD}_p(E)$'.

The next result explains the connection between this conjectural equality and the integrality properties of Birch and Swinnerton-Dyer elements.

 \begin{proposition} \label{prop eta} Set $r:= r_{\rm alg}$ and fix a $\ZZ_p$-basis $\bm{x}$ of the lattice ${\bigwedge}_{\ZZ_p}^{r-1} H^2(\ZZ_S,T)_{\rm tf}$. Then ${\rm BSD}_p(E)$ is valid if and only if there is an equality of $\ZZ_p$-lattices
 \begin{eqnarray}\label{eta bsd}
 \ZZ_p\cdot \eta_{\bm{x}}^{\rm BSD} = \# H^2(\ZZ_S,T)_{\rm tors}\cdot {\bigwedge}_{\ZZ_p}^r H^1(\ZZ_S,T).
 \end{eqnarray}
 In particular, the validity of ${\rm BSD}_p(E)$ implies that $\eta_{\bm{x}}^{\rm BSD}$ belongs to ${\bigwedge}_{\ZZ_p}^r H^1(\ZZ_S,T)$.
 \end{proposition}

 \begin{proof} It is well-known that the validity of ${\rm BSD}_p(E)$ is equivalent to the equality of lattices that underlies the statement of the Tamagawa Number Conjecture (or `TNC' for short) for the pair $(h^1(E)(1),\ZZ_p)$ (this has been shown, for example, by Kings in \cite{kings}). It is therefore sufficient to show that the equality (\ref{eta bsd}) is equivalent to the TNC and to do this we must recall the formulation of the
latter conjecture.

The statement of the TNC involves a canonical isomorphism of $\CC_p$-modules
\begin{equation}\label{TNC iso}\vartheta: \CC_p \otimes_{\ZZ_p}{\det}^{-1}_{\ZZ_p}(\rgamma_c(\ZZ_S,T^\ast(1))) \xrightarrow{\sim} \CC_p\end{equation}
that arises as follows. Firstly, global duality induces a canonical  isomorphism
$${\det}_{\ZZ_p}^{-1}(\rgamma_c(\ZZ_S,T^\ast(1))) \simeq {\det}_{\ZZ_p}^{-1}(\rgamma(\ZZ_S,T)) \otimes_{\ZZ_p} T^\ast(1)^{+}$$
(cf. \cite[Prop. 2.22]{sbA}) and hence also a canonical isomorphism
\begin{eqnarray}\label{coh}
&&\CC_p \otimes_{\ZZ_p}{\det}_{\ZZ_p}^{-1}(\rgamma_c(\ZZ_S,T^\ast(1))) \\
&\simeq& \CC_p \otimes_{\QQ_p}\left({\bigwedge}_{\QQ_p}^r H^1(\ZZ_S,V) \otimes_{\QQ_p} {\bigwedge}_{\QQ_p}^{r-1}H^2(\ZZ_S,V)^\ast \otimes_{\QQ_p} V^\ast(1)^+ \right).\nonumber
\end{eqnarray}

The isomorphism $\vartheta$ in (\ref{TNC iso}) is then obtained by combining the latter isomorphism with the canonical `comparison' isomorphism
$$V^\ast(1)^+ \simeq \QQ_p \otimes_\QQ H^1(E(\CC),\QQ(1))^+ \simeq \QQ_p \otimes_\QQ H_1(E(\CC),\QQ)^+$$
and the period-regulator isomorphism
$$\lambda: \CC_p \otimes_{\QQ_p}{\bigwedge}_{\QQ_p}^{r} H^1(\ZZ_S,V) \simeq \CC_p \otimes_{\QQ_p} \left(H_1(E(\CC),\QQ)^{+,\ast} \otimes_\QQ {\bigwedge}_{\QQ_p}^{r-1} H^2(\ZZ_S, V) \right)$$
constructed earlier.

If $\mathfrak{z}$ is the unique element of $\CC_p\otimes_{\ZZ_p} {\det}_{\ZZ_p}^{-1}(\rgamma_c(\ZZ_S,T^\ast(1)))$ that satisfies
 $\vartheta(\mathfrak{z})=L_S^\ast(E,1),$ then the TNC predicts that
$$\ZZ_p \cdot \mathfrak{z} ={\det}_{\ZZ_p}^{-1}(\rgamma_c(\ZZ_S,T^\ast(1))).$$

Given this, the claimed result is a consequence of the fact that the isomorphism (\ref{coh}) sends the element $\mathfrak{z}$ to  $$\eta_{\bm{x}}^{\rm BSD} \otimes \bm{x}^\ast \otimes e^+\delta(\xi) \in \CC_p \otimes_{\QQ_p}\left({\bigwedge}_{\QQ_p}^r H^1(\ZZ_S,V) \otimes_{\QQ_p} {\bigwedge}_{\QQ_p}^{r-1}H^2(\ZZ_S,V)^\ast \otimes_{\QQ_p} V^\ast(1)^+ \right),$$
and the lattice ${\det}_{\ZZ_p}^{-1}(\rgamma_c(\ZZ_S,T^\ast(1)))$ to %
$$\# H^2(\ZZ_S,T)_{\rm tors} \cdot {\bigwedge}_{\ZZ_p}^r H^1(\ZZ_S,T) \otimes_{\ZZ_p} {\bigwedge}_{\ZZ_p}^{r-1} H^2(\ZZ_S,T)_{\rm tf}^\ast \otimes_{\ZZ_p} T^\ast(1)^+.$$
\end{proof}

\subsection{Bockstein regulator maps}\label{sec boc}
In this subsection, we shall introduce a canonical construction of Bockstein regulator maps (see (\ref{boc}) below).

We first set some notations. Let $F/\QQ$ be a finite abelian extension unramified outside $S$ and $G$ its Galois group. Since all results and conjectures we study are of $p$-adic nature, we may assume that $[F:\QQ]$ is a $p$-power. In particular, since $p$ is odd, $F$ is a totally real field. The augmentation ideal
$$I_F:=\ker (\ZZ_p[G] \twoheadrightarrow \ZZ_p)$$
and the augmentation quotients
$$Q_F^a:=I_F^a/I_F^{a+1}$$
for a non-negative integer $a$ will play important roles. We remark that $Q_F^0$ is understood to be $\ZZ_p[G]/I_F=\ZZ_p$.

For simplicity, in this subsection we shall abbreviate the ideal $I_F$ to $I$. 

At the outset we note that the tautological short exact sequence
$$0 \to I/I^2 \to \ZZ_p[G]/I^2 \to \ZZ_p \to 0$$
gives rise to a canonical exact triangle of complexes of $\ZZ_p$-modules of the form
\begin{multline*}\rgamma(\cO_{F,S},T)\otimes^{\DL}_{\ZZ_p[G]} I/I^2 \to \rgamma(\cO_{F,S},T)\otimes^{\DL}_{\ZZ_p[G]} \ZZ_p[G]/I^2 \to \rgamma(\cO_{F,S},T)\otimes^{\DL}_{\ZZ_p[G]} \ZZ_p.
\end{multline*}

Next we recall (from, for example, \cite[Prop. 1.6.5]{FK}) that  $\rgamma(\cO_{F,S},T)$ is acyclic outside degrees one and two and that there exists a canonical isomorphism in the derived category of $\ZZ_p$-modules
\begin{eqnarray}\label{control}
\rgamma(\cO_{F,S},T)\otimes^{\DL}_{\ZZ_p[G]}\ZZ_p \simeq \rgamma(\ZZ_S,T).
\end{eqnarray}

Taking account of these facts, the above triangle gives rise to a morphism of complexes of $\ZZ_p$-modules
\[ \delta_F: \rgamma(\ZZ_S,T)\to\bigl(\rgamma(\ZZ_S,T)
\otimes_{\ZZ_p}^{\DL}I/I^2 \bigr)[1]\]
and hence to a composite homomorphism of $\ZZ_p$-modules
\begin{eqnarray}\label{def beta}
\beta_F: H^1(\ZZ_S,T) &\xrightarrow{(-1)\times H^1(\delta_F)}& H^2(\rgamma(\ZZ_S,T)\otimes_{\ZZ_p}^{\DL}I/I^2 )\\
&=&H^2(\ZZ_S,T)\otimes_{\ZZ_p}I/I^2  \nonumber \\
&\twoheadrightarrow & H^2(\ZZ_S,T)_{\rm tf}\otimes_{\ZZ_p}I/I^2,\nonumber
\end{eqnarray}
in which the equality is valid since $\rgamma(\ZZ_S,T)$ is acyclic in degrees greater than two and the last map is induced by the natural map from $H^2(\ZZ_S,T)$ to $H^2(\ZZ_S,T)_{\rm tf}$.

We write
$${\rm Boc}_{F}: {\bigwedge}_{\ZZ_p}^r H^1(\ZZ_S,T) \to H^1(\ZZ_S,T) \otimes_{\ZZ_p} {\bigwedge}_{\ZZ_p}^{r-1} H^2(\ZZ_S,T)_{\rm tf}
\otimes_{\ZZ_p} Q_F^{r-1}$$
for the homomorphism of $\ZZ_p$-modules with the property that
\[ {\rm Boc}_{F}\bigl(y_1\wedge \cdots \wedge y_{r}\bigr) = \sum_{i=1}^{r}(-1)^{i+1}y_i \otimes \bigl(
\beta_{F}(y_{1}) \wedge \cdots \wedge \beta_{F}(y_{i-1}) \wedge \beta_F(y_{i+1}) \wedge \cdots \wedge \beta_{F}(y_{r})\bigr)\]
for all elements $y_i$ of $H^1(\ZZ_S,T)$. 

%
%

Then, each choice of basis element $\bm{x}$ of the (free, rank one) $\ZZ_p$-module ${\bigwedge}_{\ZZ_p}^{r-1}H^2(\ZZ_S,T)_{\rm tf}$,
 gives rise to a composite `Bockstein regulator' homomorphism
\begin{align}\label{boc}
{\rm Boc}_{F,\bm{x}}: {\bigwedge}_{\ZZ_p}^r H^1(\ZZ_S,T)
\xrightarrow{{\rm Boc}_F} &\, H^1(\ZZ_S,T) \otimes_{\ZZ_p} {\bigwedge}_{\ZZ_p}^{r-1} H^2(\ZZ_S,T)_{\rm tf}
\otimes_{\ZZ_p} Q_F^{r-1}\\
 \xrightarrow{{\rm id}\otimes\phi_{\bm{x}}\otimes{\rm id}} &\,  H^1(\ZZ_S,T)\otimes_{\ZZ_p} Q_F^{r-1},\notag
\end{align}
where $\phi_{\bm{x}}$ is the isomorphism ${\bigwedge}_{\ZZ_p}^{r-1} H^2(\ZZ_S,T)_{\rm tf}\simeq \ZZ_p$ induced by the choice of ${\bm x}$.

\begin{remark}\label{r=1 reg} If $r=1$ and $\bm{x}=1$ is the canonical basis of ${\bigwedge}_{\ZZ_p}^{r-1}H^2(\ZZ_S,T)_{\rm tf} =\ZZ_p$, then ${\rm Boc}_{F,\bm{x}}= {\rm Boc}_{F}$ is simply equal to the identity map on $H^1(\ZZ_S,T)$.\end{remark}

\subsection{The Generalized Perrin-Riou Conjecture} In the sequel we shall write $r_{\rm an}$ for the analytic rank $\ord_{s=1} L(E,s)$ of $E$.

\subsubsection{}In \cite{PR}, Perrin-Riou investigates relations between Kato's Euler system and the $p$-adic Birch-Swinnerton-Dyer Conjecture. In particular, she formulates the following conjecture.

\begin{conjecture}[{Perrin-Riou \cite{PR}, see also \cite{buyuk perrin}}]\label{PR}\
\begin{itemize}
\item[(i)] The element ${}_{c,d}z_\QQ$ is non-zero if and only if $r_{\rm an}$ is at most one.
\item[(ii)] If $r_{\rm an}=r_{\rm alg}=1$, then in $\CC_p \otimes_{\ZZ_p} H^1(\ZZ_S,T) \simeq \CC_p\otimes_\ZZ E(\QQ)$ one has
\begin{eqnarray}\label{PR formula}
{}_{c,d}z_\QQ=cd(c-1)(d-1) \frac{L_S'(E,1)}{\Omega_\xi\cdot R_\infty}  \log_\omega(x)\cdot x, \end{eqnarray}
where $x$ is any element of $E(\QQ)$ that generates $E(\QQ)_{\rm tf}$.
\end{itemize}
\end{conjecture}

\begin{remark} This conjecture is a slight modification of, but equivalent to, Perrin-Riou's original formulation of the conjecture. 
%
By Kato's reciprocity law (\ref{kato rec}), the element ${}_{c,d}z_\QQ$ is explicitly related to $L(E,1)$ and, in particular, does not vanish if $r_{\rm an}=0$. Perrin-Riou's conjecture predicts that ${}_{c,d}z_\QQ$ does not vanish even if $r_{\rm an}=1$ and, moreover, that it should be explicitly related to the first derivative $L'(E,1)$ via the formula (\ref{PR formula}).\end{remark}

By Remark \ref{bsd1}, we immediately obtain the following interpretation of Perrin-Riou's conjecture in terms of the BSD element.

\begin{proposition}\label{PR2} If $r_{\rm an}= r_{\rm alg}=1$ and ${\bm{x}} = 1$, then
Conjecture \ref{PR}(ii) is valid if and only if one has ${}_{c,d}z_\QQ = {}_{c,d}\eta_{\bm{x}}^{\rm BSD}.$
\end{proposition}

\begin{remark} An interpretation of Perrin-Riou's conjecture in the same style as Proposition \ref{PR2} was previously given by Sakamoto and the first and the third authors in \cite[\S 6]{bss2}. (In fact, a natural `equivariant' refinement of this conjecture is also formulated in loc. cit.)
\end{remark}

\subsubsection{}\label{gen PR}We shall now give a precise formulation of Conjecture \ref{conj01}.

For this purpose we will always assume the validity of Hypothesis \ref{hyp}. We also use the notation $I_F$ and $Q_{F}^{a}$ introduced in \S\ref{sec boc}.

%
%


We set $r := r_{\rm alg}$ and
 write
\begin{eqnarray}\label{iota}
\iota_F: H^1(\ZZ_S,T) \otimes_{\ZZ_p} Q_F^{r-1}& \to& H^1(\cO_{F,S},T) \otimes_{\ZZ_p} Q_F^{r-1} \\
&\to& H^1(\cO_{F,S},T) \otimes_{\ZZ_p}\ZZ_p[G]/I_F^r \nonumber
\end{eqnarray}
for the composite homomorphism that is induced by the restriction map $H^1(\ZZ_S,T) \to H^1(\cO_{F,S},T)$ and the natural inclusion $Q_F^{r-1} \hookrightarrow \ZZ_p[G]/I_F^r$. (This map $\iota_F$ is actually injective - see the discussion in \S\ref{darmon derivatives sec}.)

Motivated by constructions of Darmon in \cite{D} and \cite{DH} (relating to cyclotomic units and to Heegner points respectively), we define the  `Darmon norm' of ${}_{c,d}z_F$ to be the element of $H^1(\cO_{F,S},T) \otimes_{\ZZ_p} \ZZ_p[G]$ obtained by setting
\begin{eqnarray*}\label{darmon norm}
\cN_{F/\QQ}({}_{c,d}z_F) :=\sum_{\sigma \in G} \sigma ( {}_{c,d}z_F) \otimes \sigma^{-1}.
\end{eqnarray*}


We can now give a precise formulation of Conjecture \ref{conj01}. This prediction involves the Birch-Swinnerton-Dyer element ${}_{c,d}\eta_{\bm{x}}^{\rm BSD}$ and Bockstein regulator map ${\rm Boc}_{F,\bm{x}}$ that were respectively defined in \S\ref{sec bsd} and \S\ref{sec boc}.

\begin{conjecture}[{The Generalized Perrin-Riou Conjecture}]\label{mrs} Set $r := r_{\rm alg}$. Then for each $\ZZ_p$-basis element $\bm{x}$ of ${\bigwedge}_{\ZZ_p}^{r-1}H^2(\ZZ_S,T)_{\rm tf}$ the following claims are valid.

\begin{itemize}
\item[(i)] The element ${}_{c,d}\eta_{\bm{x}}^{\rm BSD}$ belongs to ${\bigwedge}_{\ZZ_p}^r H^1(\ZZ_S,T).$

\item[(ii)] The image in $H^1(\cO_{F,S},T)\otimes_{\ZZ_p} \ZZ_p[G]/I_F^r$ of the Darmon norm $\cN_{F/\QQ} ({}_{c,d}z_F)$ of ${}_{c,d}z_F$ is equal to $\iota_F \bigl({\rm Boc}_{F,\bm{x}}({}_{c,d}\eta_{\bm{x}}^{\rm BSD})\bigr).$
\end{itemize}
\end{conjecture}

\begin{remark}\label{PR equiv}\

\noindent{}(i) Proposition \ref{prop eta} shows that Conjecture \ref{mrs}(i) is implied by the validity of ${\rm BSD}_p(E)$.

\noindent{}(ii) Assume $r_{\rm alg}=1$ and that $\bm{x}=1$ in ${\bigwedge}_{\ZZ_p}^{r-1}H^2(\ZZ_S,T)_{\rm tf} = \ZZ_p$. Then in this case one has
$$\cN_{F/\QQ} ({}_{c,d}z_F) = {\N}_{F/\QQ}({}_{c,d}z_F) \text{ in }H^1(\cO_{F,S},T)\otimes_{\ZZ_p}\ZZ_p[G]/I_F \simeq H^1(\cO_{F,S},T),$$
where ${\N}_{F/\QQ}:=\sum_{\sigma \in G}\sigma$. In particular, since ${\rm Cor}_{F/\QQ}({}_{c,d}z_F)={}_{c,d}z_\QQ$ and ${\rm Boc}_{F,\bm{x}}$ is the identity map on $H^1(\ZZ_S,T)$ (by Remark \ref{r=1 reg}), Conjecture \ref{mrs} is equivalent in this case to an equality ${}_{c,d}z_\QQ ={}_{c,d}\eta_{\bm{x}}^{\rm BSD}.$  From Proposition \ref{PR2} it therefore follows that if $r_{\rm an}=r_{\rm alg}=1$ then Conjecture \ref{mrs} is equivalent to Perrin-Riou's conjecture (as stated in Conjecture \ref{PR}(ii)). This observation proves Theorem \ref{theorem01} and also motivates us to refer to Conjecture \ref{mrs} as the `Generalized
Perrin-Riou Conjecture'.
\end{remark}

\begin{remark} The formulation of Conjecture \ref{mrs} can also be regarded as a natural analogue for elliptic curves of the conjectural `refined class number formula for $\mathbb{G}_m$' concerning Rubin-Stark elements that was originally formulated independently by Mazur and Rubin  \cite[Conj. 5.2]{MRGm} and by the third author \cite[Conj. 3]{sano} and then subsequently refined by the present authors in \cite[Conj. 5.4]{bks1}. 
\end{remark}

\begin{remark}
It is straightforward to show that the element ${\rm Boc}_{F,\bm{x}}({}_{c,d}\eta_{\bm{x}}^{\rm BSD}) $, and hence also the validity of Conjecture \ref{mrs}(ii), is independent of the choice of basis element $\bm{x}$.
\end{remark}

\begin{remark}\label{rem inj} In \S\ref{darmon derivatives sec} we will reinterpret Conjecture \ref{mrs} in terms of a natural `Darmon-type' derivative of ${}_{c,d}z_F$.\end{remark}

\subsection{An algebraic analogue}\label{alg bsd sec} We now formulate an analogue of Conjecture \ref{mrs} that is more algebraic, and elementary, in nature.

To do this we recall that if $\sha(E/\QQ)$ is finite, then the Birch and Swinnerton-Dyer Formula for $E$ predicts that
\begin{equation}\label{BSD eq}L_S^\ast(E,1) =   \left(\prod_{\ell \in S \setminus \{\infty\}} L_\ell \right)\frac{\# \sha(E/\QQ)\cdot{\rm Tam}(E)\cdot \Omega^+\cdot R_\infty}{\# E(\QQ)_{\rm tors}^2},\end{equation}
where $\Omega^+$ is the usual real N\'{e}ron period of $E$,
$L_\ell$ is the standard Euler factor at $\ell$ of the Hasse-Weil $L$-function (so that $(\prod_{\ell \in S\setminus \{\infty\}} L_\ell)L^\ast(E,1)=L_S^\ast(E,1)$) and ${\rm Tam}(E)$ is the product of Tamagawa factors.

\begin{definition}\label{alg bsd def} Set $r := r_{\rm alg}$. Then for each element $\bm{x}$ of ${\bigwedge}_{\QQ_p}^{r-1}H^2(\ZZ_S,V)$ the  {\it algebraic Birch and Swinnerton-Dyer element} $\eta_{\bm{x}}^{\rm alg} = \eta^{\rm alg}_{\bm{x}}(\xi,S) $ of the data $\xi$, $S$ and $\bm{x}$ is  the element of $\CC_p \otimes_{\ZZ_p} {\bigwedge}_{\ZZ_p}^r H^1(\ZZ_S,T)$ obtained by setting
$$\eta_{\bm{x}}^{\rm alg}:=\lambda^{-1} \left( \left(\prod_{\ell \in S \setminus \{\infty\}} L_\ell \right)\frac{\# \sha(E/\QQ)\cdot {\rm Tam}(E)\cdot\Omega^+\cdot R_\infty}{\# E(\QQ)_{\rm tors}^2} \cdot (e^+\delta(\xi)^\ast \otimes \bm{x})\right).$$
The `$(c,d)$-modified algebraic Birch and Swinnerton-Dyer element' of the given data is then defined by setting
\[ {}_{c,d}\eta_{\bm{x}}^{\rm alg} := cd(c-1)(d-1)\cdot\eta_{\bm{x}}^{\rm alg}.\]
\end{definition}

\begin{remark}\label{alg an comp} It is clear that, if ${\bm{x}}$ is non-zero, then the Birch and Swinnerton-Dyer Formula  (\ref{BSD eq}) is valid for $E$ if and only if the elements $\eta_{\bm{x}}^{\rm alg}$ and ${}_{c,d}\eta_{\bm{x}}^{\rm alg}$ are respectively equal to the Birch and Swinnerton-Dyer elements $\eta_{\bm{x}}^{\rm BSD}$ and ${}_{c,d}\eta_{\bm{x}}^{\rm BSD}$ from Definition \ref{def eta}.\end{remark}

An easy exercise shows that if $\bm{x}$ is a $\ZZ_p$-basis element of ${\bigwedge}_{\ZZ_p}^{r-1}H^2(\ZZ_S,T)_{\rm tf}$, then there is an equality of lattices
\begin{eqnarray} \label{alg lattice}
\ZZ_p\cdot \eta_{\bm{x}}^{\rm alg}=\# H^2(\ZZ_S,T)_{\rm tors}\cdot {\bigwedge}_{\ZZ_p}^rH^1(\ZZ_S,T)
\end{eqnarray}
and hence $\eta_{\bm{x}}^{\rm alg}$ belongs to ${\bigwedge}_{\ZZ_p}^rH^1(\ZZ_S,T)$.

Upon combining this fact with Remark \ref{alg an comp}, one is led to formulate the following algebraic analogue of Conjecture \ref{mrs}.

\begin{conjecture}[{The Refined Mazur-Tate Conjecture}]\label{algebraic mrs}  Set $r := r_{\rm alg}$. Then for each $\ZZ_p$-basis element $\bm{x}$ of ${\bigwedge}_{\ZZ_p}^{r-1}H^2(\ZZ_S,T)_{\rm tf}$ the image in $H^1(\cO_{F,S},T)\!\otimes_{\ZZ_p}\ZZ_p[G]/I_F^r$ of $\cN_{F/\QQ} ({}_{c,d}z_F)$  is equal to $\iota_F \bigl({\rm Boc}_{F,\bm{x}}({}_{c,d}\eta_{\bm{x}}^{\rm alg})\bigr)$.
\end{conjecture}

\begin{remark} We refer to this algebraic analogue of Conjecture \ref{mrs} as a refined Mazur-Tate Conjecture since in the complementary article \cite{bks5} we are able to prove that, under certain mild and natural hypotheses, the equality predicted by Conjecture \ref{algebraic mrs} is strictly finer than the celebrated congruences for modular elements that are conjectured by  Mazur and Tate in \cite{MT}. This fact is in turn a key ingredient in the approach used in \cite{bks5} to obtain the first verifications of the conjecture of Mazur and Tate for elliptic curves of strictly positive rank. \end{remark}

\section{Fitting ideals and order of vanishing}\label{zefi sec}

In this section we shall discuss a further arithmetic property of Kato's zeta elements and, in particular, use it to prove Theorem \ref{vanishing evidence}.

Throughout we fix $F$, $G$ and $I_F$ as in \S \ref{sec boc} and continue to assume that $H^1(\ZZ_S,T)$ is $\ZZ_p$-free. 
However, unless explicitly stated, in this subsection we do {\em not} need to assume either that $r_{\rm alg} > 0$ or that  $\sha(E/\QQ)[p^\infty]$ is finite.


\subsection{A `main conjecture' at finite level} We write $m$ for the conductor of $F$ and set
$$t_{c,d}:=cd(c-\sigma_c)(d-\sigma_d) \in \ZZ_p[G],$$
where $\sigma_a$ is the element of $G$ obtained by restriction of the automorphism of $\QQ(\mu_m)$ that sends $\zeta_m$ to $\zeta_m^a$.

We then propose the following conjecture involving the initial Fitting ideal of the $\ZZ_p[G]$-module $H^2(\cO_{F,S},T)$.

\begin{conjecture}\label{order}
$$\left\{\Phi({}_{c,d} z_F) \mid  \Phi \in \Hom_{\ZZ_p[G]}( H^1(\cO_{F,S},T), \ZZ_p[G]) \right\}=t_{c,d}\cdot {\rm Fitt}_{\ZZ_p[G]}^0(H^2(\cO_{F,S},T)).$$
\end{conjecture}

\begin{remark} Conjecture \ref{order} is analogous to the `weak main conjecture' for modular elements that is formulated by Mazur and Tate \cite[Conj. 3]{MT}. (In fact, since our conjecture predicts an equality rather than simply an inclusion, it corresponds to a strengthening of \cite[Conj. 3]{MT}). It is also an analogue of the conjectures \cite[Conj. 7.3]{bks1} and \cite[Conj. 3.6(ii)]{bks2-2} that were formulated by the present authors in the setting of the multiplicative group.
\end{remark}

The prediction in Conjecture \ref{order} can be studied by using the equivariant theory of Euler systems developed by Sakamoto and the first and third authors in \cite{bss2}. In this way, the following evidence for Conjecture \ref{order} is obtained in \cite[Th. 6.11]{bss2}.

\begin{proposition}\label{prop bss}
Assume that the following conditions are all satisfied.
\begin{itemize}
\item[(a)] $p>3$;
\item[(b)] $\sha(E/F)[p^\infty]$ and $\sha(E/\QQ)[p^\infty]$ are finite;
\item[(c)] the image of the representation $G_\QQ \to {\rm Aut}(T_p(E)) \simeq {\rm GL}_2(\ZZ_p)$ contains ${\rm SL}_2(\ZZ_p)$;
\item[(d)] $E(\QQ_\ell)[p]$ vanishes for all primes $\ell $ in $S$.
\end{itemize}
Then for any homomorphism $\Phi : H^1(\cO_{F,S},T)\to \ZZ_p[G]$ of $\ZZ_p[G]$-modules one has
$$\Phi({}_{c,d} z_F) \in {\rm Fitt}_{\ZZ_p[G]}^0(H^2(\cO_{F,S},T)).$$
\end{proposition}

\subsection{The proof of Theorem \ref{vanishing evidence}}

In the rest of this section, we assume the conditions (a), (b), (c) and (d) in Theorem \ref{vanishing evidence} (which are the same as those in Proposition \ref{prop bss}). In particular, $E[p]$ is irreducible by (c), and we may assume $T=T_p(E)$.

\subsubsection{}The connection between Conjecture \ref{order} and Conjecture \ref{conj01}(i) is explained by the following result.

\begin{proposition}\label{link prop} Assume that $E(F)[p]$ vanishes and that $\sha(E/\QQ)[p^\infty]$ is finite. Set $a:= \max \{ 0, r_{\rm alg}-1\}$ and define an ideal of $\ZZ_p[G]$ by setting
\[ I_{F,S,a} := \#H^2(\ZZ_S,T)_{\rm tors}\cdot I_F^a + I_F^{a+1} \subset I_F^a.\]
Then ${\mathcal N}_{F/\QQ}(z_{F})$ belongs to $H^1({\mathcal O}_{F,S}, T) \otimes_{\ZZ_p} I_{F,S,a}$ whenever one has
\[ \Phi({}_{c,d} z_F)\in {\rm Fitt}_{\ZZ_p[G]}^0(H^2(\cO_{F,S},T))\]
for all $\Phi\in \Hom_{\ZZ_p[G]}\bigl(H^1(\cO_{F,S},T),  \ZZ_p[G]\bigr)$.
\end{proposition}

\begin{proof} At the outset we fix a surjective homomorphism of $\ZZ_p[G]$-modules of the form
$$P'_F \twoheadrightarrow H^1(\cO_{F,S},T)^\ast$$
in which $P'_F$ is both finitely generated and free.

Then the linear dual $P_F := (P'_F)^\ast$ is a finitely generated free $\ZZ_p[G]$-module and the above surjection induces an injective homomorphism of $\ZZ_p[G]$-modules
\begin{equation}\label{dual inj}H^1(\cO_{F,S},T)^{\ast \ast} \hookrightarrow P_F,\end{equation}
the cokernel of which is $\ZZ_p$-free. In addition, since $E(F)[p]$ is assumed to vanish the $\ZZ_p$-module $H^1(\cO_{F,S},T)$ is free and so identifies with
$H^1(\cO_{F,S},T)^{\ast \ast}$.

We may therefore use (\ref{dual inj}) to identify $H^1(\cO_{F,S},T)$ as a submodule of the free module $P_F$ in such a way that the quotient $P_F/H^1(\cO_{F,S},T)$ is torsion-free.

Having done this, the argument of  \cite[Prop. 4.17]{bks1} shows that the validity of the displayed inclusion for all $\Phi$ in
$\Hom_{\ZZ_p[G]}\bigl(H^1(\cO_{F,S},T),  \ZZ_p[G]\bigr)$ is equivalent to asserting that one has
$$\cN_{F/\QQ}({}_{c,d}z_F)\in P_F \otimes_{\ZZ_p} J_{F,S}$$
with $J_{F,S} := {\rm Fitt}_{\ZZ_p[G]}^0(H^2(\cO_{F,S},T))$ and, moreover, that the projection of $\cN_{F/\QQ}({}_{c,d}z_F)$ to $P_F\otimes_{\ZZ_p} (J_{F,S}/I_F\cdot J_{F,S})$ belongs to the submodule $P_F^G \otimes_{\ZZ_p}(J_{F,S}/I_F\cdot J_{F,S})$.

To complete the proof it is therefore enough to show that
\begin{equation}\label{fitting inclusion} J_{F,S} \subset I_{F,S,a}.\end{equation}

To do this we note that $H^i(\cO_{F,S},T)$ vanishes for all $i > 2$ and hence that the natural corestriction map $H^2(\cO_{F,S},T)\twoheadrightarrow H^2(\ZZ_S,T)$ is surjective.

In addition, since $\sha(E/\QQ)[p^\infty]$ is assumed to be finite, the $\ZZ_p$-rank of $H^2(\ZZ_S,T)$ is equal to $a$ and so the $\ZZ_p$-module $H^2(\ZZ_S,T)$ is isomorphic to $H^2(\ZZ_S,T)_{\rm tors}\oplus \ZZ_p^a$.

The corestriction map therefore induces a surjective homomorphism of $\ZZ_p[G]$-modules
\[ H^2(\cO_{F,S},T)\twoheadrightarrow H^2(\ZZ_S,T)_{\rm tors}\oplus \ZZ_p^a\]
and hence an inclusion of Fitting ideals
\begin{multline*} J_{F,S} = {\rm Fitt}_{\ZZ_p[G]}^0(H^2(\cO_{F,S},T))\\ \subset {\rm Fitt}_{\ZZ_p[G]}^0(H^2(\ZZ_S,T)_{\rm tors}\oplus \ZZ_p^a) = {\rm Fitt}_{\ZZ_p[G]}^0(H^2(\ZZ_S,T)_{\rm tors})\cdot I_F^{a}.\end{multline*}

To deduce (\ref{fitting inclusion}) from this it is thus enough to note the image of
 ${\rm Fitt}_{\ZZ_p[G]}^0(H^2(\ZZ_S,T)_{\rm tors})$ under the natural map $\ZZ_p[G] \to \ZZ_p[G]/I_F \simeq \ZZ_p$ is equal to
\[ {\rm Fitt}_{\ZZ_p}^0((H^2(\ZZ_S,T)_{\rm tors})_G) = {\rm Fitt}_{\ZZ_p}^0(H^2(\ZZ_S,T)_{\rm tors}) = \#H^2(\ZZ_S,T)_{\rm tors}\cdot \ZZ_p.\]
%
\end{proof}

\begin{remark} The containment discussed in Proposition \ref{link prop} would imply that
\begin{eqnarray}\label{order2}
\Phi({}_{c,d} z_F) \in I_F^{a} \text{ for all }\Phi \in \Hom_{\ZZ_p[G]}(H^1(\cO_{F,S},T),\ZZ_p[G]).
\end{eqnarray}
This prediction constitutes an analogue for Kato's Euler system ${}_{c,d} z_F$ of the `weak vanishing' conjecture for modular elements that is formulated by Mazur and Tate in \cite[Conj. 1]{MT}.
\end{remark}

\subsubsection{}If the algebraic rank $r := r_{\rm alg}$ of $E$ over $\QQ$ is strictly positive, then the integer $a$ in Proposition \ref{link prop} is equal to $r-1$ and so one has $I_F^a = I_F^{r-1}$.

One therefore obtains a proof of Theorem \ref{vanishing evidence} directly upon combining the results of Propositions \ref{prop bss} and \ref{link prop}.

\section{Derivatives of Kato's Euler system}\label{sec zp}
In this section, we shall define a canonical `Darmon derivative'  ${}_{c,d}\kappa_F$ of Kato's zeta element ${}_{c,d}z_F$ and use it to reinterpret the conjectures formulated above. 

In particular, in this way we are able to formulate more explicit versions of the Conjectures \ref{mrs} and \ref{algebraic mrs} for subfields $F$ of the cyclotomic
$\ZZ_p$-extension of $\QQ$.


{\it Throughout this section, we assume that $H^1(\ZZ_S,T)$ is $\ZZ_p$-free 
and $\sha(E/\QQ)[p^\infty]$ is finite.}

\subsection{Darmon derivatives}\label{darmon derivatives sec}

We use the notations in \S \ref{sec boc}. 

At the outset we note that the map $\iota_F$ in (\ref{iota}) is injective.

This follows easily from the facts that $H^1(\ZZ_S,T)$ is $\ZZ_p$-free and that $H^1(\ZZ_S,T)$ identifies with the submodule $H^1(\cO_{F,S},T)^G$ of $G$-invariant elements in $H^1(\cO_{F,S},T)$ (since $H^0(\ZZ_S,T)$ vanishes).

Conjecture \ref{mrs} is therefore equivalent to asserting the existence of a unique element
$${}_{c,d}\kappa_F \in H^1(\ZZ_S,T) \otimes_{\ZZ_p}Q_F^{r-1}$$
with the property that both
$$\iota_F({}_{c,d}\kappa_F)=\cN_{F/\QQ}({}_{c,d}z_F) \text{ and }{}_{c,d}\kappa_F={\rm Boc}_{F,\bm{x}}({}_{c,d}\eta_{\bm{x}}^{\rm BSD}).$$
(In particular, if $r=1$, then ${}_{c,d}\kappa_F$ is simply equal to ${}_{c,d}z_\QQ$.)

We regard this element ${}_{c,d}\kappa_F$ as a `Darmon-type derivative' of the zeta element ${}_{c,d}z_F$ and first consider conditions under which it can be unconditionally defined.

\subsubsection{}To do this we fix a finitely generated free $\ZZ_p[G]$-module $P_F$ and an injective homomorphism of $\ZZ_p[G]$-modules $j_F: H^1(\cO_{F,S},T) \to P_F$ as in the proof of Proposition \ref{link prop}.

We use $j_F$ to regard $H^1(\ZZ_S,T) = H^1(\cO_{F,S},T)^G$ as a submodule of $P_\QQ:=P_F^G$. Then, just as in (\ref{iota}), there are natural injective homomorphisms
$$\iota_F: P_\QQ \otimes_{\ZZ_p}Q_F^a \hookrightarrow P_F \otimes_{\ZZ_p} Q_F^a \hookrightarrow P_F\otimes_{\ZZ_p} \ZZ_p[G]/I_F^{a+1}$$
(where, we recall, $Q_F^a$ denotes $I_F^a/I_F^{a+1}$).



%

\begin{definition}\label{def kappa} Set $a:={\rm max}\{0,r_{\rm alg}-1\}$ and assume that the containment (\ref{order2}) is valid for all $\Phi$ in $\Hom_{\ZZ_p[G]}(H^1(\cO_{F,S},T),\ZZ_p[G])$. Then the argument of Proposition \ref{link prop} implies the existence of a unique element ${}_{c,d}\kappa_F$ of $P_\QQ \otimes_{\ZZ_p}Q_F^a$ with the property that
$$\iota_F({}_{c,d}\kappa_F) = \cN_{F/\QQ}({}_{c,d}z_F)$$
in $P_\QQ \otimes_{\ZZ_p} Q_F^a$. We shall refer to ${}_{c,d}\kappa_F$ as the {\it Darmon derivative} of ${}_{c,d}z_F$ with respect to the embedding $j_F$. \end{definition}



\begin{remark} If one restricts the embeddings $j_F$ by requiring that the associated module $P_F$ has minimal possible rank, then the derivatives ${}_{c,d}\kappa_F$ can be checked to be independent, in a natural sense, of the choice of $j_F$.\end{remark} 

\subsubsection{}Conjecture \ref{mrs} predicts that the element ${}_{c,d}\kappa_F$ belongs to the image of the (injective) homomorphism 
\begin{equation}\label{injective a} H^1(\ZZ_S,T)\otimes_{\ZZ_p}Q_F^a \to P_\QQ\otimes _{\ZZ_p}Q_F^a\end{equation}
induced by $j_F$. At this stage, however, we can only verify this prediction in certain special cases.
 
In the next section we shall verify that it is valid if $F$ is contained in the cyclotomic $\ZZ_p$-extension $\QQ_\infty$ of $\QQ$. In the following result we record some evidence in the general case. 

Before stating the result we note that the hypothesis in its first paragraph is valid whenever the data $E, F, S$ and $p$ satisfy the conditions (a), (b), (c) and (d) of Proposition \ref{prop bss} and that, in general, its validity would follow from that of Conjecture \ref{order}. 

In particular, claim (ii) of this result is a natural analogue for zeta elements of one of the main results of Darmon in \cite[Th. 2.5]{DH} concerning Heegner points.  

\begin{theorem}\label{prop kappa} Set $z := {}_{c,d} z_F$ and $\kappa := {}_{c,d}\kappa_F$. If one has $\Phi(z)\in {\rm Fitt}_{\ZZ_p[G]}^0(H^2(\cO_{F,S},T))$ for every $\Phi$ in $\Hom_{\ZZ_p[G]}\bigl(H^1(\cO_{F,S},T),  \ZZ_p[G]\bigr)$, then the following claims are valid. 

\begin{itemize}
\item[(i)] If $p^N$ is the minimum of the exponents of the groups $\#H^2(\ZZ_S,T)_{\rm tors}\cdot Q_F^a$ and $H^2(\ZZ_S,T)_{\rm tors}$, then  $p^N\cdot\kappa$ belongs to the image of the map (\ref{injective a}).   
\item[(ii)]  The image of $\kappa$ under the natural map 
\[ P_\QQ\otimes_{\ZZ_p}Q_F^a \to P_\QQ\otimes_{\ZZ_p}Q_F^a \otimes_\ZZ \ZZ/(p)\]
belongs to the image of the map 
\[ H^1(\ZZ_S,T)\otimes_{\ZZ_p}Q_F^a \otimes_\ZZ \ZZ/(p)\to P_\QQ \otimes_{\ZZ_p}Q_F^a \otimes_\ZZ \ZZ/(p)\]
induced by (\ref{injective a}). 
\end{itemize}
\end{theorem}

\begin{proof} The proof of claim (i) requires a refinement of the construction used to prove Proposition \ref{link prop}. This relies on the fact that the complex
$$C_F:=\rhom_{\ZZ_p}(\rgamma_c(\cO_{F,S},T^\ast(1)),\ZZ_p[-2])$$
is a perfect complex of $\ZZ_p[G]$-modules that is acyclic outside degrees zero and one, and that there exists a canonical isomorphism
\begin{eqnarray}\label{exh1}
H^0(C_F) \simeq H^1(\cO_{F,S},T)
\end{eqnarray}
and an exact sequence
\begin{eqnarray}\label{exh2}
0 \to H^2(\cO_{F,S},T)  \to H^1(C_F) \to \ZZ_p[G]\otimes_{\ZZ_p} T^\ast(1)^{+,\ast} \to 0.
\end{eqnarray}
(See \cite[Prop. 2.22]{sbA}.) 

In particular, by \cite[Prop. A.11(i)]{sbA}, one finds that $C_F$ is represented by a complex of the form $P_F \to P_F,$ where $P_F$ is a finitely generated free $\ZZ_p[G]$-module and the first term is placed in degree zero.

In this way we obtain an exact sequence
\begin{eqnarray*} \label{tateseq}
0\to H^1(\cO_{F,S},T) \to P_F \xrightarrow{f_F} P_F \to H^1(C_F) \to 0,
\end{eqnarray*}

Then (\ref{control}) implies that $C_\QQ$ is represented by the complex $P_\QQ \xrightarrow{f_\QQ} P_\QQ$ obtained by taking $G$-invariants of the complex $P_F \xrightarrow{f_F} P_F$ and hence that there is an exact sequence
\begin{equation}\label{Q seq} 0 \to H^1(\ZZ_S,T) \to P_\QQ \xrightarrow{f_\QQ} P_\QQ \to H^1(C_\QQ) \to 0.\end{equation}

Set $\tilde Q_F^a := \#H^2(\ZZ_S,T)_{\rm tors}\cdot Q^a_F \subset \ZZ_p[G]/I_F^{a+1}$. Then the above sequences combine to give a commutative diagram
$$
\small
\xymatrix{
0 \ar[r]&H^1(\cO_{F,S},T) \otimes_{\ZZ_p} \ZZ_p[G]/I_F^{a+1}\ar[r] & P_F \otimes_{\ZZ_p} \ZZ_p[G]/I_F^{a+1} \ar[r]^{\tilde f_F}& P_F \otimes_{\ZZ_p} \ZZ_p[G]/I_F^{a+1}\\
0\ar[r] & H^1(\ZZ_S,T) \otimes_{\ZZ_p} \tilde Q_F^{a} \ar[r] \ar[u]_{\tilde\iota_F}& P_\QQ \otimes_{\ZZ_p} \tilde Q_F^{a} \ar[r]^{\tilde f_\QQ} \ar[u]_{\tilde\iota_F}& P_\QQ \otimes_{\ZZ_p} \tilde Q_F^{a}\ar[u]_{\tilde\iota_F}
}
$$
in which the maps $\tilde\iota_F$ are obtained by restricting $\iota_F$.

Then the argument of Proposition \ref{link prop} implies that 
\begin{equation}\label{kappa inclusion} \kappa\in P_\QQ \otimes_{\ZZ_p} \tilde Q_F^{a}\subset P_\QQ \otimes_{\ZZ_p}Q_F^a,\end{equation}
 and so the commutativity of this diagram implies that
\[  \tilde\iota_F(\tilde f_\QQ(\kappa)) = \tilde f_F(\tilde \iota_F(\kappa)) = \tilde f_F(\cN_{F/\QQ}(z)) = 0\]
and hence, since $\tilde \iota_F$ is injective, that $\tilde f_\QQ(\kappa)=0$.

Now, the exact sequence (\ref{Q seq}) induces exact sequences
\[ 0\to H^1(\ZZ_S,T)\otimes_{\ZZ_p} \tilde Q_F^{a} \to P_\QQ\otimes_{\ZZ_p} \tilde Q_F^{a} \xrightarrow{\mu_1} \im(f_\QQ)\otimes_{\ZZ_p} \tilde Q_F^{a} \to 0\]
and
\[ 0 \to {\rm Tor}_1^{\ZZ_p}\bigl(H^2(\ZZ_S,T)_{\rm tors},\tilde Q_F^{a}\bigr) \xrightarrow{\mu_2} \im(f_\QQ)\otimes_{\ZZ_p} \tilde Q_F^{a} \xrightarrow{\mu_3} P_\QQ\otimes_{\ZZ_p} \tilde Q_F^{a}.\]
with the property that $\mu_3\circ\mu_1$ is equal to $\tilde f_\QQ$. (The first sequence here is exact since the $\ZZ_p$-module $\im(f_\QQ)$ is free and the second is exact as consequence of the fact that (\ref{exh2}) identifies $H^2(\ZZ_S,T)_{\rm tors}$ with $H^1(C_\QQ)_{\rm tors}$.)

These sequences combine with the equality $\tilde f_\QQ(\kappa)=0$ to imply $\mu_1(\kappa)$ belongs to the image of $\mu_2$ in the lower sequence above.

Thus, since the definition of $p^N$ ensures it annihilates the group ${\rm Tor}_1^{\ZZ_p}\bigl(H^2(\ZZ_S,T)_{\rm tors},\tilde Q_F^{a}\bigr)$, it follows that $\mu_1(p^N\cdot \kappa)$ vanishes, and hence that $p^N\cdot \kappa$ belongs to $H^1(\ZZ_S,T)\otimes_{\ZZ_p} \tilde Q_F^{a} \subset
 H^1(\ZZ_S,T)\otimes_{\ZZ_p} Q_F^{a}$. This proves claim (i).


Turning to claim (ii), we note first that if $H^2(\ZZ_S,T)_{\rm tors}$ is trivial, then claim (i) implies $\kappa$ belongs to the image of the map  (\ref{injective a}) and so claim (ii) follows immediately. 
 
On the other hand, if $H^2(\ZZ_S,T)_{\rm tors}$ is non-trivial, then $\tilde Q_F^{a}$ is contained in $p\cdot Q_F^{a}$ and so (\ref{kappa inclusion}) implies that the projection of $\kappa$ to  $P_\QQ\otimes_{\ZZ_p}Q_F^a \otimes \ZZ/(p)$ vanishes. 
 In this case, therefore, the result of claim (ii) is also clear. 
\end{proof}  

\begin{remark} If $G$ has exponent $p$ and $r_{\rm alg}>0$, then $a >0$ and so  $Q_F^a$ is annihilated by $p$. In any such case, therefore, Theorem \ref{prop kappa}(ii) implies (under the stated hypotheses) that $\kappa$ belongs to the image of the map (\ref{injective a}). In general, the argument of Theorem \ref{prop kappa} shows that the group $H^2(\ZZ_S,T)_{\rm tors}$ constitutes the obstruction to attempts to deduce this containment from Euler system arguments (via the result of Proposition \ref{prop bss}). 
 To describe this obstruction more explicitly we assume that $E[p]$ is an irreducible $G_\QQ$-representation. In this case, one can assume $T=T_p(E)$ and then global duality gives rise to an exact sequence 
\[ E(\QQ)\otimes_\ZZ \ZZ_p \to \bigoplus_{\ell \in S \setminus \{\infty\}} \varprojlim_n E(\QQ_\ell)/p^n \to \bigl(H^2(\ZZ_S,T)_{\rm tors}\bigr)^\vee \to \sha(E/\QQ)[p^\infty] \to 0,\]
in which the first arrow denotes the natural diagonal map.  
\end{remark}


\subsection{Iwasawa-Darmon derivatives} \label{sec id}
To consider the above constructions in an Iwasawa-theoretic setting we shall use the following notations for non-negative integers $n$ and $i$:

\begin{itemize}
\item $\QQ_n$: the $n$-th layer of the cyclotomic $\ZZ_p$-extension $\QQ_\infty/\QQ$ (i.e., the subfield of $\QQ_\infty$ such that $[\QQ_n:\QQ]=p^n$),
\item $G_n:=\Gal(\QQ_n/\QQ)$,
\item $I_n:=\ker(\ZZ_p[G_n]\twoheadrightarrow \ZZ_p),$
\item $Q_n^a:=I_n^a/I_n^{a+1}$ with $a:=\max\{0, r_{\rm alg}-1\}$ as above,
\item $H^i_n:=H^i(\cO_{\QQ_n,S},T)$,
\item ${}_{c,d}z_n:={}_{c,d}z_{\QQ_n}$,
\item $\Gamma:=\Gal(\QQ_\infty/\QQ),$
\item $\Lambda:=\ZZ_p[[\Gamma]],$
\item $I:=\ker(\ZZ_p[[\Gamma]] \twoheadrightarrow \ZZ_p), $
\item $Q^a = I^{a}/I^{a+1}$,
\item $\HH^i:= \varprojlim_n H^i_n$.
\end{itemize}


\subsubsection{}\label{sec iw}


%
We first verify the prediction (\ref{order2}) in this setting.

\begin{proposition}\label{kato vanish}
For any non-negative integer $n$, the element ${}_{c,d}z_n$ belongs to $I_n^a \cdot H^1_n.$

In particular, the weak vanishing order prediction of (\ref{order2}) holds for the field $F=\QQ_n$ for every $n$.
\end{proposition}

\begin{proof}
We use Kato's result on the Iwasawa Main Conjecture \cite[\S 12]{katoasterisque}. By \cite[Th. 12.4(2)]{katoasterisque}, we know that $\QQ \otimes_{\ZZ} \HH^1$ is a free $\QQ\otimes_\ZZ \Lambda$-module of rank one. 
This together with the injectivity of $\HH^1/I \HH^1 \rightarrow H^1(\ZZ_S,T)$ 
and the assumption that $H^1(\ZZ_S,T)$ is $\ZZ_{p}$-free implies that 
$\HH^1$ is a free $\Lambda$-module of rank one.
Since 
$$\HH^2 \twoheadrightarrow H^2(\ZZ_S,T)_{\rm tf} \simeq \ZZ_p^a$$
is surjective, the characteristic ideal of $\HH^2$ is in $I^a$.
Therefore, the characteristic ideal of 
$\HH^1/\langle {}_{c,d}z_\infty\rangle$ is also in it by \cite[Th. 12.5(3)]{katoasterisque}, where 
${}_{c,d}z_\infty:=({}_{c,d}z_n)_n$ 
(note that $({}_{c,d}z_n)_n$ is in the inverse limit $\varprojlim_n H^1_n 
=\HH^1$). 
This shows that 
$${}_{c,d}z_\infty \in I^a \cdot \HH^1,$$
which implies the conlusion of Proposition \ref{kato vanish}. 
\end{proof}

By using Proposition \ref{kato vanish}, we can now explicitly construct the Darmon derivative of ${}_{c,d}z_n$. To do this we fix a topological generator $\gamma$ of $\Gamma$ and denote the image of $\gamma$ in $G_n$ by the same symbol. In view of Proposition \ref{kato vanish} one has
\begin{eqnarray*}\label{def w}
{}_{c,d}z_n=(\gamma-1)^a w_n
\end{eqnarray*}
for some choice of element $w_n$ of $H^1_n$.

We then compute
\begin{eqnarray*}\label{compute w}
\cN_{\QQ_n/\QQ}({}_{c,d}z_n)&=&\sum_{\sigma \in G_n} \sigma( {}_{c,d}z_n )\otimes \sigma^{-1} \\
&=&\sum_{\sigma \in G_n} \sigma (\gamma-1)^a w_n \otimes \sigma^{-1} \nonumber \\
&=& \sum_{\sigma \in G_n} \sigma w_n \otimes \sigma^{-1}(\gamma-1)^a \in H^1_n \otimes_{\ZZ_p} I_n^a.\nonumber
\end{eqnarray*}
Thus, in $H^1_n \otimes_{\ZZ_p}Q_n^a$, we have
$$\cN_{\QQ_n/\QQ}({}_{c,d}z_n)=\sum_{\sigma \in G_n}\sigma w_n \otimes (\gamma-1)^a .$$
Hence, the derivative in Definition \ref{def kappa} is explicitly given by
\begin{eqnarray}\label{kappa exp}
{}_{c,d}\kappa_n:= {\rm Cor}_{\QQ_n/\QQ}(w_n) \otimes (\gamma-1)^a \in H^1_0 \otimes_{\ZZ_p}Q_n^a.
\end{eqnarray}
One easily sees that this element is well-defined, i.e., independent of the choice of $w_n$. Furthermore, the collection $({}_{c,d}\kappa_n)_n$ is an inverse system, so we can give the following definition.
\begin{definition}\label{def iw}
We define the {\it Iwasawa-Darmon derivative} of Kato's Euler system by
$${}_{c,d}\kappa_\infty:=({}_{c,d}\kappa_n)_n \in \varprojlim_n H^1_0 \otimes_{\ZZ_p}Q_n^a =H^1(\ZZ_S,T) \otimes_{\ZZ_p}Q^a.$$
We also define the normalized version
$$\kappa_\infty:=\frac{1}{cd(c-1)(d-1)} \cdot {}_{c,d}\kappa_\infty \in H^1(\ZZ_S,V) \otimes_{\ZZ_p} Q^{a}.$$

\end{definition}

\begin{remark} The Iwasawa-Darmon derivative can be regarded as a natural analogue of the `cyclotomic $p$-units' that are defined by Solomon in  \cite{solomon} in the setting of the classical cyclotomic unit Euler system. In a more general setting, it is an analogue of the derivative $\kappa$ of the (conjectural) Rubin-Stark Euler system that occurs in \cite[Conj. 4.2]{bks2}.
\end{remark}

\begin{remark}\label{zq}
 If $r_{\rm alg}$ is at most one, then $a=0$, $Q^a=\ZZ_p$ and in $H^1(\ZZ_S,V)$ one has
$$\kappa_\infty=z_\QQ:=\frac{1}{cd(c-1)(d-1)} \cdot {}_{c,d}z_\QQ $$
so that Definition \ref{def iw} gives nothing new in this case.
\end{remark}

\subsection{The Generalized Perrin-Riou Conjecture at infinite level}

In this section we assume Hypothesis \ref{hyp} in order to state an Iwasawa-theoretic version of Conjecture \ref{mrs}. We set $r:=r_{\rm alg}$.

\subsubsection{}To do this we fix a $\ZZ_p$-basis $\bm{x}$ of ${\bigwedge}_{\ZZ_p}^{r-1} H^2(\ZZ_S,T)_{\rm tf}$ and write
$${\rm Boc}_{n,\bm{x}}={\rm Boc}_{\QQ_n,\bm{x}}: {\bigwedge}_{\ZZ_p}^{r} H^1(\ZZ_S,T) \to H^1(\ZZ_S, T)\otimes_{\ZZ_p} I_n^{r-1}/I_n^{r}$$
for the Bockstein regulator map (\ref{boc}) for the field $\QQ_n$, as defined in \S \ref{sec boc}.

As $n$ varies these maps combine to induce a homomorphism
$$\varprojlim_n {\rm Boc}_{n,\bm{x}}: {\bigwedge}_{\ZZ_p}^{r} H^1(\ZZ_S,T) \to H^1(\ZZ_S,T)\otimes_{\ZZ_p}\varprojlim_n I_n^{r-1}/I_n^{r}= H^1(\ZZ_S,T)\otimes_{\ZZ_p} Q^{r-1}$$
and hence also, by scalar extension, a homomorphism
\begin{eqnarray}\label{boc2}
{\rm Boc}_{\infty,\bm{x}}: \CC_p \otimes_{\ZZ_p} {\bigwedge}_{\ZZ_p}^r H^1(\ZZ_S,T) \to \CC_p \otimes_{\ZZ_p} H^1(\ZZ_S,T) \otimes_{\ZZ_p} Q^{r-1}.
\end{eqnarray}

We recall from Definition \ref{def eta} the Birch and Swinnerton-Dyer element $\eta_{\bm{x}}^{\rm BSD}$ that is constructed (unconditionally) in the  space $\CC_p \otimes_{\ZZ_p} {\bigwedge}_{\ZZ_p}^r H^1(\ZZ_S,T)$.

\begin{conjecture}\label{mrs2} One has 
$$\kappa_\infty ={\rm Boc}_{\infty,\bm{x}}(\eta_{\bm{x}}^{\rm BSD})$$
in $\CC_p \otimes_{\ZZ_p}H^1(\ZZ_S,T) \otimes_{\ZZ_p}Q^{r-1}$.
\end{conjecture}

\begin{remark} In contrast to the more general situation considered in Conjecture \ref{mrs} we do not here need to assume $\eta_{\bm{x}}^{\rm BSD}$ belongs to ${\bigwedge}_{\ZZ_p}^r H^1(\ZZ_S,T)$. This is because the group $Q^{r-1}$ is $\ZZ_p$-torsion-free and so one does not lose anything by defining the Bockstein homomorphism ${\rm Boc}_{\infty,\bm{x}}$ on $\CC_p$-modules. In particular, if $r=1$, then the discussion of Remark \ref{PR equiv} shows that Conjecture \ref{mrs2} is equivalent to Perrin-Riou's original conjecture. Finally, we observe that Conjecture \ref{mrs2} is a natural analogue for elliptic curves of the conjecture formulated for the multiplicative group in \cite[Conj. 4.2]{bks2}.
\end{remark}


\subsubsection{}\label{exp int}

We shall now give an explicit interpretation of Conjecture \ref{mrs2} in terms of the leading term $L_S^\ast(E,1)$ (see Proposition \ref{prop exp} below).

Take a basis $\{x_1,\ldots,x_r\}$ of $E(\QQ)_{\rm tf}$. We define an element $\bm{x} \in {\bigwedge}_{\QQ_p}^{r-1} H^2(\ZZ_S,V)$ as the element corresponding to
$$1 \otimes x_1 \otimes (x_1^\ast \wedge \cdots \wedge x_r^\ast) \in \QQ_p \otimes_{\ZZ_p} \left( E_1(\QQ_p)\otimes_\ZZ {\bigwedge}_\ZZ^r E(\QQ)^\ast \right)$$
under the isomorphism
$${\bigwedge}_{\QQ_p}^{r-1} H^2(\ZZ_S,V) \simeq \QQ_p \otimes_{\ZZ_p} \left( E_1(\QQ_p)\otimes_\ZZ {\bigwedge}_\ZZ^r E(\QQ)^\ast \right)$$
induced by (\ref{short}).
We note that, by linearity, the definition of the Bockstein regulator map (\ref{boc2}) is extended for any element in ${\bigwedge}_{\QQ_p}^{r-1} H^2(\ZZ_S,V)$, which is not necessarily a $\ZZ_p$-basis of ${\bigwedge}_{\ZZ_p}^{r-1} H^2(\ZZ_S,T)_{\rm tf}$. Thus ${\rm Boc}_{\infty,\bm{x}}$ is defined for above $\bm{x}$.

Let $\omega$ be the fixed N\'eron differential and $\log_\omega: E(\QQ ) \to E(\QQ_p) \to \QQ_p$ the formal logarithm associated to $\omega$. We give the following definition.

\begin{definition}\label{def alg}
We define the {\it Bockstein regulator} associated to $\omega$ by setting 
$$R_\omega^{\rm Boc}:=\log_\omega(x_1)\cdot {\rm Boc}_{\infty,\bm{x}}(x_1\wedge \cdots \wedge x_r) \in (\QQ_p \otimes_{\ZZ} E(\QQ))\otimes_{\ZZ_p}Q^{r-1}.$$
(Here we identify $H^1(\ZZ_S,V)=\QQ_p \otimes_\ZZ E(\QQ)$ by (\ref{h1}).)
One can check that this does not depend on the choice of the basis $\{x_1,\ldots,x_r\}$ of $E(\QQ)_{\rm tf}$.
\end{definition}

\begin{remark}
 The Bockstein regulator defined above is closely related to classical $p$-adic regulators: for details, see  Theorems \ref{reg prop} and \ref{reg prop 2} below.
\end{remark}

\begin{remark}\label{remark boc1}
When $r=1$, then ${\rm Boc}_{\infty,\bm{x}}$ is the identity map and one has 
$$R_\omega^{\rm Boc}=\log_\omega(x)\cdot x \in \QQ_p \otimes_\ZZ E(\QQ)$$
for any generator $x$ of $E(\QQ)_{\rm tf}$.
\end{remark}

\begin{remark}\label{rem alg}
Let $\Omega_\xi$ be as in (\ref{per xi}) and $R_\infty$ the N\'eron-Tate regulator. Then one can check that
$${\rm Boc}_{\infty,\bm{x}}(\eta_{\bm{x}}^{\rm BSD})=\frac{L_S^\ast(E,1)}{\Omega_\xi\cdot R_\infty} \cdot R_\omega^{\rm Boc}.$$
In fact, by the definition of the Birch and Swinnerton-Dyer element, one checks that
\begin{eqnarray}\label{element formula}
\eta_{\bm{x}}^{\rm BSD}=\frac{L_S^\ast(E,1)}{\Omega_\xi \cdot R_\infty}\cdot \log_\omega(x_1)\cdot x_1\wedge \cdots \wedge x_r.
\end{eqnarray}
\end{remark}

By Remark \ref{rem alg}, we obtain the following interpretation of Conjecture \ref{mrs2}.

\begin{proposition}\label{prop exp}
Conjecture \ref{mrs2} is valid if and only if one has
$$\kappa_\infty=\frac{L_S^\ast(E,1)}{\Omega_\xi\cdot R_\infty} \cdot R_\omega^{\rm Boc}$$
in $\CC_p \otimes_{\ZZ_p} H^1(\ZZ_S,T)\otimes_{\ZZ_p}Q^{r-1} \simeq (\CC_p\otimes_\ZZ E(\QQ ))\otimes_{\ZZ_p}Q^{r-1}.$
\end{proposition}

\subsubsection{}Using Proposition \ref{prop exp} we state an Iwasawa-theoretic version of the `algebraic' variant Conjecture \ref{algebraic mrs} of Conjecture \ref{mrs}. This conjecture is therefore a natural `algebraic' variant of Conjecture \ref{mrs2}.

We recall that $L_\ell$ denotes the Euler factor at a prime $\ell$ so that one has
\[ \left(\prod_{\ell \in S \setminus \{\infty\}}L_\ell\right)\cdot L^\ast(E,1)=L_S^\ast(E,1).\]

We also write $v_\xi$ for the non-zero rational number that is defined by the equality  
\begin{eqnarray} \label{omega}
\Omega^+=v_\xi \cdot \Omega_\xi
\end{eqnarray}
where $\Omega^+$ is the real N\'{e}ron period that occurs in (\ref{BSD eq}).

\begin{conjecture}\label{mrs3} If $\sha(E/\QQ)$ is finite, then in $(\QQ_p \otimes_\ZZ E(\QQ))\otimes_{\ZZ_p} Q^{r-1}$ one has 
$$\kappa_\infty =v_\xi\left(\prod_{\ell \in S \setminus \{\infty\}} L_\ell \right)\frac{\# \sha(E/\QQ)\cdot {\rm Tam}(E)}{\# E(\QQ)_{\rm tors}^2} \cdot R_\omega^{\rm Boc}. $$
\end{conjecture}

\begin{remark}
One checks easily that Conjecture \ref{mrs3} is equivalent to an equality 
$$\kappa_\infty={\rm Boc}_{\infty,\bm{x}}(\eta_{\bm{x}}^{\rm alg}),$$
where $\bm{x}$ is any non-zero element of ${\bigwedge}_{\QQ_p}^{r-1} H^2(\ZZ_S,V)$ and $\eta_{\bm{x}}^{\rm alg}$ is the algebraic Birch and Swinnerton-Dyer element that is defined (unconditionally) in Definition \ref{alg bsd def}.
\end{remark}

\begin{remark} In Corollary \ref{cor padic} below we will show that Conjecture \ref{mrs3} is a refinement of the $p$-adic Birch-Swinnerton-Dyer Formula (from  \cite[Chap. II, \S10]{MTT}). Similarly, in Corollary \ref{cor beilinson} we will show that Conjecture \ref{mrs2} leads to an explicit formula for the leading term of the $p$-adic $L$-function (which we will refer to as a `$p$-adic Beilinson Formula'). 

A key advantage of the formulations of Conjectures \ref{mrs2} and \ref{mrs3} is that they do not involve the $p$-adic $L$-function and so are not in principle dependent on the precise reduction type of $E$ at $p$. In particular, the conjectures make sense (and are canonical) even when $E$ has additive reduction at $p$.
\end{remark}

\section{$p$-adic height pairings and the Bockstein regulator}\label{sec height}

 In this section, as an important preliminary to the proofs of Theorem \ref{theorem1i} and Corollaries \ref{cor1} and \ref{cor iw1i}, we shall make an explicit comparison of the Bockstein regulator $R_\omega^{\rm Boc}$ defined in Definition \ref{def alg} with the various notions of classical $p$-adic regulator (see Theorems \ref{reg prop} and \ref{reg prop 2} below).


In the following, we say `$p$ is $-$' if $E$ has $-$ reduction at $p$. For example, `$p$ is good ordinary' means that $E$ has good ordinary reduction at $p$.

{\it In this section, we assume that $E$ does not have additive reduction at $p$.}


We shall use the same notations as in \S\ref{sec 2} and \S\ref{sec zp}.

\subsection{Review of $p$-adic height pairings}\label{review height}

In this section, we give a review of the construction of $p$-adic height pairing using Selmer complexes.

\subsubsection{The ordinary case}\label{pordinary}
Suppose first that $p$ is ordinary, i.e., good ordinary or multiplicative. In this case we follow Nekov\'a\v r's construction of a $p$-adic height pairing in \cite[\S11]{nekovar}. (It is possible to treat this case in a more general context in \S \ref{psuper} below, but it requires the theory of $(\varphi,\Gamma)$-modules.)

We recall the definition of Nekov\'a\v r's Selmer complex.

To do this we note that, since $p$ is ordinary, we have a canonical filtration $F^+ V \subset V$ of $G_{\QQ_p}$-modules (due to Greenberg, see \cite{greenberg}).

We set $F^+ T:= T \cap F^+V$. For any non-negative integer $n$, we also denote the unique $p$-adic place of $\QQ_n$ by $\frp$.

Then, following the exact triangle given in (the third row of) \cite[(6.1.3.2)]{nekovar}, we define the Selmer complex of $T$ by setting
$$\widetilde {\rgamma}_f(\QQ_n, T):= {\rm Cone} \left( \rgamma(\cO_{\QQ_n,S},T) \to \rgamma(\QQ_{n,\frp},T/F^+T)\oplus \bigoplus_{v \in S_{\QQ_n}\setminus \{\frp\}} \rgamma_{/f}(\QQ_{n,v},T)\right)[-1].$$

We set
\[ \widetilde H^i_f(\QQ_n,T):=H^i(\widetilde {\rgamma}_f(\QQ_n,T))\,\,\text{ and }\,\, \widetilde H^i_f(\QQ_n,V):=\QQ_p\otimes_{\ZZ_p} \widetilde H^i_f(\QQ_n,T).\]

We have a natural isomorphism
$$\widetilde {\rgamma}_f(\QQ_n,T) \otimes_{\ZZ_p[G_n]}^{\DL} \ZZ_p \simeq \widetilde {\rgamma}_f(\QQ,T)$$
(see \cite[Prop. 8.10.1]{nekovar} or \cite[Prop. 1.6.5(3)]{FK}), and so we can define $(-1)$-times the Bockstein map
$$\widetilde H^1_f(\QQ,T) \to \widetilde H_f^2(\QQ,T) \otimes_{\ZZ_p} I_n/I_n^2$$
associated to the complex $\widetilde {\rgamma}_f(\QQ_n, T)$ (in the same way as (\ref{def beta})). Taking $\varprojlim_n$ and $\QQ_p\otimes_{\ZZ_p} -$, we obtain a map
\begin{eqnarray}\label{beta tilde}
\widetilde \beta:\widetilde H^1_f(\QQ,V) \to \widetilde H_f^2(\QQ,V) \otimes_{\ZZ_p} I/I^2.
\end{eqnarray}
Combining this map with the global duality map
$$\widetilde H^2_f(\QQ,V) \to \widetilde H^1_f(\QQ,V)^\ast$$
(see \cite[\S 6.3]{nekovar}),
we obtain a pairing
$$\langle -,- \rangle_p: \widetilde H^1_f(\QQ,V) \times \widetilde H^1_f(\QQ,V) \to \QQ_p \otimes_{\ZZ_p}I/I^2.$$
Noting that there is a natural embedding $\QQ_p \otimes_\ZZ E(\QQ) \hookrightarrow \widetilde H^1_f(\QQ,V)$ (see Remark \ref{rem selmer} below), we obtain the $p$-adic height pairing
$$\langle -,- \rangle_p: E(\QQ)\times E(\QQ) \to \QQ_p \otimes_{\ZZ_p}I/I^2.$$

\begin{remark}\label{rem selmer}
If $p$ is good ordinary or non-split multiplicative, then $\widetilde H^1_f(\QQ,V)$ coincides with the usual Selmer group $H^1_f(\QQ,V) $ (see \cite[\S 0.10]{nekovar}). If $p$ is split multiplicative, then we have a canonical decomposition
$$\widetilde H^1_f(\QQ,V) \simeq H^1_f(\QQ,V) \oplus \QQ_p$$
(see \cite[\S 11.4.2]{nekovar}). In any case, we have a canonical embedding $\QQ_p \otimes_\ZZ E(\QQ) \hookrightarrow \widetilde H^1_f(\QQ,V)$.
\end{remark}

\begin{remark}
For comparisons of the above $p$-adic height pairing with the classical ones, see \cite[\S\S 11.3 and 11.4]{nekovar}.

\end{remark}

\subsubsection{The supersingular case}\label{psuper}
Suppose that $p$ is good supersingular. In this case we follow the construction of the $p$-adic height pairing due to Benois \cite{benois}. His construction uses Selmer complexes associated to $(\varphi,\Gamma)$-modules, which was studied by Pottharst \cite{pot}. See also the review in \cite{BB}.


We fix one of the roots $\alpha \in \overline \QQ_p$ of the polynomial $X^2-a_p X +p$.
We set
$$L:=\QQ_p(\alpha).$$
We also set
$$V_L:=L\otimes_{\QQ_p} V \text{ and }D_L:=
D_{\rm crys}(V_L) =
D_{\rm dR}(V_L)\simeq L \otimes_\QQ H^1_{\rm dR}(E/\QQ),$$
which is endowed with an action of the Frobenius operator $\varphi$ and also a natural decreasing filtration $\{ D_L^i \}_{i \in \ZZ}$ such that $D_L^0 \simeq L \otimes_\QQ \Gamma(E,\Omega_{E/\QQ}^1)$. We set
$$t_{V,L}:=D_L/D_L^0 \simeq L \otimes_\QQ {\rm Lie}(E).$$
Let $N_\alpha$ be the subspace of $D_L$ on which $\varphi$ acts via $\alpha p^{-1}$. Explicitly, $N_\alpha$ is the subspace generated by $\varphi (\omega) -\alpha^{-1}\omega \in D_L$. Then the natural projection $D_L \twoheadrightarrow D_L/D_L^0=t_{V,L}$ induces an isomorphism
\begin{eqnarray}\label{split}
N_\alpha \xrightarrow{\sim} t_{V,L}.
\end{eqnarray}
A subspace of $D_L$ with this property is called a `splitting submodule' in \cite[\S 4.1.1]{benois}.


We shall define a $p$-adic height pairing
$$\langle -,- \rangle_{p}=\langle -,-\rangle_{p,\alpha} : E(\QQ) \times E(\QQ) \to L\otimes_{\ZZ_p } I/I^2.$$
Since there is a natural embedding $\QQ_p\otimes_\ZZ E(\QQ)\hookrightarrow H^1_f(\QQ,V)$, it is sufficient to construct a pairing
$$\langle -,- \rangle_{p} : H^1_f(\QQ,V) \times H^1_f(\QQ,V) \to L \otimes_{\ZZ_p } I/I^2.$$

We recall some basic facts from the theory of $(\varphi,\Gamma)$-modules.
Let $\drig(V_L)$ denote the $(\varphi,\Gamma_{\QQ_p})$-module associated $V_L$ (where $\Gamma_{\QQ_p}:=\Gal(\QQ_p(\mu_{p^\infty})/\QQ_p)$). (See \cite[Th. 2.1.3]{benois}.) By \cite[Th. 2.2.3]{benois}, there is a submodule $\DD_\alpha \subset \drig(V_L)$ corresponding to $N_\alpha \subset D_L$. For a general $(\varphi,\Gamma_{\QQ_p})$-module $\DD$, one can define a complex (the `Fontaine-Herr complex')
$$\rgamma(\QQ_p,\DD),$$
which is denoted by $C_{\varphi,\gamma_{\QQ_p}}^\bullet (\DD)$ in \cite[\S 2.4]{benois}.
When $\DD=\drig(V_L)$, this is naturally quasi-isomorphic to $\rgamma(\QQ_p,V_L)$ (see \cite[Prop. 2.5.2]{benois}). So there is a natural morphism in the derived category of $L$-vector spaces
$$\rgamma(\ZZ_S,V_L) \to \rgamma(\QQ_p,V_L) \simeq \rgamma(\QQ_p, \drig(V_L)) \to \rgamma(\QQ_p,\drig(V_L)/\DD_\alpha).$$
We define the Selmer complex by
$$\widetilde {\rgamma}_{f}(\QQ,V_L):={\rm Cone}\left(\rgamma(\ZZ_S,V_L) \to \rgamma(\QQ_p,\drig(V_L)/\DD_\alpha) \oplus \bigoplus_{\ell \in S\setminus \{p\}} \rgamma_{/f}(\QQ_\ell,V_L) \right)[-1].$$
(We adopt \cite[(2.6)]{BB} as the definition.) We set $\widetilde H^i_f(\QQ,V_L):=H^i(\widetilde {\rgamma}_f(\QQ,V_L))$.
It is known that
$$H^1_f(\QQ_,V_L)\simeq \widetilde H^1_f(\QQ,V_L).$$
(See \cite[Th. III]{benois}.)

We next study the Iwasawa theoretic version.
We set
$$\cH:=\left\{f(X)=\sum_{n=0}^\infty c_n X^n \in L[[X]] \ \middle| \  \text{$f(X)$ converges on the open unit disk }\right\}.$$
Then, for a general $(\varphi,\Gamma_{\QQ_p})$-module $\DD$, one can define an Iwasawa cohomology complex of $\cH$-modules
$$\rgamma_{\rm Iw}(\QQ_p,\DD).$$
(See \cite[\S 2.8]{benois}.)
We fix a topological generator $\gamma \in \Gamma$. Then $\Gamma$ acts on $\cH$ by identifying $X=\gamma-1$. We set
$$\overline V_L:=V_L \otimes_L \cH,$$
where $G_\QQ$ acts on $\cH$ via
$$G_\QQ \twoheadrightarrow \Gamma \xrightarrow{\gamma \mapsto \gamma^{-1}} \Gamma.$$
When $\DD=\drig(V_L)$, we have a natural quasi-isomorphism $\rgamma_{\rm Iw}(\QQ_p, \DD) \simeq \rgamma(\QQ_p, \overline V_L)$ (see \cite[Th. 2.8.2]{benois}). Thus there is a natural morphism in the derived category of $\cH$-modules
$$\rgamma(\ZZ_S,\overline V_L) \to \rgamma(\QQ_p, \overline V_L) \simeq \rgamma_{\rm Iw}(\QQ_p, \drig(V_L)) \to \rgamma_{\rm Iw}(\QQ_p, \drig(V_L)/\DD_\alpha).$$
We define the Iwasawa Selmer complex by
$$\widetilde {\rgamma}_{f,{\rm Iw}}(\QQ, V_L):= {\rm Cone}\left(  \rgamma(\ZZ_S,\overline V_L) \to \rgamma_{\rm Iw}(\QQ_p, \drig(V_L)/\DD_\alpha) \oplus \bigoplus_{\ell \in S \setminus \{p\}} \rgamma_{/f}(\QQ_\ell,\overline V_L)\right)[-1].$$
We know the following `control theorem'
\begin{eqnarray}\label{control2}
\widetilde {\rgamma}_{f,{\rm Iw}}(\QQ,V_L) \otimes_{\cH}^{\DL} L \simeq \widetilde {\rgamma}_{f}(\QQ,V_L).
\end{eqnarray}
(See \cite[Th. 1.12]{pot}.)

We now give the definition of the $p$-adic height pairing. Let $\cI:=(X)$ be the augmentation ideal of $\cH$. Note that $\cI/\cI^2 $ is identified with $L \otimes_{\ZZ_p} I/I^2$. From the exact sequence
$$0\to \cI/\cI^2 \to \cH/\cI^2 \to L \to 0,$$
we obtain the exact triangle
$$\widetilde {\rgamma}_{f,{\rm Iw}}(\QQ,V_L) \otimes_{\cH}^{\DL} \cI/\cI^2 \to \widetilde {\rgamma}_{f,{\rm Iw}}(\QQ,V_L) \otimes_{\cH}^{\DL} \cH/\cI^2 \to \widetilde {\rgamma}_{f,{\rm Iw}}(\QQ,V_L) \otimes_{\cH}^{\DL} L.$$
By the control theorem (\ref{control2}), we have
$$\widetilde {\rgamma}_f(\QQ,V_L) \lotimes_L \cI/\cI^2 \to \widetilde {\rgamma}_{f,{\rm Iw}}(\QQ,V_L) \otimes_{\cH}^{\DL} \cH/\cI^2 \to \widetilde {\rgamma}_f(\QQ,V_L).$$
The $(-1)$-times connecting homomorphism of this triangle gives a map
$$\widetilde H^1_f(\QQ,V_L) \to H^2(\widetilde {\rgamma}_f(\QQ,V_L) \lotimes_L \cI/\cI^2 ) = \widetilde H_f^2(\QQ,V_L) \otimes_L \cI/\cI^2.$$
Composing this map with the global duality map
$$\widetilde H^2_f(\QQ,V_L) \to \widetilde H^1_f(\QQ,V_L)^\ast$$
(see \cite[Th. 3.1.5]{benois}), we obtain
$$\widetilde H^1_f(\QQ,V_L) \to \widetilde H^1_f(\QQ,V_L)^\ast \otimes_L \cI/\cI^2.$$
This gives the desired $p$-adic height pairing.

\begin{remark}
The above construction makes sense even when $p$ is good ordinary. In this case, $\alpha$ is canonically chosen so that $\ord_p(\alpha)<1$, and we can take $N_\alpha$ to be $D_{\rm crys}(F^+V)$. One sees that the $p$-adic height pairing with this choice coincides with that in \S \ref{pordinary}.

\end{remark}

\begin{remark}
Comparisons of this $p$-adic height pairing with the classical ones are studied in detail by Benois \cite{benois}. In particular, this $p$-adic height pairing coincides with the one constructed by Nekov\'a\v r in \cite{nekp}, which is used by Kobayashi in \cite{kobss}.
\end{remark}

\subsection{A comparison result}\label{compare}

We shall define the $p$-adic regulator and compare it with the  Bockstein regulator $R_\omega^{\rm Boc}$. In this subsection, we assume Hypothesis \ref{hyp}.

\subsubsection{}Let $L$ be the splitting field of the polynomial $X^2-a_pX+p$ over $\QQ_p$. Note that $L=\QQ_p$ unless $p$ is supersingular.

Let
$$\langle -,-\rangle_{p}: E(\QQ) \times E(\QQ) \to L \otimes_{\ZZ_p} I/I^2$$
be the $p$-adic height pairing defined above. (When $p$ is supersingular, this depends on the choice of a root $\alpha$ of $X^2-a_pX +p$.)

\begin{definition}\label{def preg}
The $p$-adic regulator
$$R_{p}=R_{p,\alpha} \in L \otimes_{\ZZ_p} Q^r$$
is defined to be the discriminant of the $p$-adic height pairing, i.e.,
$$R_p:=\det(\langle x_i,x_j\rangle_p)_{1\leq i ,j\leq r}$$
with $\{x_1,\ldots,x_r\}$ a basis of $E(\QQ)_{\rm tf}$.
\end{definition}

The $p$-adic height pairing induces a map
\begin{eqnarray}\label{pad2} E(\QQ) \times (\QQ_p \otimes_\ZZ E(\QQ))\otimes_{\ZZ_p} Q^{r-1} \to L\otimes_{\ZZ_p} Q^r
\end{eqnarray}
$$ (x,(a\otimes y)\otimes b) \mapsto a \cdot b\cdot  \langle x,y \rangle_{p}, $$
which we denote also by $\langle -,-\rangle_p$.

The following gives a relation between $R_p$ and $R_\omega^{\rm Boc}$.

\begin{theorem}\label{reg prop}
For any $x \in E(\QQ)$ we have
$$\langle x, R_\omega^{\rm Boc}\rangle_p=\log_\omega(x) \cdot R_p.$$
\end{theorem}

\subsubsection{}The proof of Theorem \ref{reg prop} will be given in \S\ref{proof of reg prop}. However, we first need to prove several preliminary technical results.

\begin{lemma}\label{sym}
The $p$-adic height pairing is symmetric, i.e.,
$$\langle x,y \rangle_p =\langle y,x\rangle_p$$
for any $x,y \in E(\QQ)$.
\end{lemma}

\begin{proof}
See \cite[Cor. 11.2.2]{nekovar} and \cite[Th. I]{benois} in the ordinary and supersingular cases respectively.
\end{proof}

\begin{lemma}\label{comm}
The following diagram is commutative.
$$
\xymatrix{
E(\QQ) \ar[r] \ar[rd]_{\varprojlim_n \beta_n}& (L \otimes_\ZZ E(\QQ))^\ast \otimes_{\ZZ_p} I/I^2 \ar@{->>}[d]^{(\ref{short})} \\
 & L\otimes_{\QQ_p }H^2(\ZZ_S,V) \otimes_{\ZZ_p} I/I^2,
}
$$
where the horizontal arrow is the map induced by the $p$-adic height pairing
$$x \mapsto (y \mapsto \langle x,y\rangle_p).$$
(For the definition of $\beta_n:=\beta_{\QQ_n}$, see (\ref{def beta}).)
\end{lemma}

\begin{proof}
We first suppose that $p$ is ordinary. We have the commutative diagram
$$\xymatrix{
\widetilde \rgamma_f(\QQ,T)\otimes_{\ZZ_p}^{\DL}I_n/I_n^2 \ar[r] \ar[d] &\widetilde \rgamma_f(\QQ_n,T)\otimes^{\DL}_{\ZZ_p[G_n]} \ZZ_p[G_n]/I_n^2  \ar[r] \ar[d]& \widetilde \rgamma_f(\QQ,T) \ar[d] \\
\rgamma(\ZZ_S,T)\otimes_{\ZZ_p}^{\DL}I_n/I_n^2 \ar[r] & \rgamma(\cO_{\QQ_n,S},T)\otimes^{\DL}_{\ZZ_p[G_n]} \ZZ_p[G_n]/I_n^2  \ar[r] & \rgamma(\ZZ_S,T),
}$$
whose rows are exact triangles.
The map $\beta_n$ is defined by the connecting homomorphism of the bottom triangle. On the other hand, the $p$-adic height pairing is defined by the connecting homomorphism of the top triangle. Thus the claim follows from the functoriality of the connecting homomorphism, i.e., the commutativity of the diagram
$$\xymatrix{
\widetilde H^1_f(\QQ,T) \ar[r] \ar[d]& \widetilde H^2_f(\QQ,T) \otimes_{\ZZ_p} I_n/I_n^2 \ar[d] \\
H^1(\ZZ_S,T) \ar[r] &H^2(\ZZ_S,T) \otimes_{\ZZ_p}I_n/I_n^2,
}$$
where the horizontal arrows are connecting homomorphisms.

Next, suppose that $p$ is good supersingular. With the notations in \S \ref{psuper}, we have the commutative diagram with exact rows
$$\xymatrix{
\widetilde \rgamma_f(\QQ,V_L)\otimes_{L}^{\DL}\cI/\cI^2 \ar[r] \ar[d] &\widetilde \rgamma_{f,{\rm Iw}}(\QQ,V_L)\otimes^{\DL}_{\cH} \cH/\cI^2  \ar[r] \ar[d]& \widetilde \rgamma_f(\QQ,V_L) \ar[d] \\
\rgamma(\ZZ_S,V_L)\otimes_{L}^{\DL}\cI/\cI^2 \ar[r] & \rgamma(\ZZ_S,
\overline V_L)\otimes^{\DL}_{\cH} \cH/\cI^2  \ar[r] & \rgamma(\ZZ_S,V_L).
}$$
Since the map $\varprojlim_n \beta_n$ coincides with the map defined by the connecting homomorphism of the bottom triangle (by Shapiro's lemma), the claim follows by the same argument as in the ordinary case.\end{proof}

\begin{lemma}\label{alg}
Let $M$ and $N$ be $L$-vector spaces of dimension $r$ and $r-1$ respectively. Suppose that an exact sequence
\begin{eqnarray}\label{nml}
0 \to N \xrightarrow{\iota} M \xrightarrow{\ell} L \to 0
\end{eqnarray}
and $L$-linear maps $f: M \to M^\ast$ and $g: M \to N^\ast$ are given. Assume the following.
\begin{itemize}
\item[(a)] The diagram
$$
\xymatrix{
M \ar[r]^f \ar[rd]_g& M^\ast \ar@{->>}[d]^{\iota^\ast} \\
 & N^\ast
}
$$
is commutative.
\item[(b)] The map $f$ satisfies $f(x)(y)=f(y)(x)$ for any $x,y \in M$.
\end{itemize}
Then for any $x \in M$ the following diagram is commutative.
\begin{eqnarray}\label{bigcomm}
\xymatrix{
{\bigwedge}_L^r M \ar[r]^{\bigwedge^r f} \ar[dd]_{\bigwedge^{r-1} g}& {\bigwedge}_L^r M^\ast \ar[rd]^{\ell(x)\times} \\
 & & {\bigwedge}_L^r M^\ast \\
M \otimes_L {\bigwedge}_L^{r-1} N^\ast \ar[r]_{\delta}^{\simeq} & M \otimes_L {\bigwedge}_L^r M^\ast \ar[ru]_{f(x)\otimes \id}.
}
\end{eqnarray}
Here $\delta$ is the natural isomorphism induced by (\ref{nml}), and the left vertical arrow is defined by
$$\left({\bigwedge}^{r-1}g \right)(x_1\wedge\cdots \wedge x_r)=\sum_{i=1}^r (-1)^{i+1}x_i \otimes g(x_1)\wedge \cdots \wedge g(x_{i-1}) \wedge g(x_{i+1})\wedge \cdots \wedge g(x_r).$$

\end{lemma}

\begin{proof}


Let $\{x_1,\ldots,x_r\}$ be a basis of $M$ and fix $x \in M$. It is sufficient to prove
$$f(x)\circ \delta\circ \left({\bigwedge}^{r-1}g \right)(x_1\wedge\cdots \wedge x_r) =\ell(x)\cdot f(x_1)\wedge\cdots \wedge f(x_r).$$

We shall describe the left hand side explicitly. 
Using assumption (a), we have
\begin{eqnarray}\label{delta exp}
&&\delta\circ \left({\bigwedge}^{r-1}g \right)(x_1\wedge\cdots \wedge x_r) \\
&=& \sum_{i=1}^r x_i\otimes f(x_1) \wedge \cdots \wedge f(x_{i-1})\wedge \ell \wedge f(x_{i+1})\wedge\cdots \wedge f(x_r). \nonumber
\end{eqnarray}
Thus we have
$$f(x)\circ \delta\circ \left({\bigwedge}^{r-1}g \right)(x_1\wedge\cdots \wedge x_r)  = \sum_{i=1}^r f(x)(x_i) \cdot f(x_1) \wedge \cdots \wedge f(x_{i-1})\wedge \ell \wedge f(x_{i+1})\wedge\cdots \wedge f(x_r). $$

Suppose first that $f$ is bijective. Then $\{f(x_1),\ldots,f(x_r)\}$ is a basis of $M^\ast$ and we can write
$$\ell = \sum_{i=1}^r a_i f(x_i) \text{ in }M^\ast$$
with some $a_1,\ldots,a_r \in L$.
By assumption (b), we have $f(x)(x_i)=f(x_i)(x)$ and so we compute
\begin{eqnarray*}
&&\sum_{i=1}^r f(x)(x_i) \cdot f(x_1) \wedge \cdots \wedge f(x_{i-1})\wedge \ell \wedge f(x_{i+1})\wedge\cdots \wedge f(x_r)\\
&=& \sum_{i=1}^r a_i f(x_i)(x) \cdot f(x_1)\wedge\cdots \wedge f(x_r)\\
&=& \ell(x)\cdot f(x_1)\wedge\cdots \wedge f(x_r).
\end{eqnarray*}
This proves the lemma in this case.

Suppose next that $f$ is not bijective. Then $\{f(x_1),\ldots,f(x_r)\}$ is linearly dependent so we may assume
$$f(x_1)=\sum_{i=2}^r a_i f(x_i)$$
with some $a_2,\ldots,a_r \in L$. 

We then compute

\begin{eqnarray*}
&&\sum_{i=1}^r f(x)(x_i) \cdot f(x_1) \wedge \cdots \wedge f(x_{i-1})\wedge \ell \wedge f(x_{i+1})\wedge\cdots \wedge f(x_r)\\
&=& \sum_{i=1}^r f(x_i)(x) \cdot f(x_1) \wedge \cdots \wedge f(x_{i-1})\wedge \ell \wedge f(x_{i+1})\wedge\cdots \wedge f(x_r) \\
&=& \left(\sum_{i=2}^r a_i f(x_i)(x) \right)\cdot \ell \wedge f(x_2) \wedge \cdots \wedge f(x_r) \\
 && +\sum_{i=2}^r f(x_i)(x) \cdot \left(\sum_{j=2}^r a_j f(x_j) \right)\wedge f(x_2) \wedge \cdots \wedge f(x_{i-1}) \wedge \ell \wedge f(x_{i+1}) \wedge \cdots \wedge f(x_r) \\
 &=&\sum_{i=2}^r a_i f(x_i)(x) \cdot \ell \wedge f(x_2) \wedge \cdots \wedge f(x_r) \\
 &&+\sum_{i=2}^r a_i f(x_i)(x) \cdot f(x_i)\wedge f(x_2)\wedge \cdots \wedge f(x_{i-1})\wedge \ell \wedge f(x_{i+1}) \wedge \cdots \wedge f(x_r) \\
 &=&0.
\end{eqnarray*}
Since $\bigwedge^r f$ is also zero in this case, this proves the desired commutativity.
\end{proof}

\subsubsection{}\label{proof of reg prop}We are now ready to prove Theorem \ref{reg prop}.

To do this we first apply Lemma \ref{alg} with $M:=L \otimes_\ZZ E(\QQ)$, $N:=L \otimes_{\QQ_p } H^2(\ZZ_S, V)^\ast$ and the exact sequence
$$0 \to L\otimes_{\QQ_p} H^2(\ZZ_S,V)^\ast \to L\otimes_\ZZ E(\QQ) \xrightarrow{\log_\omega} L \to 0,$$
which is obtained from (\ref{short}) (so we let $\ell$ in (\ref{nml}) be $\log_\omega$). We fix a $\ZZ_p$-basis of $I/I^2$ and identify it with $\ZZ_p$. By letting
$$f: M \to M^\ast; \ x \mapsto (y\mapsto \langle x,y\rangle_p)$$
and
$$g:=\varprojlim_n \beta_n :M \to N^\ast,$$
we see that assumptions (a) and (b) in Lemma \ref{alg} are satisfied by Lemmas \ref{comm} and \ref{sym} respectively. 

Let $\{ x_1,\ldots ,x_r\}$ be a basis of $E(\QQ)_{\rm tf} \subset M$. By the definition of $R_p$, we have
$$\left( {\bigwedge}^r f\right)(x_1\wedge \cdots \wedge x_r) = R_p\cdot x_1^\ast\wedge\cdots \wedge x_r^\ast \in {\bigwedge}_L^r M^\ast.$$
On the other hand, we have
\begin{eqnarray}\label{nek formula}
\delta \circ \left( {\bigwedge}^{r-1}g\right)(x_1\wedge\cdots \wedge x_r) = R_\omega^{\rm Boc} \otimes  (x_1^\ast \wedge \cdots \wedge x_r^\ast) \in M \otimes_L {\bigwedge}_L^r M^\ast,
\end{eqnarray}
where $\delta$ is as in (\ref{bigcomm}). This again follows from the definition of $R_\omega^{\rm Boc}$. Hence, for any $x \in E(\QQ)$, the commutativity of (\ref{bigcomm}) implies
$$f(x)(R_\omega^{\rm Boc}) = \ell(x) \cdot R_p,$$
i.e.,
$$\langle x, R_\omega^{\rm Boc}\rangle_p=\log_\omega(x) \cdot R_p.$$
This completes the proof of Theorem \ref{reg prop}.

\subsection{Schneider's height pairing}
 We now consider the case that $p$ is split multiplicative. In this case, the classical $p$-adic height pairing constructed by Schneider \cite{sch} is different from that of Nekov\' a\v r constructed above. Explicitly, Schneider's $p$-adic height pairing
$$\langle -,-\rangle_p^{\rm Sch}: E(\QQ ) \times E(\QQ) \to \QQ_p\otimes_{\ZZ_p}I/I^2$$
is related to Nekov\'a\v r's height pairing $\langle -, -\rangle_p$ by
\begin{eqnarray}\label{def sch}
\ell_p( \langle x,y\rangle_p^{\rm Sch})=\ell_p(\langle x,y\rangle_p) -\frac{\log_\omega(x)\log_\omega(y)}{\log_p(q_E)} \text{ in }\QQ_p,
\end{eqnarray}
where $\ell_p$ denotes the isomorphism
\begin{eqnarray}\label{lp}
\ell_p:\QQ_p \otimes_{\ZZ_p}I/I^2 \xrightarrow{\gamma-1 \mapsto \gamma} \QQ_p \otimes_{\ZZ_p} \Gamma \xrightarrow{\chi_{\rm cyc}}\QQ_p \otimes_{\ZZ_p} (1+p\ZZ_p) \xrightarrow{\log_p} \QQ_p,
\end{eqnarray}
with $\chi_{\rm cyc}$ the cyclotomic character, and $q_E \in \QQ_p$ is the $p$-adic Tate period of $E$. (See \cite[Th. 11.4.6]{nekovar}, where Schneider's height is denoted by $h_\pi^{\rm norm}$.) Note that, by the so-called `Saint Etienne Theorem' of Barr\'{e}-Sirieix, Diaz, Gramain and Philibert \cite{SaintEtienne}, one has $\log_p(q_E)\neq 0$ and so the above formula makes sense. Since the relation (\ref{def sch}) characterizes $\langle-,-\rangle_p^{\rm Sch}$, we adopt it as the definition of Schneider's $p$-adic height pairing.

\begin{definition}
We define Schneider's $p$-adic regulator
$$R_{p}^{\rm Sch} \in \QQ_p \otimes_{\ZZ_p} Q^r$$
by the discriminant of Schneider's $p$-adic height pairing, i.e.,
$$R_p^{\rm Sch}:=\det(\langle x_i,x_j\rangle_p^{\rm Sch})_{1\leq i ,j\leq r}$$
with $\{x_1,\ldots,x_r\}$ a basis of $E(\QQ)_{\rm tf}$.
\end{definition}

We identify $\QQ_p \otimes_{\ZZ_p}I/I^2 = \QQ_p$ via the isomorphism $\ell_p$.
By using the relation (\ref{def sch}), one checks that
$$R_p^{\rm Sch}=R_p -\frac{1}{\log_p(q_E)} \sum_{i=1}^r \log_\omega(x_i)\det \begin{pmatrix}
\langle x_1,x_1\rangle_p &\langle x_1,x_2 \rangle_p & \cdots & \log_\omega(x_1)& \cdots & \langle x_1,x_r\rangle_p \\
\langle x_2,x_1\rangle_p & \cdots& \cdots&\log_\omega(x_2)& \cdots & \langle x_2,x_r \rangle_p \\
\vdots & && \vdots&  &\vdots \\
\langle x_r,x_1 \rangle_p &\cdots&\cdots& \log_\omega(x_r)&\cdots& \langle x_r,x_r  \rangle_p
\end{pmatrix},$$
where the vector $(\log_\omega(x_j))_j$ is put on the $i$-th column in the matrix on the right hand side. In fact, this follows from the elementary formula
$$\det(a_{ij}+c b_i b_j)=\det(a_{ij}) +c \sum_{i=1}^r b_i \det \begin{pmatrix}
a_{11} & a_{12} & \cdots & b_1& \cdots & a_{1r} \\
a_{21} & \cdots& \cdots&b_2& \cdots & a_{2r} \\
\vdots & && \vdots&  &\vdots \\
a_{r1} &\cdots&\cdots& b_{r}&\cdots& a_{rr}
\end{pmatrix}$$
(with the vector $(b_j)_j$ put on the $i$-th column). 
Furthermore, by (\ref{delta exp}) and (\ref{nek formula}), we have
$$R_\omega^{\rm Boc}= \sum_{i=1}^r x_i \otimes \det \begin{pmatrix}
\langle x_1,x_1\rangle_p &\langle x_1,x_2 \rangle_p & \cdots &\log_\omega(x_1)& \cdots & \langle x_1,x_r\rangle_p \\
\langle x_2,x_1\rangle_p & \cdots& \cdots&\log_\omega(x_2)& \cdots & \langle x_2,x_r \rangle_p \\
\vdots & && \vdots&  &\vdots \\
\langle x_r,x_1 \rangle_p &\cdots&\cdots& \log_\omega(x_r)&\cdots& \langle x_r,x_r  \rangle_p
\end{pmatrix},$$
and hence we have
$$R_p^{\rm Sch}=R_p - \frac{\log_\omega(R_\omega^{\rm Boc})}{\log_p(q_E)}.$$
From this and Theorem \ref{reg prop}, we obtain the following.

\begin{theorem}\label{reg prop 2}
For any $x \in E(\QQ)$ we have
$$\langle x , R_\omega^{\rm Boc} \rangle_p^{\rm Sch} =\log_\omega(x)\cdot R_p^{\rm Sch}.$$
\end{theorem}

\section{The Generalized Rubin Formula and consequences}

In this section we relate Conjectures \ref{mrs2} and \ref{mrs3} to the $p$-adic analogue of the Birch and Swinnerton-Dyer conjecture formulated by Mazur, Tate and Teitelbaum in \cite{MTT} (see Corollaries \ref{cor padic} and \ref{cor beilinson}).

In particular, we continue to assume in this section that $E$ does not have additive reduction at $p$.

\subsection{Review of the $p$-adic $L$-function}

In this subsection, we review the $p$-adic $L$-function of Mazur-Tate-Teitelbaum \cite{MTT}. See also the review in \cite[\S 16.1]{katoasterisque}.

When $p$ is good, let $\alpha \in \overline \QQ_p$ be a root of $X^2-a_p X +p$ such that $\ord_p(\alpha) <1$ (an `allowable root'), and $\beta(:=p/\alpha)$ the other root. Note that, when $p$ is good ordinary, $\alpha$ is uniquely determined by this property.

When $p$ is split (resp. non-split) multiplicative, we set $\alpha:=1$ (resp. $-1$) and $\beta:=p$ (resp. $-p$).

We set
$$L:=\QQ_p(\alpha).$$
Note that $L=\QQ_p$ unless $p$ is supersingular.

Recall that $\QQ_\infty/\QQ$ denotes the cyclotomic $\ZZ_p$-extension and $\Gamma:=\Gal(\QQ_\infty/\QQ)$. Let $\widehat \Gamma$ denote the set of $\overline \QQ$-valued characters of $\Gamma$ of finite order.

Recall also that an embedding $\overline \QQ \hookrightarrow \CC$ is fixed. For a positive integer $m$, let $\zeta_m \in \overline \QQ$ be the element corresponding to $e^{2\pi\sqrt{-1}/m} \in \CC$. We also fix an isomorphism $\CC \simeq \CC_p$. From this, we obtain an embedding $\overline \QQ \hookrightarrow \overline \QQ_p$. Thus each character in $\widehat \Gamma$ is regarded both $\overline \QQ_p$ and $\CC$-valued.

As in \S \ref{sec 2}, we fix a N\'eron differential $\omega \in \Gamma(E,\Omega_{E/\QQ}^1)$. Let $\xi$ be the element of ${\rm SL}_2(\ZZ)$ used in the construction of Kato's Euler system (and normalized as in (\ref{xi condition}). Let $\Omega_\xi$ be the real period associated to $(\omega,\xi)$ (see (\ref{per xi})).

We fix a topological generator $\gamma$ of $\Gamma$. Then we have a natural identification
$$\cO_L[[\Gamma]]=\cO_L[[\gamma-1]].$$
Let $|-|_p: \CC_p \to \RR_{\geq 0} $ denote the $p$-adic absolute value normalized by $|p|_p=p^{-1}$. For a positive integer $h$, we define
$$\cH_{h}:=\left\{ \sum_{n=0}^\infty c_n (\gamma-1)^n \in L[[\gamma-1]] \ \middle| \  \underset{n \to \infty}{\lim} \frac{|c_{n}|_p}{n^h} =0\right\}$$
and
$$\cH_{\infty}:=\bigcup_h \cH_{h}.$$
For any continuous character $\chi: \Gamma \to \overline \QQ_p^\times$ and $f =\sum_n c_n (\gamma-1)^n \in \cH_{\infty}$, we can define a natural evaluation
$$\chi(f):= \sum_n c_n (\chi(\gamma)-1)^n \in \overline \QQ_p.$$

It is known that there is a unique element (the `$p$-adic $L$-function' of $E$)
$$\cL_{S,p}=\cL_{S,p,\alpha,\omega,\xi} \in \cH_{1}$$
that has the following property: for any character $\chi \in \widehat \Gamma$ one has
$$\chi(\cL_{S,p})=\begin{cases}\displaystyle
\left(1-\frac 1\alpha \right)\left(1-\frac 1\beta \right)^{-1} \frac{L_S(E,1)}{\Omega_\xi} &\text{if $\chi=1$},\\
\displaystyle\frac{\tau(\chi)}{\alpha^n} \frac{L_S(E,\chi^{-1},1)}{\Omega_\xi} &\text{if $\chi$ has conductor $p^n>1$}.
\end{cases}$$
Here in the latter case $\tau(\chi)$ denotes the Gauss sum
$$\tau(\chi):=\sum_{\sigma \in \Gal(\QQ(\mu_{p^n})/\QQ)} \chi(\sigma)\zeta_{p^n}^\sigma,$$
and $L_S(E,\chi^{-1},s)$ denotes the $S$-truncated Hasse-Weil $L$-function of $E$ twisted by $\chi^{-1}$.
For the construction of $\cL_{S,p}$ from Kato's Euler system, see Theorem \ref{katothm} below.

Let $\cI:=(\gamma-1)$ be the augmentation ideal of $\cH_\infty$. For a non-negative integer $a$, we set
$$\cQ^a:=\cI^a/\cI^{a+1}.$$
Note that we have a natural identification
$$\cQ^a = L \otimes_{\ZZ_p} Q^a.$$

We know the following `order of vanishing' (which is actually a consequence of Proposition \ref{kato vanish}).


\begin{proposition}[{\cite[Th. 18.4]{katoasterisque}}]\label{padic order}
Set $r:={\rm rank}_\ZZ(E(\QQ))$. Then we have
$$\cL_{S,p} \in \begin{cases}
\cI^r &\text{if $p$ is good or non-split multiplicative,}\\
\cI^{r+1} &\text{if $p$ is split multiplicative.}
\end{cases}
$$
\end{proposition}



\subsection{The Generalized Rubin Formula}\label{gen RF section}

Let $\cL_{S,p}^{(r)} $ (resp. $\cL_{S,p}^{(r+1)}$) denote the image of $\cL_{S,p} \in \cI^r$ (resp. $\cI^{r+1}$) in $\cQ^r$ (resp. $ \cQ^{r+1}$) when $p $ is good or non-split multiplicative (resp. split multiplicative).


Recall some notations. Let
$$\langle -,-\rangle_p=\langle -,-\rangle_{p,\alpha}: E(\QQ) \times (\QQ_p \otimes_\ZZ E(\QQ))\otimes_{\ZZ_p} Q^{r-1} \to L \otimes_{\ZZ_p} Q^r=\cQ^r$$
be the map induced by the $p$-adic height pairing (see (\ref{pad2})). Let $\log_\omega: E(\QQ_p) \to \QQ_p$ be the formal logarithm associated to the fixed N\'eron differential $\omega$.
Let
$$\kappa_\infty \in H^1(\ZZ_S,V) \otimes_{\ZZ_p}Q^{r-1} \simeq (\QQ_p \otimes_\ZZ E(\QQ))\otimes_{\ZZ_p}Q^{r-1}.$$
be the Iwasawa-Darmon derivative in Definition \ref{def iw}.

The following is a generalization of `Rubin's formula' for the higher rank case.


\begin{theorem}[The Generalized Rubin Formula]\label{main}
Under Hypothesis \ref{hyp}, we have the following.
\begin{itemize}
\item[(i)] If $p$ is good or non-split multiplicative, then for any $x \in E(\QQ)$ we have
$$ \langle x, \kappa_\infty\rangle_p = \left( 1-\frac 1\alpha \right)^{-1}\left(1-\frac 1\beta\right)\log_\omega(x)\cdot  \cL_{S,p}^{(r)} \text{ in } \cQ^r.$$
\item[(ii)] If $p$ is split multiplicative, then for any $x \in E(\QQ)$ we have
$$ \langle x, \kappa_\infty\rangle_p^{\rm Sch} \cdot \frac{1}{{\rm ord}_p(q_E)} (1-{\rm rec}_p(q_E))= \left(1-\frac 1p\right)\log_\omega(x)\cdot  \cL_{S,p}^{(r+1)} \text{ in } \cQ^{r+1}.$$
Here $q_E \in \QQ_p^\times$ denotes the $p$-adic Tate period of $E$ and ${\rm rec}_p: \QQ_p^\times\to \Gamma$ the local reciprocity map.
\end{itemize}
\end{theorem}

The proof of this theorem will be given in \S \ref{section proof}.

\begin{remark}
When $r=1$, we have $\kappa_\infty=z_\QQ$ (see Remark \ref{zq}), so Theorem \ref{main}(i) asserts
$$ \langle x, z_\QQ \rangle_p = \left( 1-\frac 1\alpha \right)^{-1}\left(1-\frac 1\beta\right)\log_\omega(x)\cdot  \cL_{S,p}^{(1)} \text{ in } \cI/\cI^{2}.$$
When $p$ is good ordinary, this formula is proved by Rubin \cite[Th. 1(ii)]{rubin}, which we call  `Rubin's formula' (following Nekov\'a\v r \cite[(11.3.14)]{nekovar}). 
(Note that `$L_{\bm{z},\omega}'(\bm{1})$' in \cite[Th. 1(ii)]{rubin} corresponds to our $\left(1-\frac 1\alpha\right) \cL_{S,p}^{(1)}$.)
Thus Theorem \ref{main}(i) is regarded as a `higher rank' generalization of Rubin's formula.
\end{remark}

\begin{remark}\label{remL}
The element
$$\frac{1}{{\rm ord}_p(q_E)} (1-{\rm rec}_p(q_E)) \in \QQ_p \otimes_{\ZZ_p} I/I^2$$
appearing in Theorem \ref{main}(ii) is essentially the `$\cL$-invariant'. In fact, one checks that the image of this element under the isomorphism
$$\QQ_p \otimes_{\ZZ_p}I/I^2 \xrightarrow{\gamma-1 \mapsto \gamma} \QQ_p \otimes_{\ZZ_p} \Gamma \xrightarrow{\chi_{\rm cyc}}\QQ_p \otimes_{\ZZ_p} (1+p\ZZ_p) \xrightarrow{\log_p} \QQ_p$$
(see (\ref{lp}))
is the usual $\cL$-invariant
$$\frac{\log_p(q_E)}{\ord_p(q_E)}.$$

\end{remark}

\begin{remark}
When $r=1$, Theorem \ref{main}(ii) is obtained by Venerucci \cite[Th. 12.31]{venerucci thesis} and B\"uy\"ukboduk \cite[Th. 3.22]{buyukexc}.
\end{remark}

A proof of Theorem \ref{main} will be given in \S \ref{section proof}. We state here some consequences of the theorem. Recall that $v_\xi \in \QQ^\times$ is defined by $\Omega^+=v_\xi \cdot \Omega_\xi$ (see (\ref{omega})).

\begin{corollary}\label{cor padic}
Conjecture \ref{mrs3} implies the $p$-adic Birch-Swinnerton-Dyer Formula in \cite[Chap. II, \S 10]{MTT}, i.e.,
$$ \left( 1-\frac 1\alpha \right)^{-1}\left(1-\frac 1\beta\right)\cdot  \cL_{S,p}^{(r)} =v_\xi\left(\prod_{\ell \in S \setminus \{\infty\}} L_\ell \right)\frac{\# \sha(E/\QQ)\cdot {\rm Tam}(E)}{\# E(\QQ)_{\rm tors}^2} R_p $$
if $p$ is good or non-split multiplicative, and
$$  \cL_{S,p}^{(r+1)} =\frac{1}{{\rm ord}_p(q_E)} (1-{\rm rec}_p(q_E))\cdot  v_\xi\left(\prod_{\ell \in S \setminus \{\infty,p\}} L_\ell \right)\frac{\# \sha(E/\QQ)\cdot {\rm Tam}(E)}{\# E(\QQ)_{\rm tors}^2} R_p^{\rm Sch} $$
if $p$ is split multiplicative.

If $R_p\neq 0$ (resp. $R_p^{\rm Sch}\neq 0$), then the converse also holds when $p$ is good or non-split multiplicative (resp. split multiplicative).
\end{corollary}
\begin{proof}
We only treat the case when $p$ is good or non-split multiplicative. The case when $p$ is split multiplicative is treated similarly, by using Theorem \ref{reg prop 2}.

Conjecture \ref{mrs3} asserts
$$\kappa_\infty =v_\xi\left(\prod_{\ell \in S \setminus \{\infty\}} L_\ell \right)\frac{\# \sha(E/\QQ) {\rm Tam}(E)}{\# E(\QQ)_{\rm tors}^2} \cdot R_\omega^{\rm Boc} \text{ in }(\QQ_p \otimes_\ZZ E(\QQ))\otimes_{\ZZ_p} Q^{r-1}.$$
Take $x \in E(\QQ)$ such that $\log_\omega(x)\neq 0$. Taking $\langle x, - \rangle_p$ to both sides, we obtain
$$ \left( 1-\frac 1\alpha \right)^{-1}\left(1-\frac 1\beta\right)\log_\omega(x)\cdot  \cL_{S,p}^{(r)} =v_\xi\left(\prod_{\ell \in S \setminus \{\infty\}} L_\ell \right)\frac{\# \sha(E/\QQ) {\rm Tam}(E)}{\# E(\QQ)_{\rm tors}^2} \log_\omega(x) R_p $$
by Theorems \ref{main} and \ref{reg prop}. Since $\log_\omega(x)\neq 0$, we can cancell it from both sides and obtain the desired formula.

If $R_p\neq0$, then the map $y \mapsto (x \mapsto \langle x,y\rangle_p)$ is injective, and so the converse holds.
\end{proof}

Similarly, we also obtain the following.

\begin{corollary}\label{cor beilinson}
Conjecture \ref{mrs2}  implies the $p$-adic Beilinson Formula, i.e.,
\begin{eqnarray}\label{bei}
 \left( 1-\frac 1\alpha \right)^{-1}\left(1-\frac 1\beta\right)\cdot  \cL_{S,p}^{(r)} =\frac{L_S^\ast(E,1)}{\Omega_\xi\cdot R_\infty}R_p
 \end{eqnarray}
if $p$ is good or non-split multiplicative, and
\begin{eqnarray}\label{gs}
  \cL_{S,p}^{(r+1)} =\frac{1}{{\rm ord}_p(q_E)} (1-{\rm rec}_p(q_E))\cdot  \frac{L_{S\setminus \{p\}}^\ast(E,1)}{\Omega_\xi\cdot R_\infty}R^{\rm Sch}_p
  \end{eqnarray}
if $p$ is split multiplicative.

If $R_p\neq 0$ (resp. $R_p^{\rm Sch}\neq 0$), then the converse also holds when $p$ is good or non-split multiplicative (resp. split multiplicative).
\end{corollary}

\begin{proof}
This follows by the same argument as the proof of Corollary \ref{cor padic}, using Proposition \ref{prop exp}.
\end{proof}

\begin{remark}
When $p$ is good and $r_{\rm an} =r=1$, the formula (\ref{bei}) was proved by Perrin-Riou \cite[Cor. 1.8]{PR87} in the ordinary case, and by Kobayashi \cite[Cor. 1.3]{kobss} in the supersingular case. (It is essentially the `$p$-adic Gross-Zagier Formula'.) When $p$ is split multiplicative and $r_{\rm an}=r=0$, the formula (\ref{gs}) was first proved by Greenberg and Stevens \cite{gs} and then by Kobayashi \cite{kobayashi}.
\end{remark}

\subsection{Review of the Coleman map} \label{sec coleman}
As a preliminary of the proof of Theorem \ref{main}, we review the construction of the Coleman map. We follow the explicit construction due to Rubin \cite[Appendix]{rubinmodular}. See also \cite[\S 3]{kuriharass}.

We set
$$D:=  D_{\rm crys}(V) .$$
Let $\varphi$ denote the Frobenius operator acting on $D$. For a finite extension $K/\QQ_p$, we set
$$D_K:= K \otimes_{\QQ_p}D .$$
Let
$$[-,-]_K: (K \otimes_{\QQ_p}D_{\rm dR}(V) ) \times D_K \to K$$
denote the natural pairing.

We use the following fact.

\begin{lemma}[{\cite[Th. 16.6(1)]{katoasterisque}}]\label{lemnu}
Set $L:=\QQ_p(\alpha)$. There exists a unique $\nu=\nu_{\alpha,\omega} \in D_L$ such that
$$\varphi(\nu)=\alpha p^{-1}\nu =\beta^{-1} \nu \text{ and }[\omega,\nu]_L=1.$$
\end{lemma}

Let $\QQ_{n,p}$ denote the completion of $\QQ_n$ at the unique prime lying above $p$. We set
$$L_n:=L\cdot \QQ_{n,p}.$$
Let $\nu \in D_L$ be as in Lemma \ref{lemnu} and set
\begin{eqnarray}\label{def delta}
\delta_n&:=&\frac{1}{p^{n+1}} {\rm Tr}_{L(\mu_{p^{n+1}})/L_n}\left( \sum_{i=0}^n \zeta_{p^{n+1-i}}\varphi^{i-n-1}(\nu)  +(1-\varphi)^{-1}(\nu)\right)\\
&=&\frac{1}{\alpha^{n+1}} {\rm Tr}_{L(\mu_{p^{n+1}})/L_n}\left( \sum_{i=0}^n \frac{\zeta_{p^{n+1-i}}-1}{\beta^i}  +\frac{\beta}{\beta-1}\right) \nu  \in D_{L_n}.\nonumber
\end{eqnarray}
This element satisfies
$${\rm Tr}_{L_{n+1}/L_n}(\delta_{n+1})=\delta_n$$
and for any character $\chi$ of $G_n$
\begin{eqnarray}\label{dn}
\sum_{\sigma \in G_n} \sigma(\delta_n)\chi(\sigma) =\begin{cases}
\displaystyle\left( 1-\frac 1\alpha \right) \left( 1-\frac 1\beta\right)^{-1} \nu &\text{if $\chi=1$,}\\
\displaystyle\frac{\tau(\chi)}{\alpha^m}  \nu&\text{if $\chi$ has conductor $p^m>1$}
\end{cases}
\end{eqnarray}
in $D_{L(\mu_{p^{n+1}})}$ (see \cite[Lem. A.1]{rubinmodular} or \cite[Lem. 3.1]{kuriharass}).

As in \S \ref{sec id}, we set
$$H^i_n:=H^i(\cO_{\QQ_n,S},T) \text{ and }\HH^i:=\varprojlim_n H^i_n.$$
We define a map
$${\rm Col}_n: H^1_n\to L[G_n]$$
by
$${\rm Col}_n(z):=\sum_{\sigma \in G_n} {\rm Tr}_{L_n/L}([\exp_n^\ast(z), \sigma \delta_n]_{L_n})\sigma,$$
where
$$\exp_n^\ast=\exp^\ast_{\QQ_{n,p},V}: H^1_n \to H^1(\QQ_{n,p},T) \to \QQ_{n,p}\otimes_{\QQ_p}D_{\rm dR}(V) $$
denotes the Bloch-Kato dual exponential map. This map induces a map on the inverse limit
$${\rm Col}:={\varprojlim_n {\rm Col}_n}: \HH^1  \to \cH_{\infty}.$$
This is the definition of the Coleman map.

We set
\begin{eqnarray}\label{tcd}
t_{c,d}:=cd(c-\sigma_c)(d-\sigma_d) \in \ZZ_p[[\Gamma]].
\end{eqnarray}
Here $\sigma_a \in \Gamma$ is the restriction of the automorphism of $\QQ(\mu_{p^{\infty}})$ characterized by $\zeta_{p^{n}}^{\sigma_a}=\zeta_{p^{n}}^a$ for every $n$.

The following result is well-known.

\begin{theorem}[Kato {\cite[Th. 16.6(2)]{katoasterisque}}]\label{katothm}
We have
$${\rm Col}(({}_{c,d}z_n)_n) =t_{c,d}\cdot \cL_{S,p}.$$
\end{theorem}

\subsection{The proof of Theorem \ref{main}} \label{section proof}


In this subsection, we prove Theorem \ref{main}.

\subsubsection{}We first establish several important preliminary results.

We initially suppose that $p$ is good or non-split multiplicative, and give a proof of Theorem \ref{main}(i).


We shall use the derivative introduced by Nekov\'a\v r in \cite[\S 11.3.14]{nekovar}, based on the idea of Rubin in \cite{rubin}. 

With the notations in \S \ref{review height}, we set
$$F^-V:=\begin{cases}
V/F^+V &\text{if $p$ is ordinary},\\
\drig(V_L)/\DD_\alpha &\text{if $p$ is supersingular}.
\end{cases}$$

For $y \in \HH^1$, we define `Rubin's derivative'
$$\cD(y) \in H^1(\QQ_p, F^-V) \otimes_{\ZZ_p} I/I^2$$
as follows. (Compare the definition given by Nekov\'a\v r in \cite[\S 11.3.14]{nekovar}, where the symbol `$D x_{\rm Iw}$' is used.) 

Suppose first that $p$ is ordinary. We have a commutative diagram with exact rows and columns

\begin{eqnarray} \label{DistTri}
\small
\quad \quad
\xymatrix{
\widetilde \rgamma_f(\QQ,V) \lotimes_{\ZZ_p}I/I^2  \ar[r] \ar[d]& \rgamma(\ZZ_S,V)\lotimes_{\ZZ_p} I/I^2 \ar[r]\ar[d]& \rgamma(\QQ_p,F^-V)\lotimes_{\ZZ_p}I/I^2 \ar[d]^i\\
\widetilde \rgamma_{f,{\rm Iw}}(\QQ,V) \lotimes_\Lambda \Lambda/I^2 \ar[r]\ar[d]& \rgamma_{\rm Iw}(\ZZ_S,V) \lotimes_\Lambda \Lambda/I^2 \ar[r]^{{\rm loc}_p}\ar[d]& \rgamma_{\rm Iw}(\QQ_p,F^-V)\lotimes_\Lambda \Lambda/I^2 \ar[d]\\
\widetilde \rgamma_f(\QQ,V)  \ar[r]& \rgamma(\ZZ_S,V) \ar[r]& \rgamma(\QQ_p,F^-V) .
}
\label{big diag}
\end{eqnarray}
Here $\widetilde \rgamma_f(\QQ,V):=\widetilde \rgamma_f(\QQ,T)\lotimes_{\ZZ_p}\QQ_p$ and
$$\widetilde \rgamma_{f,{\rm Iw}}(\QQ,V):=\left(\varprojlim_n \widetilde \rgamma_f(\QQ_n,T)\right) \lotimes_{\ZZ_p}\QQ_p.$$
$\rgamma_{\rm Iw}(\ZZ_S,V)$ and $\rgamma_{\rm Iw}(\QQ_p,F^-V)$ are defined in a similar way.

We regard $y \in \HH^1$ as an element of $H^1(\rgamma_{\rm Iw}(\ZZ_S,V)\lotimes_\Lambda \Lambda/I^2)$. Since $y_0$ lies in $\widetilde H^1_f(\QQ,V)$ and $H^0(\QQ_p,F^-V)=0$, a diagram chasing shows that there exists a unique element $\cD(y) \in H^1(\QQ_p, F^-V) \otimes_{\ZZ_p} I/I^2$ such that
$${\rm loc}_p(y)=i(\cD(y)) \text{ in }H^1(\rgamma_{\rm Iw}(\QQ_p,F^-V)\lotimes_\Lambda \Lambda/I^2).$$
This gives the definition of Rubin's derivative in this case.

When $p$ is supersingular, Rubin's derivative is defined in the same way, by considering the commutative diagram with exact rows and columns
$$
\xymatrix{
\widetilde \rgamma_f(\QQ,V_L) \lotimes_{\ZZ_p}I/I^2  \ar[r] \ar[d]& \rgamma(\ZZ_S,V_L)\lotimes_{\ZZ_p} I/I^2 \ar[r]\ar[d]& \rgamma(\QQ_p,F^-V)\lotimes_{\ZZ_p}I/I^2 \ar[d]\\
\widetilde \rgamma_{f,{\rm Iw}}(\QQ,V_L) \lotimes_\cH \cH/\cI^2 \ar[r]\ar[d]& \rgamma(\ZZ_S,\overline V_L) \lotimes_\cH \cH/\cI^2 \ar[r]\ar[d]& \rgamma_{\rm Iw}(\QQ_p,F^-V)\lotimes_\cH \cH/\cI^2 \ar[d]\\
\widetilde \rgamma_f(\QQ,V_L)  \ar[r]& \rgamma(\ZZ_S,V_L) \ar[r]& \rgamma(\QQ_p,F^-V) .
}
$$


Let
$$(-,-)_p: H^1_f(\QQ_p, V) \times H^1(\QQ_p,F^-V) \to H^2(\QQ_p, L(1)) \simeq L$$
be the pairing defined by the cup product. This pairing induces
\begin{eqnarray}\label{cup}
(-,-)_p: E(\QQ) \times (H^1(\QQ_p,F^-V) \otimes_{\ZZ_p} I/I^2) \to L\otimes_{\ZZ_p} I/I^2=\cI/\cI^2.
\end{eqnarray}

The following is an abstract version of Rubin's formula.

\begin{theorem}[{Rubin, Nekov\'a\v r}] \label{rubin abs}
Suppose that $p$ is not split multiplicative. For any $x \in E(\QQ)$ and $y=(y_n)_n \in \varprojlim_n H^1_n=\HH^1$, we have
$$\langle x, y_0 \rangle_p = (x,\cD(y))_p \text{ in }\cI/\cI^2.$$
\end{theorem}
\begin{proof}
This is proved in \cite[Prop. 11.3.15]{nekovar}. We give a proof for the reader's convenience.
We treat only the ordinary case, since the supersingular case is treated in a similar way.

Recall that the map $\widetilde \beta: \widetilde H^1_f(\QQ,V) \to \widetilde H^2_f(\QQ,V)\otimes_{\ZZ_p}I/I^2$ in (\ref{beta tilde}) is defined to be $(-1)$-times the connecting homomorphism of the left vertical triangle of (\ref{big diag}). Let $\delta : H^1(\QQ_p,F^-V)\otimes_{\ZZ_p}I/I^2 \to \widetilde H^2_f(\QQ,V)$ be the connecting homomorphism of the upper horizontal triangle of (\ref{big diag}). Then, by the compatibility of connecting homomorphisms (see \cite[Lem. 1.2.19]{nekovar}), we have
$$\widetilde \beta(y_0)=\delta(\cD(y)).$$
We identify $\widetilde H^2_f(\QQ,V) =\widetilde H^1_f(\QQ,V)^\ast = \QQ_p\otimes_\ZZ E(\QQ)^\ast$ by global duality. Then for any $x \in E(\QQ)$ we have
$$\widetilde \beta(y_0)(x)=\langle x,y_0 \rangle_p$$
by the definition of the $p$-adic height pairing. On the other hand, by the compatibility between local and global duality, we have
$$\delta(\cD(y))(x)=(x,\cD(y))_p.$$
Thus we have
$$\langle x, y_0 \rangle_p = (x,\cD(y))_p .$$
\end{proof}

We shall now apply Theorem \ref{rubin abs} in our setting.

\begin{lemma}\label{kappa uni}
Let ${}_{c,d}\kappa_\infty \in H^1_0\otimes_{\ZZ_p}Q^{r-1}$ be the Iwasawa-Darmon derivative in Definition \ref{def iw}. Then there exists a unique $w=(w_n)_n \in \varprojlim_n H^1_n =\HH^1$ such that
$${}_{c,d}z_n=(\gamma-1)^{r-1}w_n$$
for every $n$ and
$${}_{c,d}\kappa_\infty =w_0 \otimes (\gamma-1)^{r-1}.$$
\end{lemma}
\begin{proof}
By Proposition \ref{kato vanish}, one can take $w_n \in H^1_n$ such that ${}_{c,d}z_n=(\gamma-1)^{r-1}w_n$. This element is well-defined modulo $H^1_0$, so we see that the collection $(w_n)_n$ is an inverse system in $\varprojlim_n H^1_n/p^n$.
However, since $\varprojlim_n H^1_n/p^n$ is isomorphic to $\varprojlim_n H^1_n =\HH^1$, we can take each $w_n \in H^1_n$ so that $(w_n)_n \in \HH^1$. The description of ${}_{c,d}\kappa_\infty$ follows from (\ref{kappa exp}).
\end{proof}

By Lemma \ref{kappa uni}, we can define the `Rubin's derivative of the Iwasawa-Darmon derivative'
$$\cD({}_{c,d}\kappa_\infty):=\cD(w)\cdot (\gamma-1)^{r-1} \in H^1(\QQ_p,F^-V) \otimes_{\ZZ_p} Q^r.$$
Applying Theorem \ref{rubin abs} to this element, we obtain the following.

\begin{corollary}\label{cor kappa}
For any $x \in E(\QQ)$, we have
$$\langle x, {}_{c,d}\kappa_\infty \rangle_p = (x, \cD({}_{c,d}\kappa_\infty))_p \text{ in }\cQ^r,$$
where
$$(-,-)_p: E(\QQ) \times (H^1(\QQ_p,F^-V) \otimes_{\ZZ_p} Q^r) \to \cQ^r$$
is the map induced by (\ref{cup}).
\end{corollary}

\begin{lemma}\label{col3}
Let $y \in  \HH^1$. Then we have
$${\rm Col}(y) \in \cI$$
and
$${\rm Col}(y)=(\exp_0(\delta_0),\cD(y))_p \text{ in }\cI/\cI^2,$$
where $\exp_0=\exp_{\QQ_p,V}: D_L \to L\otimes_{\QQ_p} H_f^1(\QQ_p,V)$ denotes the Bloch-Kato exponential map.
\end{lemma}
\begin{proof}
We shall show the first claim. By the construction of the Coleman map, it is sufficient to show that
$$\sum_{\sigma \in G_n}{\rm Tr}_{L_n/L}\left( [\exp^\ast_n(y_n),\sigma \delta_n]_{L_n}\right)=0$$
for every $n$. The left hand side is equal to $[\exp_0^\ast(y_0),\delta_0]_{L}$. Since $y_0$ lies in $H^1_f(\QQ,V)$, we know that $\exp_0^\ast(y_0)=0$ and so we have proved the first claim.

Next, we shall show the second claim. Note that, by construction, we have
$${\rm Col}_n(y_n)=\sum_{\sigma \in G_n}(\exp_n(\delta_n),\sigma y_n)_{L_n}\sigma^{-1},$$
where $\exp_n: D_{L_n} \to H^1_f(L_n,V)$ denotes the Bloch-Kato exponential map and
$$(-,-)_{L_n}: H^1_f(L_n,V)\times H^1(\QQ_{n,p},F^-V) \to L$$
denotes the cup product pairing. Noting this, one verifies
$${\rm Col}(y)=(\exp_0(\delta_0),\cD(y))_p \text{ in }\cI/\cI^2$$
by the definition of $\cD(y)$.
\end{proof}


\subsubsection{Proof of Theorem \ref{main}(i)} Let $w \in \HH^1$ be the element in Lemma \ref{kappa uni}. We compute
\begin{eqnarray*}
t_{c,d} \cdot \mathcal{L}_{S,p}&=& {\rm Col}(({}_{c,d}z_n)_n) \quad  \text{ (by Theorem \ref{katothm})}\\
&=&{\rm Col}(w)\cdot (\gamma-1)^{r-1} \quad \text{ (by Lemma \ref{kappa uni})}\\
& \in& \cI^r \quad \text{ (by Lemma \ref{col3})}.
\end{eqnarray*}
Hence, in the quotient $\cQ^r=\cI^r/\cI^{r+1}$, we compute
\begin{eqnarray*}
t_{c,d} \cdot \mathcal{L}_{S,p}^{(r)}&=& (\exp_0 (\delta_0), \cD(w))_p\cdot (\gamma-1)^{r-1}  \quad \text{ (by Lemma \ref{col3})}\\
&=& (\exp_0(\delta_0), \cD({}_{c,d}\kappa_\infty))_p .
\end{eqnarray*}
By (\ref{dn}), note that
$$\delta_0=\left( 1-\frac 1 \alpha\right)\left( 1-\frac 1 \beta \right)^{-1} \nu.$$
Since $[\omega,\nu]_L=1$ by Lemma \ref{lemnu}, we have
$$\left( 1-\frac 1 \alpha\right)^{-1}\left(1-\frac 1\beta \right)\log_\omega(x) \exp_0(\delta_0)=x \text{ in }H^1_f(\QQ_p,V)$$
for any $x \in E(\QQ)$. Thus we have
\begin{eqnarray*}
\left( 1-\frac 1 \alpha\right)^{-1}\left(1-\frac 1\beta \right)\log_\omega(x) t_{c,d} \cdot \mathcal{L}_{S,p}^{(r)}&=&  (x, \cD({}_{c,d}\kappa_\infty))_p \\
&=& \langle x, {}_{c,d}\kappa_\infty \rangle_p \quad \text{ (by Corollary \ref{cor kappa})}.
\end{eqnarray*}
Upon multiplying both sides by $t_{c,d}^{-1}$ we obtain the desired formula.

This completes the proof of Theorem \ref{main}(i).

\subsubsection{}We now suppose that $p$ is split multiplicative and prepare for the proof of Theorem \ref{main}(ii).

Note first that, by Tate's uniformization, we have an exact sequence of $G_{\QQ_p}$-modules
\begin{eqnarray}\label{texact}
0 \to \ZZ_p(1) \to T \to \ZZ_p \to 0.
\end{eqnarray}
This means that $F^+V \simeq \QQ_p(1)$ and $F^-V :=V/F^+V \simeq \QQ_p$.

Since $H^0(\QQ_p,F^- V)$ does not vanish in this case,
Rubin's derivative $\cD(y)$ is not determined uniquely, so we impose more
condition to define it. Let 
$$\rho_p: H^0(\QQ_p, F^{-}V) \to H^1(\QQ_p, F^{-}V) \otimes_{\ZZ_p} I/I^2$$
be
the connecting homomorphism obtained from the right vertical exact triangle in 
(\ref{DistTri}). We know that
$$\im (\rho_p) = \langle \log_p \chi_{\rm cyc} \rangle \otimes_{\ZZ_p} I/I^2,$$
where we regard $\log_p \chi_{\rm cyc}: G_{\QQ_p} \to \QQ_p$ as an element of $H^1(\QQ_p, F^{-}V)= H^1(\QQ_p, \QQ_p)=\Hom_{\rm cont}(G_{\QQ_p}, \QQ_p)$. 
(See the proof of \cite[Lem. 15.1]{venerucci thesis} for example.) Let
$$\pi_p: H^1(\QQ_p, V) \otimes_{\ZZ_p} I/I^2 \to
H^1(\QQ_p, F^{-}V) \otimes_{\ZZ_p} I/I^2$$
be the map induced by $V \twoheadrightarrow F^{-} V$. Then one sees that $\im (\rho_p) \cap \im (\pi_p) =0$ (since $\log_p(q_E)\neq 0$), by which one can take a unique element
$$\cD(y) \in  \im (\pi_p)$$
such that ${\rm loc}_p(y)=i(\cD(y))$ in $H^1(\rgamma_{\rm Iw}(\QQ_p,F^-V)\lotimes_\Lambda \Lambda/I^2)$. 
Compare Venerucci's construction \cite[Lem. 15.1]{venerucci thesis} (where $I/I^2$ is identified with $\ZZ_p$). 

An analogue of Theorem \ref{rubin abs} is as follows.

\begin{theorem}\label{rubin vene}
Suppose that $p$ is split multiplicative. For any $x \in E(\QQ)$ and $y=(y_n)_n \in \varprojlim_n H^1_n =\HH^1$, we have
$$\langle x,y_0 \rangle_p^{\rm Sch} = (x,\cD(y))_p \text{ in }\QQ_p\otimes_{\ZZ_p}I/I^2.$$
\end{theorem}

\begin{proof}
We identify $\QQ_p \otimes_{\ZZ_p} I/I^2 =\QQ_p$ via the isomorphism $\ell_p$ in (\ref{lp}). By  Venerucci's computation \cite[Prop. 15.2]{venerucci thesis}, we have
$$\log_\omega(x) \cdot \cD(y)({\rm Fr}_p) = -\frac{\log_p(q_E)}{\ord_p(q_E)} \langle x,y_0\rangle_p^{\rm Sch}.$$
(See also \cite[(127)]{venerucci thesis}.) Here $\cD(y)({\rm Fr}_p)$ means the evaluation of $\cD(y) \in H^1(\QQ_p,\QQ_p)=\Hom_{\rm cont}(G_{\QQ_p}, \QQ_p)$ at the arithmetic Frobenius ${\rm Fr}_p$ (this corresponds to ${\rm Der}_p(\bm{x})$ in \cite[\S 15]{venerucci thesis}, where $\bm{x}$ corresponds to our $y$). Since $\cD(y)({\rm Fr}_p)=-\frac{\log_p(q_E)}{\ord_p(q_E)}\exp^\ast_\omega(\cD(y))$ (see \cite[(6)]{kobayashi} or (\ref{l formula}) below) and $\log_p(q_E)\neq 0$, we have
$$\log_\omega(x)\exp^\ast_\omega(\cD(y))=\langle x,y_0 \rangle_p^{\rm Sch}.$$
Since the left hand side is equal to $(x,\cD(y))_p$, we obtain the desired formula.
\end{proof}

The following is an analogue of Corollary \ref{cor kappa}
\begin{corollary}\label{cor kappa2}
For any $x \in E(\QQ)$, we have
$$\langle x, {}_{c,d}\kappa_\infty \rangle_p^{\rm Sch} = (x, \cD({}_{c,d}\kappa_\infty))_p \text{ in }\QQ_p\otimes_{\ZZ_p}Q^r.$$
\end{corollary}

Since $E$ over $\QQ_p$ is a Tate curve, we have an isomorphism
$E(\QQ_p) \simeq \QQ_{p}^{\times}/\langle q_{E} \rangle$.
We denote by $\lambda_{p}$ the composite map
$$\lambda_{p}: \QQ_{p}^{\times} \to
(\QQ_{p}^{\times}/\langle q_{E} \rangle) \otimes \QQ_{p}
\simeq E(\QQ_p) \otimes \QQ_{p} \to H^1(\QQ_{p},V)$$
where the final map is the Kummer map.
This map $\lambda_{p}$ also coincides with the composite
$\QQ_{p}^{\times} \to H^1(\QQ_p, \QQ_{p}(1))=H^1(\QQ_p, F^+V) \to H^1(\QQ_{p},V)$
where the first map is the Kummer map.
Therefore, for any $a \in \QQ_{p}^{\times}$ and $z \in H^1(\QQ_{p},V)$
we have
$$(\lambda_{p}(a),z)_{p}=(a, \pi_{p}(z))_{\GG_m}$$
where $\pi_p: H^1(\QQ_p, V) \to H^1(\QQ_p, \QQ_p)$ is the
natural map induced by $V \twoheadrightarrow F^{-}V=\QQ_p$, and
$(-,-)_{\GG_m}$ is the pairing induced by the cup product
$H^1(\QQ_{p},\QQ_p(1)) \times H^1(\QQ_{p},\QQ_p) \to
H^2(\QQ_{p},\QQ_p(1))\simeq \QQ_p$.

The following result explains how the $\cL$-invariant occurs in our  generalized version of Rubin's formula.

\begin{lemma} \label{LinvariantLemma}
For any $z \in H^1(\QQ_{p},V)$ we have
$$(\lambda_{p}(p),z)_{p} \cdot (\gamma-1)=(p, \pi_{p}(z))_{\GG_m} \cdot (\gamma-1)
=-\frac{\log_{p}\chi_{\rm cyc}(\gamma)}{\ord_p(q_E)}
\exp^\ast_\omega(z)(1-{\rm rec}_p(q_E))$$
in $\QQ_p \otimes_{\ZZ_p} I/I^2$.
\end{lemma}

\begin{proof} We write $\log_{q_E}: (\QQ_p^\times/ \langle q_E\rangle) \otimes \QQ_p \to \QQ_p$ for the logarithm that vanishes on $\langle q_E \rangle$ and note that this coincides with the formal logarithm via the isomorphism $E(\QQ_p) \simeq \QQ_{p}^{\times}/\langle q_{E} \rangle$. We also write $\exp_{q_E}$ for the inverse of $\log_{q_E}$. 

Then, by using the equality of functions  
$$\log_{q_E}=\log_p - \frac{\log_p(q_E)}{\ord_p(q_E)} \cdot \ord_p$$
(cf. the proof of  \cite[Cor. 3.7]{venerucci}), one computes that 

\begin{eqnarray*}
\lambda_{p}(p)&=&\lambda_{p}(\exp_{q_E} (\log_{q_E} (p)))\\
&=&\lambda_{p}\left(\exp_{q_E} \left(-\frac{\log_p(q_E)}{\ord_p(q_E)}\right)\right)\\
&=&-\exp_{\QQ_p,V}\left(\frac{\log_{p} (q_{E})}{\ord_p(q_E)}\nu \right)
\end{eqnarray*}
in $E(\QQ_p) \otimes \QQ_{p}$. 
Thus we have
\begin{eqnarray}\label{l formula}
(\lambda_{p}(p),z)_{p}=-\frac{\log_{p} (q_{E})}{\ord_p(q_E)}
\exp^\ast_\omega(z).
\end{eqnarray}
(See also \cite[(6)]{kobayashi}.) The claim follows by noting
$$1-{\rm rec}_p(q_E) = \frac{\log_p(q_E)}{\log_p \chi_{\rm cyc}(\gamma)}\cdot (\gamma-1).$$
\end{proof}

Let $U_n^1$ be the group of principal local units in $\QQ_{n,p}$. Let $(d_n)_n \in \varprojlim_n U_n^1$ be the system constructed by Kobayashi in \cite[\S 2]{kobayashi}. This system is related to our $(\delta_n)_n$ defined in (\ref{def delta}) by
$$\delta_n = \log_p(d_n)\cdot \nu \text{ in }\QQ_{n,p}\otimes_{\QQ_p}D_{\rm crys}(V).$$
Since $d_0=1$, Hilbert's theorem 90 implies that there exists $x_n \in \QQ_{n,p}^\times$ such that
$$d_n=\frac{\gamma x_n}{x_n}.$$
We regard $x_n \in H^1(\QQ_{n,p},\ZZ_p(1))$ via the Kummer map.
The element ${\rm Cor}_{\QQ_{n,p}/\QQ_p}(x_n)$ is well-defined in $H^1(\QQ_p,\ZZ/p^n(1))$, i.e., independent of the choice of $x_n$.
We define
$$d':=({\rm Cor}_{\QQ_{n,p}/\QQ_p}(x_n))_n \in \varprojlim_n H^1(\QQ_p,\ZZ/p^n(1)) \simeq H^1(\QQ_p,\ZZ_p(1)) .$$

For each field $\QQ_{n,p}$ with $n \geq 0$ we also write 
\begin{eqnarray}\label{cup Gm}
(-,-)_{\GG_m}: H^1(\QQ_{n,p},\ZZ_p(1)) \times H^1(\QQ_{n,p},\ZZ_p) \to H^2(\QQ_{n,p},\ZZ_p(1))\simeq \ZZ_p
\end{eqnarray}
for the pairing defined by the cup product. Let
$$\pi_p: H^1_n=H^1(\cO_{\QQ_n,S},T) \to H^1(\QQ_{n,p},T) \to
H^1(\QQ_{n,p},\ZZ_p)$$
be the map induced by the surjection $T \twoheadrightarrow  \ZZ_p$ in (\ref{texact}).



We define
$${\rm Col}_n': H^1_n \to \ZZ/p^n[G_n]$$
by
$${\rm Col}_n'(z):=\sum_{\sigma \in G_n}(\sigma x_n, \pi_p(z))_{\GG_m}\sigma$$
and set
$${\rm Col}':=\varprojlim_n {\rm Col}_n': \HH^1 \to \varprojlim_n \ZZ/p^n[G_n]\simeq \Lambda.$$

\begin{lemma}\label{col4}\
\begin{itemize}
\item[(i)] The Coleman map ${\rm Col}: \HH^1 \to \Lambda$ coincides with $(\gamma^{-1}-1)\cdot {\rm Col}'.$
\item[(ii)] Let $y \in \HH^1$. Then we have
$${\rm Col}'(y) \in I$$
and
$${\rm Col}'(y)= (d', \cD(y))_{\GG_m} \text{ in }\QQ_p\otimes_{\ZZ_p}I/I^2,$$
where
$$(-,-)_{\GG_m}: H^1(\QQ_p,\QQ_p(1)) \times (H^1(\QQ_{p}, \QQ_p)\otimes_{\ZZ_p}I/I^2) \to \QQ_p\otimes_{\ZZ_p}I/I^2$$
is induced by (\ref{cup Gm}).
\end{itemize}
\end{lemma}

\begin{proof}
Claim (i) follows directly from construction. (See also Kobayashi's computation of ${\rm Col}_n(z)$ in \cite[p. 573]{kobayashi}.)

Claim (ii) is proved in the same way as Lemma \ref{col3} and so, for brevity, we omit the proof.
\end{proof}


\subsubsection{Proof of Theorem \ref{main}(ii)}
Let $w \in \HH^1$ be the element in Lemma \ref{kappa uni}. We compute
\begin{eqnarray*}
t_{c,d} \cdot \mathcal{L}_{S,p}&=& {\rm Col}(({}_{c,d}z_n)_n) \quad  \text{ (by Theorem \ref{katothm})}\\
&=&{\rm Col}(w)\cdot (\gamma-1)^{r-1} \quad \text{ (by Lemma \ref{kappa uni})}\\
&=&{\rm Col}'(w) \cdot (\gamma^{-1}-1)(\gamma-1)^{r-1} \quad \text{ (by Lemma \ref{col4}(i))} \\
&\in & I^{r+1} \quad \text{ (by Lemma \ref{col4}(ii)).}
\end{eqnarray*}

Thus, in $I^{r+1}/I^{r+2}=Q^{r+1}$, we further compute
\begin{eqnarray*}
&&t_{c,d} \cdot \mathcal{L}_{S,p}^{(r+1)}\\
&=& -{\rm Col}'(w) \cdot (\gamma-1)^{r}  \\
&= &- (d',\cD(w))_{\GG_m} \cdot (\gamma-1)^r \quad \text{ (by Lemma \ref{col4}(ii))}
\end{eqnarray*}
Since
$$(d',\cD(w))_{\GG_m}=\left(1-\frac{1}{p} \right)^{-1}(\log_{p}\chi_{\rm cyc}(\gamma))^{-1}(p, \cD(w))_{\GG_m}$$
(see Kobayashi \cite[p. 574, line 2]{kobayashi}, note that
`${\N} x_n$' in \cite{kobayashi} is congruent to $d'$ modulo $p^n$),
Lemma \ref{LinvariantLemma} implies that
$$- (d',\cD(w))_{\GG_m} \cdot (\gamma-1)^r
=\left(1-\frac 1p \right)^{-1} \exp^\ast_\omega(\cD(w))
\cdot\frac{1}{\ord_p(q_E)}(1-{\rm rec}_p(q_E))\cdot
(\gamma-1)^{r-1}.
$$

Note that, for any $x \in E(\QQ)$ and $y \in H^1(\QQ_p,V)$, we have
$$\log_\omega(x)\exp^\ast_\omega(y) = (x, y)_p.$$
Hence we have
\begin{eqnarray*}
&&\left(1-\frac 1p \right)\log_\omega(x) t_{c,d} \cdot \mathcal{L}_{S,p}^{(r+1)}\\
&=&(x,\cD(w))_p \cdot\frac{1}{\ord_p(q_E)}(1-{\rm rec}_p(q_E)) \cdot (\gamma-1)^{r-1}\\
&=&(x,\cD({}_{c,d}\kappa_\infty))_p \cdot\frac{1}{\ord_p(q_E)}(1-{\rm rec}_p(q_E)) \\
&=&\langle x,{}_{c,d}\kappa_\infty\rangle_p^{\rm Sch} \cdot\frac{1}{\ord_p(q_E)}(1-{\rm rec}_p(q_E))  \quad \text{ (by Corollary \ref{cor kappa2})}.
\end{eqnarray*}
Upon multiplying both sides by $t_{c,d}^{-1}$ we obtain the desired formula.

This thereby completes the proof of Theorem \ref{main}.

\section{The Iwasawa Main Conjecture and descent theory}

The aim of this section is to directly relate Conjectures \ref{mrs2} and \ref{mrs3} with a natural main conjecture of Iwasawa theory. The main results in this section are Theorems \ref{th1} and \ref{th2}.

As before, we always assume that $p$ is odd and that $H^1(\ZZ_S,T)$ is $\ZZ_p$-free.

\subsection{Review of the Iwasawa Main Conjecture}

We use the notations in \S \ref{sec id}.

We set
$$C_n:=\rhom_{\ZZ_p}(\rgamma_c(\cO_{\QQ_n,S},T^\ast(1)),\ZZ_p[-2])$$
and $C_\infty:=\varprojlim_n C_n$. Then we have a canonical isomorphism
$$H^0(C_\infty)\simeq \HH^1$$
and an exact sequence
\begin{eqnarray}\label{h2 iw}
0\to \HH^2 \to H^1(C_\infty) \xrightarrow{f} \Lambda \otimes_{\ZZ_p} T^\ast(1)^{+,\ast} \to 0.
\end{eqnarray}
(See (\ref{exh1}) and (\ref{exh2}).)
Let $Q(\Lambda)$ denote the quotient field of $\Lambda$.
Kato proved that
\begin{eqnarray*}\label{kato lambda}
Q(\Lambda) \otimes_\Lambda \HH^i  \begin{cases}
\simeq Q(\Lambda) &\text{if $i=1$},\\
=0 &\text{if $i=2$}.
\end{cases}
\end{eqnarray*}
(See \cite[Th. 12.4]{katoasterisque}.)
Hence, we have a canonical isomorphism
\begin{eqnarray}\label{det can}
Q(\Lambda) \otimes_\Lambda {\det}_\Lambda(C_\infty) \simeq Q(\Lambda) \otimes_\Lambda (\HH^1 \otimes_{\ZZ_p} T^\ast(1)^+).
\end{eqnarray}
We set

$${}_{c,d}z_\infty:=({}_{c,d}z_n)_n \in \varprojlim_n H^1_n=\HH^1$$
and
$$z_\infty:=t_{c,d}^{-1} \cdot {}_{c,d}z_\infty \in Q(\Lambda) \otimes_{\Lambda} \HH^1,$$
where $t_{c,d} \in \Lambda$ is as in (\ref{tcd}). We then define
$$\mathfrak{z}_\infty \in Q(\Lambda) \otimes_\Lambda {\det}_\Lambda(C_\infty) $$
to be the element corresponding to
$$z_\infty \otimes e^+\delta(\xi) \in Q(\Lambda)\otimes_\Lambda (\HH^1 \otimes_{\ZZ_p} T^\ast(1)^+)$$
under the isomorphism (\ref{det can}), where $\delta(\xi) \in \ZZ_p \otimes_\ZZ \sH \simeq T^\ast(1)$ is defined in \S \ref{sec kato}.

\begin{conjecture}[Iwasawa Main Conjecture]\label{IMC}
We have
$$\langle \mathfrak{z}_\infty \rangle_\Lambda ={\det}_\Lambda(C_\infty).$$
\end{conjecture}

\begin{remark}\label{rem char}
Since $\Lambda$ is a regular local ring, we see by \cite[Chap. I, Prop. 2.1.5]{katolecture} that Conjecture \ref{IMC} is equivalent to the equality
$${\rm char}_\Lambda(\HH^1/\langle z_\infty \rangle_\Lambda)= {\rm char}_\Lambda(\HH^2).$$
Thus Conjecture \ref{IMC} is equivalent to \cite[Conj. 12.10]{katoasterisque} (by letting $f$ in loc. cit. be the normalized newform corresponding to $E$).
\end{remark}

\subsection{Consequences of the Iwasawa Main Conjecture}

We now state main results of this section.


\begin{theorem}\label{th1}
Assume Hypothesis \ref{hyp}. Then
Conjecture \ref{IMC} (Iwasawa Main Conjecture) implies Conjecture \ref{mrs3} up to $\ZZ_p^\times$, i.e., there exists $u \in \ZZ_p^\times$ such that
$$\kappa_\infty =u\cdot v_\xi\left(\prod_{\ell \in S \setminus \{\infty\}} L_\ell \right)\frac{\# \sha(E/\QQ)\cdot {\rm Tam}(E)}{\# E(\QQ)_{\rm tors}^2} \cdot R_\omega^{\rm Boc} \text{ in }(\QQ_p \otimes_\ZZ E(\QQ))\otimes_{\ZZ_p} Q^{r-1}.$$
\end{theorem}

Combining this theorem with Corollary \ref{cor padic}, we immediately obtain the following.
\begin{corollary}\label{cor th1}
Assume Hypothesis \ref{hyp}. Then
Conjecture \ref{IMC} (Iwasawa Main Conjecture) implies the $p$-adic Birch-Swinnerton-Dyer Formula up to $\ZZ_p^\times$, i.e., there exists $u \in \ZZ_p^\times$ such that
$$ \left( 1-\frac 1\alpha \right)^{-1}\left(1-\frac 1\beta\right)\cdot  \cL_{S,p}^{(r)} =u\cdot v_\xi\left(\prod_{\ell \in S \setminus \{\infty\}} L_\ell \right)\frac{\# \sha(E/\QQ)\cdot {\rm Tam}(E)}{\# E(\QQ)_{\rm tors}^2} R_p $$
if $p$ is good or non-split multiplicative, and
$$  \cL_{S,p}^{(r+1)} =u\cdot \frac{1}{{\rm ord}_p(q_E)} (1-{\rm rec}_p(q_E))\cdot  v_\xi\left(\prod_{\ell \in S \setminus \{\infty,p\}} L_\ell \right)\frac{\# \sha(E/\QQ)\cdot {\rm Tam}(E)}{\# E(\QQ)_{\rm tors}^2} R_p^{\rm Sch} $$
if $p$ is split multiplicative.
\end{corollary}

\begin{remark}\label{rem1}
Corollary \ref{cor th1} improves upon results of Schneider \cite[Th. 5]{sch2} and Perrin-Riou \cite[Prop. 3.4.6]{PR} in which it is shown that the Iwasawa Main Conjecture and non-degeneracy of the $p$-adic height pairing together imply the $p$-adic Birch-Swinnerton-Dyer Formula up to $\ZZ_p^\times$ under the restrictive hypothesis that the reduction of $E$ at $p$ is good ordinary and good respectively. 
\end{remark}

\begin{theorem}\label{th2}
Assume Hypothesis \ref{hyp}. Assume also that
\begin{itemize}
\item Conjecture \ref{IMC} (Iwasawa Main Conjecture) is valid,
\item Conjecture \ref{mrs2} (Generalized Perrin-Riou Conjecture at infinite level) is valid, and
\item the Bockstein regulator $R_\omega^{\rm Boc}$ in Definition \ref{def alg} does not vanish.
\end{itemize}
Then the $p$-part of the Birch-Swinnerton-Dyer Formula is valid so that there is an equality
 $$L^\ast(E,1) \cdot \ZZ_p = \left(\# \sha(E/\QQ)[p^\infty]\cdot {\rm Tam}(E) \cdot \# E(\QQ)_{\rm tors}^{-2}\cdot \Omega^+ \cdot R_\infty \right)
 \cdot \ZZ_p$$
of $\ZZ_p$-sublattices of $\CC_p$ .
\end{theorem}


\begin{remark}\label{rem3} Theorem \ref{th2} explains the precise link between the natural main conjecture of Iwasawa theory and the classical Birch-Swinnerton-Dyer Formula, even in the case of additive reduction. We note also that this result is, in effect, an analogue of the main result \cite[Th. 5.2]{bks2} of the current authors, where, roughly speaking, the following result is proved in the setting of the multiplicative group: if one assumes the validity of 
\begin{itemize}
\item the Iwasawa Main Conjecture for $\GG_m$ (cf. \cite[Conj. 3.1]{bks2}),
\item the Iwasawa-Mazur-Rubin-Sano Conjecture for $\GG_m$ (cf. \cite[Conj. 4.2]{bks2}), and
\item the injectivity of a certain Bockstein homomorphism (which is implied by the condition `(F)' in \cite[Th. 5.2]{bks2}: see \cite[Prop. 5.16]{bks2}),
\end{itemize}
then the equivariant Tamagawa Number Conjecture for $\GG_m$ is also valid.
%
\end{remark}

\subsection{The descent argument}

In the following, we assume both Hypothesis \ref{hyp} and the validity of Conjecture \ref{IMC}.

\subsubsection{A key commutative diagram}

We shall first give quick proofs of Theorems \ref{th1} and \ref{th2} by using the following key result.

\begin{theorem}\label{key}
Let $\bm{x}$ be a $\ZZ_p$-basis of ${\bigwedge}_{\ZZ_p}^{r-1} H^2(\ZZ_S,T)_{\rm tf}$. Then there is a commutative diagram
\begin{eqnarray}\label{key diag}
\xymatrix{
{\det}_{\Lambda}(C_\infty) \ar[r]^{\Pi_\infty} \ar@{->>}[dd]_{\N_\infty}& I^{r-1}\cdot \HH^1 \ar[rd]^{\cN_\infty} & \\
& & H^1_0 \otimes_{\ZZ_p} Q^{r-1}\\
{\det}_{\ZZ_p}(C_0) \ar[r]_{\Pi_{\bm{x}}} &{\bigwedge}_{\ZZ_p}^r H^1_0 \ar[ru]_{{\rm Boc}_{\infty,\bm{x}}}&
}
\end{eqnarray}
with the following properties:
\begin{itemize}
\item[(a)] $\Pi_\infty(\fz_\infty)=z_\infty$;
\item[(b)] $\cN_\infty(z_\infty)= \kappa_\infty$;
\item[(c)] $\langle \eta_{\bm{x}}^{\rm Kato} \rangle_{\ZZ_p}= \# H^2(\ZZ_S,T)_{\rm tors}\cdot {\bigwedge}_{\ZZ_p}^r H^1_0$, where $\eta_{\bm{x}}^{\rm Kato}:=\Pi_{\bm{x}}(\N_\infty(\fz_\infty))$;
\item[(d)] $\langle {\rm Boc}_{\infty,\bm{x}}(\eta_{\bm{x}}^{\rm Kato}) \rangle_{\ZZ_p} =\ZZ_p \cdot v_\xi \left( \prod_{\ell \in S\setminus \{\infty\}} L_\ell\right)\#\sha(E/\QQ)[p^\infty] {\rm Tam}(E) \# E(\QQ)_{\rm tors}^{-2}\cdot R_\omega^{\rm Boc}$.
\end{itemize}
\end{theorem}

Admitting this, we give proofs of Theorems \ref{th1} and \ref{th2}.

\begin{proof}[Proof of Theorem \ref{th1}]
It is sufficient to show that
$$\langle \kappa_\infty \rangle_{\ZZ_p}=\ZZ_p\cdot v_\xi \left( \prod_{\ell \in S\setminus \{\infty\}} L_\ell\right)\#\sha(E/\QQ)[p^\infty] {\rm Tam}(E)\# E(\QQ)_{\rm tors}^{-2}\cdot R_\omega^{\rm Boc}.$$
By the commutativity of (\ref{key diag}) and properties (a) and (b), we have
\begin{eqnarray}\label{alg mrs}
\kappa_\infty={\rm Boc}_{\infty,\bm{x}}(\eta_{\bm{x}}^{\rm Kato}).
\end{eqnarray}
Hence the claim follows from the property (d).
\end{proof}

\begin{proof}[Proof of Theorem \ref{th2}]
We assume Conjecture \ref{mrs2} and $R_\omega^{\rm Boc}\neq 0$, in addition to Hypothesis \ref{hyp} and Conjecture \ref{IMC}. Recall that Conjecture \ref{mrs2} asserts the equality
$$\kappa_\infty ={\rm Boc}_{\infty,\bm{x}}(\eta_{\bm{x}}^{\rm BSD}).$$
Combining this with (\ref{alg mrs}), we have
$${\rm Boc}_{\infty,\bm{x}}(\eta_{\bm{x}}^{\rm BSD}) ={\rm Boc}_{\infty,\bm{x}}(\eta_{\bm{x}}^{\rm Kato}).$$
Since the non-vanishing of $R_\omega^{\rm Boc}$ is equivalent to the injectivity of ${\rm Boc}_{\infty,\bm{x}}$ by construction, we have
$$\eta_{\bm{x}}^{\rm BSD}=\eta_{\bm{x}}^{\rm Kato}.$$
By the property (c) in Theorem \ref{key}, we have
$$\ZZ_p\cdot \eta_{\bm{x}}^{\rm BSD}= \# H^2(\ZZ_S,T)_{\rm tors}\cdot {\bigwedge}_{\ZZ_p}^r H^1_0.$$
By Proposition \ref{prop eta}, this is equivalent to the $p$-part of the Birch-Swinnerton-Dyer Formula, so we have completed the proof.
\end{proof}

The rest of this section is devoted to the proof of Theorem \ref{key}.

\subsubsection{Definitions of maps}\label{def map}

First, we give definitions of the maps $\Pi_\infty, \cN_\infty, \N_\infty$ and $\Pi_{\bm{x}}$ in the diagram (\ref{key diag}).

\begin{itemize}
\item The map
$$\Pi_\infty: {\det}_{\Lambda}(C_\infty) \to I^{r-1}\cdot \HH^1$$
is induced by
$$Q(\Lambda) \otimes_{\Lambda} {\det}_{\Lambda}(C_\infty) \stackrel{(\ref{det can})}{\simeq} Q(\Lambda) \otimes_{\Lambda} (\HH^1 \otimes_{\ZZ_p} T^\ast(1)^+) \simeq Q(\Lambda)\otimes_{\Lambda} \HH^1,$$
where the second isomorphism is induced by
$$T^\ast(1)^+ \simeq \ZZ_p; \ e^+\delta(\xi) \mapsto 1.$$
By Remark \ref{rem char}, the image of ${\det}_\Lambda(C_\infty)$ under this isomorphism is ${\rm char}_\Lambda(\HH^2)\cdot \HH^1$. Since ${\rm char}_\Lambda(\HH^2) \subset I^{r-1}$, we see that the image of ${\det}_\Lambda(C_\infty)$ is contained in $I^{r-1}\cdot \HH^1$ and thus $\Pi_\infty$ is defined.
By this construction, it is obvious that $\Pi_\infty(\fz_\infty)=z_\infty$, i.e., the property (a) of Theorem \ref{key} holds.

\item The construction of the map
$$\cN_\infty: I^{r-1}\cdot \HH^1 \to H^1_0 \otimes_{\ZZ_p}Q^{r-1}$$
is done in the same way as the construction of ${}_{c,d}\kappa_\infty$ from $({}_{c,d}z_n)_n$ in \S \ref{sec iw}. See the discussion after Proposition \ref{kato vanish}. (In fact, $\cN_\infty$ is defined to be the limit of the Darmon norm $\cN_{\QQ_n/\QQ}$.)
It is obvious that $\cN_\infty(z_\infty)=\kappa_\infty$, i.e., the property (b) in Theorem \ref{key} holds.

\item The surjection
$${\N}_\infty: {\det}_\Lambda (C_\infty) \twoheadrightarrow {\det}_{\ZZ_p}(C_0)$$
is defined to be the augmentation map
$${\det}_\Lambda(C_\infty) \twoheadrightarrow {\det}_\Lambda(C_\infty) \otimes_{\Lambda} \ZZ_p \simeq {\det}_{\ZZ_p}(C_0),$$
where the last isomorphism follows from the fact $C_\infty \lotimes_{\Lambda}\ZZ_p \simeq C_0$.

\item The map
$$\Pi_{\bm{x}}: {\det}_{\ZZ_p}(C_0) \to {\bigwedge}_{\ZZ_p}^r H^1_0$$
is induced by
\begin{eqnarray*}
\QQ_p \otimes_{\ZZ_p} {\det}_{\ZZ_p}(C_0) &\simeq& \QQ_p \otimes_{\ZZ_p} \left({\det}_{\ZZ_p}(H^0(C_0)) \otimes_{\ZZ_p} {\det}_{\ZZ_p}^{-1}(H^1(C_0))\right) \\
&\simeq& \QQ_p \otimes_{\ZZ_p} \left({\bigwedge}^r_{\ZZ_p}H^1_0 \otimes_{\ZZ_p} {\bigwedge}_{\ZZ_p}^{r-1}H^2(\ZZ_S,T)_{\rm tf}^\ast \otimes_{\ZZ_p}T^\ast(1)^+ \right) \\
&\simeq& \QQ_p \otimes_{\ZZ_p} {\bigwedge}_{\ZZ_p}^r H^1_0,
\end{eqnarray*}
where the second isomorphism follows from (\ref{exh1}) and (\ref{exh2}), and the last isomorphism is induced by
$$ {\bigwedge}_{\ZZ_p}^{r-1}H^2(\ZZ_S,T)_{\rm tf}^\ast \otimes_{\ZZ_p}T^\ast(1)^+ \simeq \ZZ_p; \ \bm{x}^\ast \otimes e^+\delta(\xi)\mapsto 1.$$
Since the image of ${\det}_{\ZZ_p}(C_0)$ under this isomorphism is $\# H^2(\ZZ_S,T)_{\rm tors}\cdot {\bigwedge}_{\ZZ_p}^rH^1_0$, the map $\Pi_{\bm{x}}$ is defined. This also shows that the property (c) in Theorem \ref{key} holds.

\end{itemize}
\subsubsection{The property (d)}

We have already seen that the properties (a), (b) and (c) in Theorem \ref{key} are satisfied.

We shall now verify property (d), i.e., that there is an equality of $\ZZ_p$-lattices
$$\ZZ_p\cdot \bigl({\rm Boc}_{\infty,\bm{x}}(\eta_{\bm{x}}^{\rm Kato})\bigr) =\ZZ_p\cdot c_E\cdot R_\omega^{\rm Boc},$$
where
$$c_E:=v_\xi\cdot \left( \prod_{\ell \in S\setminus \{\infty\}} L_\ell\right)\cdot\#\sha(E/\QQ)[p^\infty]\cdot {\rm Tam}(E)\cdot \# E(\QQ)_{\rm tors}^{-2} .$$

One checks that the element ${\rm Boc}_{\infty,\bm{x}}(\eta_{\bm{x}}^{\rm Kato})$ is independent of the choice of $\bm{x}$. So we take $\bm{x}$ to be as in \S \ref{exp int}, by fixing a basis $\{x_1,\ldots,x_r\}$ of $E(\QQ)_{\rm tf}$.
Note that this element $\bm{x}$ belongs to ${\bigwedge}_{\QQ_p}^{r-1} H^2(\ZZ_S,V)$ and may not be a $\ZZ_p$-basis of ${\bigwedge}_{\ZZ_p}^{r-1} H^2(\ZZ_S, T)_{\rm tf}$.
However, both ${\rm Boc}_{\infty,\bm{x}}$ and $\eta_{\bm{x}}^{\rm Kato}$ are defined for this $\bm{x}$ by linearity.

By the definition of $R_\omega^{\rm Boc}$ (see Definition \ref{def alg}), it is sufficient to show that
\begin{eqnarray}\label{bsd kato}
\langle \eta_{\bm{x}}^{\rm Kato}\rangle_{\ZZ_p}=\ZZ_p\cdot c_E\cdot \log_\omega(x_1)\cdot x_1\wedge \cdots \wedge x_r.
\end{eqnarray}

By the property (c) and (\ref{alg lattice}), we have
$$\langle \eta_{\bm{x}}^{\rm Kato} \rangle_{\ZZ_p} = \langle \eta_{\bm{x}}^{\rm alg} \rangle_{\ZZ_p}.$$
(Here $\eta_{\bm{x}}^{\rm alg}$ is defined in Definition \ref{alg bsd def}, where the finiteness of $\sha(E/\QQ)$ is assumed. 
But we may define $\eta_{\bm{x}}^{\rm alg}$, replacing $\sha(E/\QQ)$ by $\sha(E/\QQ)[p^\infty]$ 
since we only consider the $\ZZ_p$-modules here. Then we need only the finiteness of $\sha(E/\QQ)[p^\infty]$.)
On the other hand, by (\ref{element formula}), we have
$$\langle \eta_{\bm{x}}^{\rm alg} \rangle_{\ZZ_p} =\ZZ_p \cdot c_E \cdot \log_\omega(x_1)\cdot x_1\wedge \cdots \wedge x_r.$$
From this, we obtain the desired equality (\ref{bsd kato}). Hence we have proved that the property (d) holds.

\subsection{The proof of Theorem \ref{key}}

In this subsection, we prove the commutativity of the diagram (\ref{key diag}) and thus complete the proof of Theorem \ref{key}.
Our argument is similar to \cite[Lem. 5.22]{bks1}, \cite[Lem. 5.17]{bks2} and \cite[Th. 4.21]{sbA}.

Fix a non-negative integer $n$. It is sufficient to show the commutativity of the following `$n$-th layer version' of (\ref{key diag}):
\begin{eqnarray}\label{diagn}
\xymatrix{
{\det}_{\ZZ_p[G_n]}(C_n) \ar[r]^{\Pi_n} \ar@{->>}[dd]_{\N_n}& I_n^{r-1}\cdot H_n^1 \ar[rd]^{\cN_n} & \\
& & H^1_0 \otimes_{\ZZ_p} Q_n^{r-1}\\
{\det}_{\ZZ_p}(C_0) \ar[r]_{\Pi_{\bm{x}}} &{\bigwedge}_{\ZZ_p}^r H^1_0. \ar[ru]_{{\rm Boc}_{n,\bm{x}}}&
}
\end{eqnarray}

We shall describe maps $\Pi_\infty$, $\Pi_n$, $\Pi_{\bm{x}}$ and ${\rm Boc}_{n,\bm{x}}$ explicitly by choosing a useful representative of the complex $C_\infty$. Then the commutativity of the diagram is checked by an explicit computation.

\subsubsection{Choice of the representative}

We make a similar argument to \cite[\S 5.4]{bks1} or \cite[Prop. A.11]{sbA}.

One sees that the complex $C_\infty$ is represented by
$$\PP \xrightarrow{\psi} \PP,$$
where $\PP$ is a free $\Lambda$-module of rank, say, $d$. We have an exact sequence
\begin{eqnarray}\label{iwasawa tate}
0 \to \HH^1 \to \PP \xrightarrow{\psi} \PP \xrightarrow{\pi} H^1(C_\infty) \to 0.
\end{eqnarray}
Also, setting $P_n:=\PP \otimes_\Lambda \ZZ_p[G_n]$, we have an exact sequence
\begin{eqnarray}\label{n tate}
0 \to H_n^1 \to P_n \xrightarrow{\psi_n} P_n \xrightarrow{\pi_n} H^1(C_n) \to 0.
\end{eqnarray}
Let $\{b_1,\ldots,b_d\}$ be a basis of $\PP$. We denote the image of $b_i$ in $P_n$ by $b_{i,n}$. We set
$$x_i:=\pi(b_i) \in H^1(C_\infty) \text{ and }x_{i,n}:=\pi_n(b_{i,n}) \in H^1(C_n).$$
By the argument of \cite[Prop. A.11(i)]{sbA}, one may assume
\begin{itemize}
\item[(i)] $f(x_1)=1 \otimes e^+\delta(\xi)^\ast$, where $f:H^1(C_\infty) \to \Lambda \otimes_{\ZZ_p}T^\ast(1)^{+,\ast}$ is as in (\ref{h2 iw});
\item[(ii)] $\langle x_2,\ldots,x_d \rangle_\Lambda =\HH^2 \subset H^1(C_\infty)$;
\item[(iii)] $\{x_{2,0},\ldots,x_{r,0}\}$ is a $\ZZ_p$-basis of $H^2(\ZZ_S,T)_{\rm tf} \subset H^1(C_0)$.
\end{itemize}

We set
$$\psi_i:=b_i^\ast \circ \psi: \PP \to \Lambda$$
and
$$\psi_{i,n}:=b_{i,n}^\ast \circ \psi_n: P_n \to \ZZ_p[G_n].$$
Note that the property (iii) implies that
\begin{eqnarray}\label{image aug}
\im \psi_{i,n} \subset I_n \text{ for every $1 < i \leq r$}.
\end{eqnarray}

\subsubsection{Explicit descriptions of $\Pi_\infty$, $\Pi_n$ and $\Pi_{\bm{x}}$}
With the above representative of $C_\infty$, we have an identification
$${\det}_\Lambda(C_\infty)={\bigwedge}_\Lambda^d \PP \otimes_\Lambda {\bigwedge}_{\Lambda}^d \PP^\ast.$$
We define a map
$$\Pi_\infty:{\bigwedge}_\Lambda^d \PP \otimes_\Lambda {\bigwedge}_{\Lambda}^d \PP^\ast\to \PP$$
by
\begin{eqnarray}\label{explicit}
a \otimes (b_1^\ast \wedge \cdots \wedge b_d^\ast) \mapsto (-1)^{d-1} \left( {\bigwedge}_{1<i \leq d}\psi_i\right)(a) .
\end{eqnarray}
We denote this map by $\Pi_\infty$, since it coincides with $\Pi_\infty$ defined in \S \ref{def map} (see \cite[Lem. 4.3]{bks1}). In particular, its image is contained in $I^{r-1}\cdot \HH^1$. (We regard $\HH^1\subset \PP$ via (\ref{iwasawa tate}).)

Similarly, we have an identification
$${\det}_{\ZZ_p[G_n]}(C_n)={\bigwedge}_{\ZZ_p[G_n]}^d P_n \otimes_{\ZZ_p[G_n]} {\bigwedge}_{\ZZ_p[G_n]}^d P_n^\ast$$
and we define a map
$$\Pi_n : {\bigwedge}_{\ZZ_p[G_n]}^d P_n \otimes_{\ZZ_p[G_n]}  {\bigwedge}_{\ZZ_p[G_n]}^d P_n^\ast \to P_n$$
by
\begin{eqnarray}\label{explicitn}
a \otimes (b_{1,n}^\ast \wedge \cdots \wedge b_{d,n}^\ast) \mapsto (-1)^{d-1} \left( {\bigwedge}_{1<i \leq d}\psi_{i,n}\right)(a) .
\end{eqnarray}
It is clear by construction that the inverse limit $\varprojlim_n \Pi_n$ coincides with $\Pi_\infty$. Since the image of $\Pi_\infty$ is contained in $I^{r-1} \cdot \HH^1$, we see that the image of $\Pi_n$ is contained in $I^{r-1}_n \cdot H^1_n$. 

Finally, we give an explicit description of $\Pi_{\bm{x}}$. Here we take
$$\bm{x}:= x_{2,0} \wedge\cdots \wedge x_{r,0}.$$
We have an identification
$${\det}_{\ZZ_p}(C_0) = {\bigwedge}_{\ZZ_p}^d P_0 \otimes_{\ZZ_p} {\bigwedge}_{\ZZ_p}^d P_0^\ast.$$
We define a map
$$\Pi_{\bm{x}}: {\bigwedge}_{\ZZ_p}^d P_0 \otimes_{\ZZ_p} {\bigwedge}_{\ZZ_p}^d P_0^\ast \to {\bigwedge}_{\ZZ_p}^r P_0$$
by
\begin{eqnarray}\label{explicit0}
a \otimes (b_{1,0}^\ast \wedge \cdots \wedge b_{d,0}^\ast) \mapsto (-1)^{r(d-r)} \left( {\bigwedge}_{r<i \leq d}\psi_{i,0}\right)(a) .
\end{eqnarray}
This map coincides with $\Pi_{\bm{x}}$ defined in \S \ref{def map} (by \cite[Lem. 4.3]{bks1}). In particular, its image is contained in ${\bigwedge}_{\ZZ_p}^r H^1_0$.

\subsubsection{Explicit Bockstein maps}

We shall describe the Bockstein regulator map ${\rm Boc}_{n,\bm{x}}$ explicitly.

For an integer $i $ with $1 < i \leq r$, we define a map
$$\beta_{i,n}: P_0 \to I_n/I_n^2$$
by
$$\beta_{i,n}(a):=\psi_{i,n}(\widetilde a) \text{ (mod $I_n^2$)},$$
where for $a \in P_0$ we take an element $\widetilde a \in P_n$ such that $\sum_{\sigma \in G_n}\sigma ( \widetilde a)=a$ (we regard $P_0 \subset P_n$ by identifying $P_0 $ with $P_n^{G_n}$). Note that $\psi_{i,n}(\widetilde a) \in I_n$ by (\ref{image aug}) and its image in $I_n/I_n^2$ is independent of the choice of $\widetilde a$. Hence the map $\beta_{i,n}$ is well-defined.

Let $\beta_{\QQ_n}: H^1_0 \to H^2(\ZZ_S,T)_{\rm tf} \otimes_{\ZZ_p} I_n/I_n^2$ be the Bockstein map defined in (\ref{def beta}).
One checks by the definition of the connecting homomorphism that
$$-\beta_{i,n}=x_{i,0}^\ast \circ \beta_{\QQ_n} \text{ on }H^1_0.$$
From this, we see that the map
\begin{eqnarray}\label{exp boc}
{\rm Boc}_{n,{\bm{x}}}:=(-1)^{r-1}{\bigwedge}_{1<i \leq r} \beta_{i,n}: {\bigwedge}_{\ZZ_p}^r P_0 \to P_0 \otimes_{\ZZ_p} Q_n^{r-1}
\end{eqnarray}
coincides with ${\rm Boc}_{n,\bm{x}}={\rm Boc}_{\QQ_n,\bm{x}}$ defined in $\S \ref{sec boc}$ on ${\bigwedge}_{\ZZ_p}^r H^1_0$.

\subsubsection{Completion of the proof}

We prove the commutativity of (\ref{diagn}). We may assume $\bm{x}= x_{2,0} \wedge\cdots \wedge x_{r,0}$.

In view of the explicit descriptions (\ref{explicitn}), (\ref{explicit0}) and (\ref{exp boc}), it is sufficient to prove that
\begin{multline}\label{final}
(-1)^{d-1} \cN_n \circ \left( {\bigwedge}_{1<i \leq d} \psi_{i,n} \right)(b_{1,n}\wedge \cdots \wedge b_{d,n})\\
=(-1)^{r-1+r(d-r)}\left( {\bigwedge}_{1<i \leq r} \beta_{i,n}\right) \circ \left({\bigwedge}_{r<i\leq d}\psi_{i,0} \right)(b_{1,0}\wedge \cdots \wedge b_{d,0}).
\end{multline}
By computation, we have
$$\left( {\bigwedge}_{1<i \leq d} \psi_{i,n} \right)(b_{1,n}\wedge \cdots \wedge b_{d,n}) =\sum_{k=1}^d (-1)^{k+1} \det(\psi_{i,n}(b_{j,n}))_{j\neq k} \cdot b_{k,n}$$
(see \cite[Prop. 4.1]{bks1}) and so
\begin{multline*}
 \cN_n \circ \left( {\bigwedge}_{1<i \leq d} \psi_{i,n} \right)(b_{1,n}\wedge \cdots \wedge b_{d,n})\\
=\sum_{k=1}^d (-1)^{k+1} b_{k,0} \otimes \det(\psi_{i,n}(b_{j,n}))_{j\neq k} \text{ in }P_0 \otimes_{\ZZ_p} Q_n^{r-1}.
\end{multline*}
By noting
$$\psi_{i,n}(b_{j,n}) \equiv \psi_{i,0}(b_{j,0}) \text{ (mod $I_n$) for every $r<i\leq d$}, $$
we compute
\begin{multline*}
\left( {\bigwedge}_{1<i \leq r} \beta_{i,n}\right) \circ \left({\bigwedge}_{r<i\leq d}\psi_{i,0} \right)(b_{1,0}\wedge \cdots \wedge b_{d,0})\\
=(-1)^{(r-1)(d-r)}\sum_{k=1}^d (-1)^{k+1} b_{k,0} \otimes \det(\psi_{i,n}(b_{j,n}))_{j\neq k} \text{ in }P_0 \otimes_{\ZZ_p} Q_n^{r-1}.
\end{multline*}
Since we have
$$(-1)^{r-1+r(d-r)+(r-1)(d-r)}=(-1)^{d-1},$$
we therefore obtain the desired equality (\ref{final}).

\begin{acknowledgments}
The third author would like to thank Kazim B\"uy\"ukboduk for helpful discussions, especially about Rubin's formula. The authors would like to thank Takenori Kataoka 
for discussions with him and for his comments on the first draft of this paper, which were very helpful. The authors also would like to thank Christian Wuthrich for carefully reading the manuscript and giving them helpful comments.
\end{acknowledgments}

\end{document}